\documentclass[review]{elsarticle}

\usepackage{lineno,hyperref}
\modulolinenumbers[5]

\usepackage[utf8]{inputenc}
\usepackage{amsmath}
\usepackage{amsfonts}
\usepackage{amssymb}
\usepackage{amsthm}
\usepackage{xcolor}
\usepackage{natbib}
\usepackage{graphicx}
\usepackage{hyperref}
\usepackage{pifont}
\usepackage{booktabs}
\usepackage{tikz}

\journal{Journal of \LaTeX\ Templates}

\newtheorem{thm}{Theorem}[section]
\newtheorem{cor}[thm]{Corollary}
\newtheorem{lem}[thm]{Lemma}
\newtheorem{prop}[thm]{Proposition}
\newtheorem{defn}{Definition}[section]
\newtheorem{rem}{Remark}[section]

\newtheorem{standass}{Standing Assumption}[section]
\newtheorem{exmp}{Example}[section]

\numberwithin{equation}{section}

\newcommand{\RR}{{\mathbb{R}}}
\newcommand{\NN}{{\mathbb{N}}}
\newcommand{\EE}{{\mathbb{E}}}
\newcommand{\FF}{{\mathbb{F}}}
\newcommand{\PP}{{\mathbb{P}}}
\newcommand{\BB}{{\mathbb{B}}}
\newcommand{\mc}{\mathcal}
\newcommand{\norm}[1]{\|#1\|}
\newcommand{\normsq}[1]{\left\|#1\right\|^2}
\newcommand{\EEk}[1]{\EE\left[#1|\mc F_k\right]}
\newcommand{\EEx}[1]{\EE\left[#1\right]}

\newcommand{\op}{\operatorname}
\newcommand{\bs}{\boldsymbol}

\newcommand{\cmark}{\ding{51}}%
\newcommand{\xmark}{\ding{55}}%

\begin{document}

\begin{frontmatter}

\title{Convergence of sequences: A survey\tnoteref{mytitlenote}}
\tnotetext[mytitlenote]{This work was partially supported by NWO under research projects OMEGA (613.001.702) and P2P-TALES (647.003.003), and by the ERC under research project COSMOS (802348).}

%
%
%

\author[UM]{Barbara Franci}\ead{b.franci@maastrichtuniversity.nl}     
\author[Delft]{Sergio Grammatico}\ead{s.grammatico@tudelft.nl}               

\address[UM]{Department of Data Science and Knowledge Engineering, Maastricht University, Maastricht, The Netherlands}  
\address[Delft]{Delft Center for Systems and Control, Delft University of Technology, Delft, The Netherlands}

\begin{abstract}
Convergent sequences of real numbers play a fundamental role in many different problems in system theory, e.g., in Lyapunov stability analysis, as well as in optimization theory and computational game theory. In this survey, we provide an overview of the literature on convergence theorems and their connection with F\'ejer monotonicity in the deterministic and stochastic settings, and we show how to exploit these results.
\end{abstract}


\end{frontmatter}


\section{Introduction}

\begin{flushright}
\textit{Why Are Convergence Theorems Necessary?\\
The answer to this ``naive'' question is not simple.}\\
cit. Boris T. Polyak, 1987 \cite[Section 1.6.2]{polyak1987}. 
\end{flushright}

While the answer may have become clearer through the years, since many problems in applied mathematics rely on convergence theorems, it is still not simple. Besides the theoretical investigation, in fact, one fundamental aspect is how convergence theorems can be of practical use, i.e., if the assumptions are plausible for a variety of applications, for instance, in systems theory. Moreover, convergence theorems may also give qualitative information, e.g., if convergence is guaranteed for any initial point and in what sense (strongly, weakly, almost surely, in probability), which affects the range of application. The aim of this paper is to collect these results towards a complete overview, thus to be able to find the one that most suits the application at hand. In fact, many convergence results find their use in theoretical applications, such as Lyapunov stability analysis \cite{polyak1987,khalil2002,benaim1996,benaim1999}, variational analysis \cite{malitsky2015, malitsky2019, iusem2017, iusem2019, yousefian2017, yousefian2014} and game equilibrium seeking \cite{franci2020fb, koshal2013, facchinei2007, franci2019fbhf}, in automatic control, such as model predictive control \cite{lee2015} and network control problems \cite{shi2013}, as well as in other engineering areas, e.g., training and learning in generative adversarial networks \cite{franci2020gan,franci2020gen, bot2020gan}, vehicle flow control in traffic networks \cite{duvocelle2019} and in modeling the prosumer behavior in smart power grids \cite{franci2020fb, franci2019fbhf, yi2019, kannan2013}.

\subsection{Lyapunov decrease and F\'ejer monotonicity}
In the mathematical literature, many convergence results hold for sequences of numbers while in system and control theory, the state and decision variables are usually \textit{vectors} of real numbers. It is therefore important to understand the deep connection between the two theories. The bridging idea is to associate a real number to the state vector, i.e., via a function, and then prove convergence exploiting the properties of such a function. The most common example of this approach is that of Lyapunov theory where a suitable Lyapunov function is shown to be decreasing along the evolution of the state variable, thus obtaining convergence of the state vector to a target set \cite{polyak1987,khalil2002,benaim1996}. An alternative approach is to consider the distance from a target set and show that such a distance vanishes eventually via a suitable technical result on the convergence of the distance-valued sequence of real numbers.

In this work, we focus mostly on the latter methodology. To explain our choice, let us note that solving an optimization problem consist of \textit{designing a sequence} of vectors that converge to the solution, the minimum of a given cost function. Similarly, in algorithmic game theory, one usually aims at constructing a \textit{sequence} that converge to an equilibrium, e.g., a Nash equilibrium, the optimum for each player given the actions of the other players.
The key point here is that, in general, the target set is not known a priori, yet the distance of the constructed sequence from such set can be analyzed anyways. On the contrary, in Lyapunov stability analysis, the target set is usually known a priori.

By exploiting the relation between the iterations and a suitable distance-like function, we show in this paper that convergence theorems represent a key ingredient for a wide variety of system-theoretic problems in fixed-point theory, game theory and optimization \cite{polyak1987,bau2011,facchinei2007,combettes2001,eremin2009}.
In many cases, the study of iterative algorithms allows for a systematic analysis that follows from the concept of F\'ejer monotone sequence. The basic idea behind F\'ejer monotonicity is that at each step, each iterate is closer to the target set than the previous one. In a sense, the distance used for F\'ejer sequences can be seen as a specific class of Lyapunov function and F\'ejer monotonicity shows that it is decreasing along the iterates. The concept was first introduced in 1922 \cite{fejer1922}, but the term F\'ejer monotone sequence was first used thirty years later in 1954 \cite{motzkin1954} and a huge part of the studies on its properties was made in the 60s \cite{eremin2009,eremin1969,eremin1968,eremin1968speed} and still continues \cite{combettes2001,combettes2013,combettes2015,combettes2000,kohlenbach2018}.

Unfortunately, F\'ejer monotonicity is hard to obtain, therefore the concept is typically relaxed to a quasi-F\'ejer property, where a vanishing error must be considered. Such an error term in the distance inequality is common in many equilibrium problems \cite{franci2020fb,iusem2017,malitsky2020siam,duvocelle2019,kannan2013,nguyen2017,polyak1987,duflo2013,bau2011}, especially in the stochastic case where the concept of quasi-F\'ejer monotone sequence was first introduced \cite{ermoliev1988,ermol1969}. However, these properties are not necessarily enough to ensure convergence, hence, (quasi) F\'ejer monotonicity is often used in combination with convergence results on sequences of real numbers. These technical results have been used in many theoretical and computational applications that range from stochastic Nash equilibrium seeking \cite{franci2020fb, koshal2013,franci2019fbhf} to machine learning \cite{franci2020gan,bot2020gan, duvocelle2019}.


\begin{table*}
\begin{center}
\begin{tabular}{cccc}
\toprule
 {\bf{Result}} & {\bf{Reference}} & {\bf{Application}} &  {\bf{Reference}}\\
\midrule
Proposition \ref{prop_bau} & \cite[Proposition 5.4]{bau2011} & &\\
Theorem \ref{theo_qfm_comb} & \cite[Theorem 3.8]{combettes2001}& &\\
 Lemma \ref{lemma_opial} & \cite{opial1967} (Opial) & MI - Theorem \ref{malitzky_theo} & \cite[Theorem 2.5]{malitsky2020siam}\\
 & & VI - Theorem \ref{theo_mali2}& \cite[Theorem 1]{malitsky2019}\\
\midrule
 Lemma \ref{lemma_comb} & \cite[Lemma 3.1]{combettes2001}&&\\
 Lemma \ref{lemma_det_rs}& \cite[Lemma 5.31]{bau2011}& VI - Theorem \ref{theo_mali2}& \cite[Theorem 1]{malitsky2019}\\
  Corollary \ref{cor_mali1}    &  \cite[Lemma 2.8]{malitsky2015}    & VI - Theorem \ref{theo_mali3}& \cite[Theorem 3.2]{malitsky2015} \\
 & & LYAP - Theorem \ref{thoe_lyap}& \cite[Theorem 1.4.1]{polyak1987}\\
   Corollary \ref{cor_polyak}  &    \cite[Lemma 2.2.2]{polyak1987}  & &\\
   Lemma \ref{lemma_polyak1}  &   \cite[Lemma 2.2.3]{polyak1987}   & NE - Theorem \ref{theo_kannan}& \cite[Theorem 2.4]{kannan2012} \\
 Lemma \ref{lemma_neg_rs}    &      & &\\
  Lemma \ref{lemma_xu03}   &    \cite[Lemma 2.1]{xu2003}  & &\\
  Lemma \ref{lemma_xu02}   &   Extension of \cite[Lemma 2.5]{xu2002}   & NE - Theorem \ref{thoe_duvo}& \cite[Theorem 3.1]{duvocelle2019}  \\
Corollary \ref{cor_rem} & \cite[Proposition 3]{lei2020cdc} \\
  Corollary \ref{cor_qin}   &  \cite[Lemma 1.1]{qin2008}   & &\\
   Corollary \ref{cor_liu}  &  \cite[Lemma 3]{xu1998}    & MI - Theorem \ref{theo_dadashi}& \cite[Theorem 3.1]{dadashi2019} \\

   Proposition \ref{prop_alber}  &    \cite[Proposition 2]{alber1998}  & &\\
   Lemma \ref{lemma_he} & \cite[Lemma 7]{he2013} &&\\
   Lemma \ref{lemma_meno} & \cite[Lemma 2.2]{mainge2008}&&\\
  Lemma \ref{lemma_rate}   &  \cite[Lemma 2.7]{malitsky2018fbf}]   & MI - Theorem \ref{theo_mali_rate}& \cite[Theorem 2.9]{malitsky2020siam}\\
\midrule
  Proposition \ref{prop_vm}   &  \cite[Proposition 3.2]{combettes2013}     & MI - Theorem \ref{theo_vu}& \cite[Theorem 3.1]{vu2013} \\
  Theorem \ref{theo_comb_var}   &  \cite[Theorem 3.3]{combettes2013}    & MI - Theorem \ref{theo_vu}& \cite[Theorem 3.1]{vu2013}\\
 Proposition  \ref{prop_comb_vm}  &   \cite[Proposition 4.1]{combettes2013}    & &\\
\bottomrule
\end{tabular}
\caption{Convergence results for F\'ejer monotone sequences, deterministic sequences of real numbers and with variable metric (separated by the horizontal lines, respectively). For the applications, MI stands for Monotone Inclusion, VI for variational inequalities, NE for Nash Equilibrium problems and LYAP for Lyapunov analysis.}\label{table_det}
\end{center}
\end{table*}
\begin{table*}
\begin{center}
\begin{tabular}{cccc}
\toprule
 {\bf{Result}} & {\bf{Reference}} & {\bf{Application}} &  {\bf{Reference}}\\
\midrule
   Lemma \ref{lemma_rs}  & \cite{RS1971}  (Robbins--Siegmund) &  VI - Theorem \ref{theo_bot}& \cite[Theorem 4.5]{bot2020}\\
 & & VI - Theorem \ref{theo_iusem}& \cite[Theorem 3.18]{iusem2017}\\
 & & NE - Proposition \ref{prop_koshal}& \cite[Proposition 3]{koshal2013}\\
&& MPC - Proposition \ref{prop_mpc}&  \cite[Proposition 1]{lee2015}\\
 Lemma \ref{lemma_glad} & \cite[Lemma 2.2.9]{polyak1987}  (Gladyshev)     &     & \\
 Corollary \ref{cor_poggio}    &   \cite[Theorem B.2]{poggio2011}  & & \\
   Corollary \ref{cor_duflo}  &    \cite[Corollary 1.3.13]{duflo2013}  & LLN - Corollary \ref{lln}& \cite[Theorem 1.3.15]{duflo2013}\\
Lemma \ref{lemma_rs_comb} & \cite[Lemma 2.1]{combettes2019}&&\\
   Lemma \ref{lemma_fake_rs}  &   \cite[Lemma 2.2.10]{polyak1987}   & VI - Theorem \ref{theo_yousef} &\\
&&NET - Theorem \ref{theo_op_dyn} & \cite[Theorem 5]{shi2013}\\
\midrule
   Proposition \ref{prop_gen_rs}  &  \cite[Proposition 2.3]{combettes2015}   & &\\
Proposition \ref{prov_vm_rs} & \cite[Proposition 2.4]{vu2016} & &\\
\bottomrule
\end{tabular}
\caption{Convergence results for stochastic sequences of real random variables and stochastic F\'ejer monotone sequences (separated by double horizontal lines, respectively). For the applications, VI stands for variational inequalities, NE for Nash Equilibrium problems, MPC for Model Predictive Control, LLN for Law of Large Numbers and NET for Network control problems.}\label{table_stoc}
\end{center}
\end{table*}
\subsection{What this survey is about}
In this survey, we present a number of convergence theorems for sequences of real (random) numbers. We show how they can be used in combination with (quasi) F\'ejer monotone sequences or Lyapunov functions to obtain convergence of an iterative algorithm, essentially a discrete-time dynamical system, to a desired solution. Moreover, we present some applications to show how they can be adopted in a variety of settings. Specifically, we present convergence results for both deterministic and stochastic sequences of real numbers and we also include some results on F\'ejer monotone sequences and with variable metric. We show that these results help proving not only convergence of an iterative algorithm but also the Law of Large Numbers, with applications in model predictive control \cite{lee2015} and opinion dynamics \cite{shi2013} among others.

We report in Tables \ref{table_det} and \ref{table_stoc} the results for deterministic and stochastic sequences respectively, with the corresponding bibliographic source and application. 

The paper is organized as follows. In the next section, we recall some preliminary notions on the concept of ``convergence'' and of random variables. Section \ref{sec_det} is devoted to deterministic convergence results while the stochastic case is discussed in Section \ref{sec_stoc}. An extension with variable metric is considered in Section \ref{sec_vm}. Sections \ref{sec_app_det}, \ref{sec_app_stoc} and \ref{sec_app_vm} propose applications of the convergence lemmas for deterministic, stochastic, and variable metric sequences, respectively.

\subsection{What this survey is not about}

This is not a survey on solution algorithms for optimization problems and variational inequalities. Some relevant references on iterative methods include \cite{bau2011,facchinei2007,combettes2020,polyak1987,doob1953,rockafellar1970} and the references therein.

We also remark that, despite the notion of F\'ejer sequence is used throughout the paper, this is not a survey on the properties of F\'ejer monotone sequences. The interested reader may refer to \cite{combettes2001,berg1995,bau2011,combettes2015,combettes2013,combettes2000,kohlenbach2018}.

\section{Notation and Preliminaries}\label{sec_notation}
We use the nomenclature and notation from \cite{bau2011,rockafellar1970}.

$\NN$ indicates the set of natural numbers and
$\RR$ ($\bar\RR=\RR\cup\{\infty\}$) is the set of (extended) real numbers.
$\langle\cdot,\cdot\rangle:\RR^n\times\RR^n\to\RR$ denotes the standard inner product and $\norm{\cdot}$ is the associated Euclidean norm. $\BB=\{x\in\RR^n\mid\norm{x}\leq1\}$ represents the unit ball.
Let $d_{\mc X}(x)=\min_{y\in\mc X}\norm{x-y}$ be the distance between $x$ and the set $\mc X$.

We indicate that a matrix $A$ is positive definite, i.e., $x^\top Ax>0$, with $A\succ0$. 
Given a symmetric $W\succ0$, the $W$-induced inner product is $\langle x, y\rangle_{W}=\langle W x, y\rangle$ and the associated norm is defined as $\norm{x}_{W}=\sqrt{\langle W x, x\rangle}$. $\op{Id}$ is the identity operator. 
Given a continuous linear operator $T:\RR^n\to\RR^n$, the \textit{adjoint} of $T$ is the unique continuous linear operator $T^*$ such that $\forall x,y\in\RR^n$ $\langle Tx,y\rangle=\langle x,T^*y\rangle.$
Let $S(\RR^n)$ be the set of self-adjoint bounded linear operators of $\RR^n$ and let the Loewner partial order be defined for all $T_1,T_2\in S(\RR^n)$ as $T_{1} \succeq T_{2} \Leftrightarrow\forall x \in \RR^n$ $\langle T_{1} x | x\rangle \geq\langle T_{2} x | x\rangle$.
Let $\beta\geq0$ and $\mathcal{P}_{\beta}=\{L \in S(\RR^n) | L \succeq \beta \mathrm{Id}\}.$ Positive semidefinite matrices belongs to $\mathcal{P}_{\beta}$.

Unless otherwise mentioned, we use $v$, $u$ and $w$ for (real or random) numbers while we use $x$, $y$, $z$ to indicate vectors (of real numbers or random variables), i.e., $v,u,w\in\RR$ and $x,y,z\in\RR^n$, respectively. 
Capital letters indicate operators or matrices. Letters from the Greek alphabet are also used for real numbers but they mostly represent errors ($\varepsilon$), step size sequences ($\alpha$) or coefficients ($\delta$, $\gamma$); $\xi$ often indicates random quantities. Since it may be dependent on the context, when necessary, the meaning is introduced along with the symbol. In general, calligraphic capital letters indicate sets, $\mc C$ indicates a convex set and $\mc S$ a target or solution set.
Throughout the survey, we suppose that the sequence $(x^k)_{k\in\NN}$ belongs to a set $\mc X\subseteq\RR^n$. Further assumptions will be made when necessary.

Given a vector $x\in\RR^n$, we indicate the maximum entry as $x_{\max }=\max _{1 \leq i \leq N}\{x_{i}\}$ and, analogously, the minimum entry as $x_{\min }=\min _{1 \leq i \leq N}\{x_{i}\}$. Most often, the superscript $^*$, e.g., $x^*$, indicates a solution of the problem, while the bar, i.e., $\bar x$, indicates an accumulation point of an iterative process.

With reference to the application sections, we use Standing Assumptions to state technical conditions that implicitly hold throughout the paper, while Assumptions are postulated only when explicitly used.

More notation and definitions related to monotone operator theory, functional to the application sections, are postponed to Appendix \ref{appendix}.


\subsection{Convergence notions}
Let us first recall some definitions related to the notion of convergence itself.
\begin{defn}
A sequence $(x^{k})_{k\in\NN}\subseteq\RR^n$ is said to converge weakly to a point $\bar x\in \mc X$ if, for all $y \in \mc X$,
$$\langle x^{k}, y\rangle \to\langle \bar x, y\rangle \text{ as } k\to\infty.$$ 
A sequence $(x^{k})_{k\in\NN}\subseteq\RR^n$ is said to converge strongly to a point $x\in \mc X$ if 
$$\lim_{k\to\infty}\norm{x^k-x}=0.$$ 
\end{defn}
In general, strong convergence implies weak convergence. In finite dimension, the two notions are equivalent \cite[Lemma 2.51]{bau2011}, hence, in this paper, we generally talk about convergence.

Given the definition of convergence, let us define the concept of cluster point.
 \begin{figure}[t]
\centering
\includegraphics[width=\columnwidth]{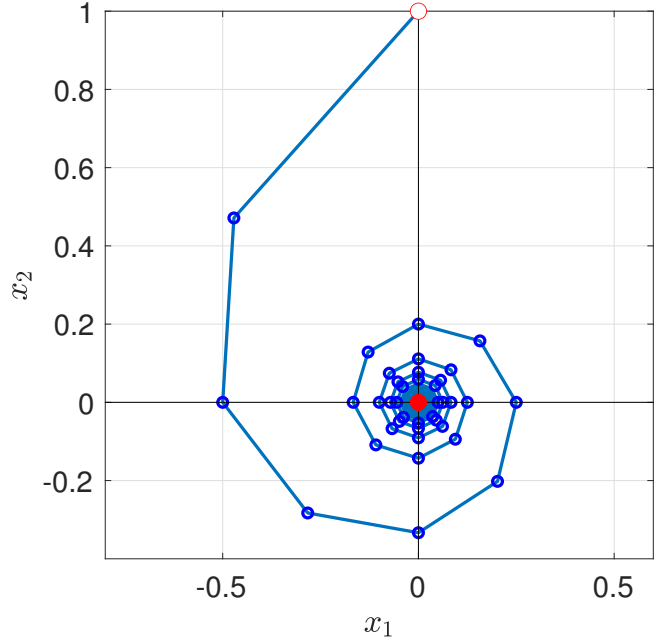}
\caption{Sequence converging to a cluster point which is also a sequential cluster point (Example \ref{ex_cluster}). The empty dot represents the initial point and the red dot is the cluster point.}\label{fig_cluster}
\end{figure}

\begin{defn}
A point $\bar x \in \RR^n$ is said to be a cluster point (or limit point or accumulation point) of a sequence $(x^k)_{k \in \NN} $ if, for every $\epsilon>0$ and for $\bar k\in\NN$ there exists $k\geq \bar k$ such that $x^k\in\{\bar x\}+\epsilon \BB$.
In other words, there is at least one $\bar k \in \NN$ such that $x^k$ lies in a neighborhood of $\bar x$ for all $k\geq \bar k$.\\
The set of all cluster points is called limit set.\\
If a sequence $(x^k)_{k \in \NN}$ in $\RR^n$ has a subsequence that converges to a point $\bar x\in \RR^n$, then $ \bar x$ is called a sequential cluster point of $(x^k)_{k \in \NN}$.
 \end{defn}

\begin{exmp}[A cluster point is also a sequential cluster point]\label{ex_cluster}
Consider the sequence $(x^k)_{k\in\NN}\subseteq\RR^2$ defined as $x^k=\frac{1}{k}\left(cos\left(k\frac{\pi}{2}\right),sin\left(k\frac{\pi}{2}\right)\right)$. The sequence converges to $\bar x=(0,0)$ as $k\to\infty$, which is a cluster point and a sequential cluster point, as shown in Figure \ref{fig_cluster}. The limit set is the singleton $\{(0,0)\}$. 
\end{exmp}

 \begin{figure}[t]
\centering
\includegraphics[width=\columnwidth]{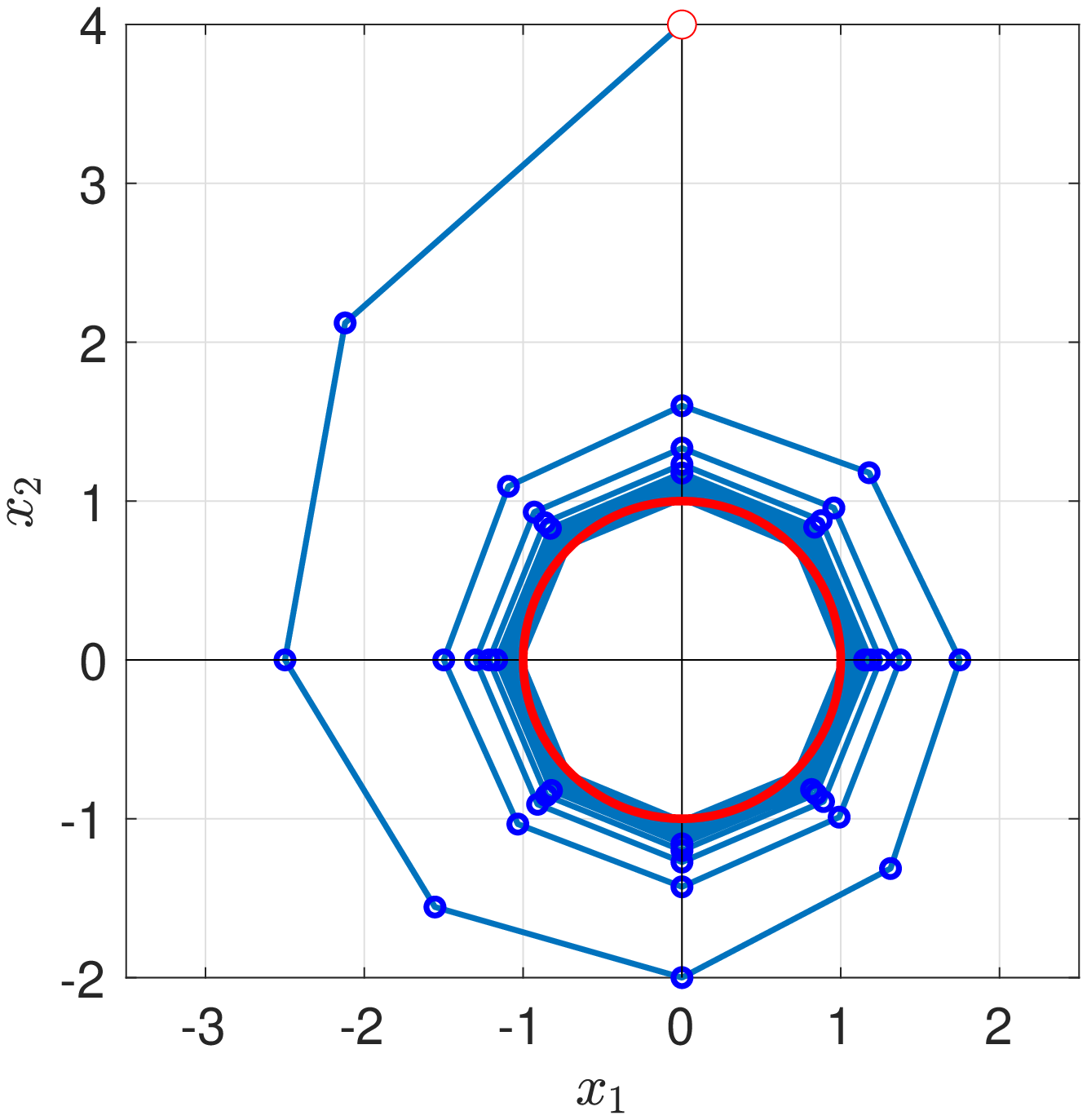}
\caption{Sequence converging to a set of sequential cluster points where none of them is a cluster point (Example \ref{ex_no_cluster}). The empty dot represents the initial point and the red circle is the limit set.}\label{fig_no_cluster}
\end{figure}

\begin{exmp}[A sequential cluster point is not necessarily a cluster point]\label{ex_no_cluster}
Consider the sequence $(x^k)_{k\in\NN}\subseteq\RR^2$ defined as $x^k=\frac{k+3}{k}\left(cos\left(k\frac{\pi}{2}\right),sin\left(k\frac{\pi}{2}\right)\right)$. The sequence does not converge but it has many sequential cluster points (see Figure \ref{fig_no_cluster}). For instance, consider $\bar x=(1,0)=(cos(2\pi),sen(2\pi))$. Then, the subsequence $(x^{k_n})$ with $k_n=4n$, $n\in\NN$ converges to $\bar x$ which in turn is a sequential cluster point. However, the limit set is given by the circumference $\{(x_1,x_2)\in\RR^2:x_1^2+x_2^2=1\}$, in red in Figure \ref{fig_no_cluster}.
 \end{exmp}

\begin{exmp}[$\omega$-limit set]
The concept of limit set reminds that of $\omega$-limit set \cite{benaim1996,benaim1999}.
Given a continuous function $f:\RR\to \RR$, the $\omega$-limit set is the set of cluster points of the forward orbit of the iterated function $f$ at $x\in \RR$, namely, 
$$\omega(x,f)=\bigcap_{n\in\NN}\{f^k(x):k>n\}.$$
In particular, given a dynamical system with flow $\phi:\RR\times \RR\to \RR$, $y$ is a $\omega$-limit point of $x$ if there exist $(t^k)_{k\in\NN}\subseteq\RR$ such that $\lim_{k\to\infty} t^k=\infty$ and $\lim_{k\to\infty} \phi(t^k,x)=y$. 
\end{exmp}

Let us conclude this section with some preliminary results related to the convergence properties of a given sequence. We consider these results common knowledge and we refer to them throughout the paper, even without a specific reference.

\begin{lem}\label{lemma_bounded}
\cite[Lemma 2.45]{bau2011} Let $(x^k)_{k \in \NN}$ be a bounded sequence in $\RR^n$. Then, $(x^k)_{k \in \NN}$ possesses a convergent subsequence.
\end{lem}

\begin{lem}\label{lemma_cluster}
\cite[Lemma 2.46]{bau2011} Let $(x^k)_{k \in \NN}$ be a sequence in $\RR^n$. Then, $(x^k)_{k \in \NN}$ converges if and only if it is bounded and possesses at most one sequential cluster point.
\end{lem}

\begin{lem}\label{lemma247}
\cite[Lemma 2.47]{bau2011} Let $(x^k)_{k \in \NN}$ be a sequence in $\RR^n$ and let $\mc X$ be a nonempty subset of $\RR^n$. Suppose that, for every $x\in \mc X$, $(\|x^k-x\|)_{k\in\NN}$ converges and that every sequential cluster point of $(x^k)_{k \in \NN}$ belongs to $\mc X$. Then, $(x^k)_{k \in \NN}$ converges to a point in $\mc X$. \end{lem}

\begin{exmp}[Assumptions of Lemmas \ref{lemma_bounded}, \ref{lemma_cluster} and \ref{lemma247}]\label{ex_lemmi}
Consider the sequence defined by $x^k=(-1)^k$ for all $k\in\NN$. Both $1$ and $-1$ are sequential cluster points but not cluster points. The sequence does not converge. Let us use $(x^k)_{k\in\NN}$ to verify Lemmas \ref{lemma_bounded}, \ref{lemma_cluster} and \ref{lemma247}. \\
The sequence is bounded in $[-1,1]$ and it has (at least) two convergent subsequences: $x^{k_n}=-1$ and $x^{k_m}=1$, $n,m\in\NN$. Hence, Lemma \ref{lemma_bounded} holds. However, the sequence is not convergent. In fact, contrary to Lemma \ref{lemma_cluster}, it has two sequential cluster points. Concerning Lemma \ref{lemma247}, we note that the sequence $(\norm{x^k-\bar x})_{k\in\NN}$ does not converge for any $\bar x\in[-1,1]\setminus\{0\}$. On the other hand, it converges for $\bar x=0$ which is not a cluster point of the sequence $(x^k)_{k\in\NN}$.
 \end{exmp}

\subsection{Probability theory}

Concerning the stochastic case, we focus on almost sure convergence. Let us first introduce the probability space $(\Omega, \mc F,\PP)$ where $\Omega$ is the sample space, $\mc F$ is the event space, and $\PP$ is the probability function defined on the event space. The symbol $\EE$ is used for the associated expected values.
\begin{defn}
A sequence $(x^k)_{k \in \NN}$ of random variables converges almost surely (a.s.) towards $\bar x\in \mc X$ if
$$\PP \!\left[\lim _{n\to \infty }\!x^{k}=\bar x\right]=1. $$
\end{defn}

From now on, results involving random variables are supposed to hold almost surely, even if it is not explicitly mentioned.
\begin{figure}
\begin{center}
\includegraphics[width=.7\textwidth]{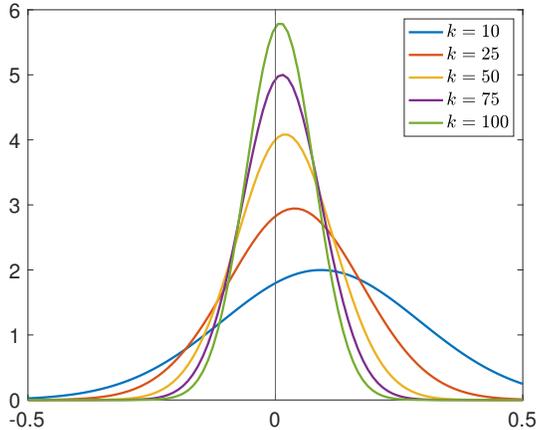}
\caption{Graphical representation of a.s. convergence. By increasing the number of iterations, the mass of the probability distribution concentrates on the limit point (Example \ref{ex_as}).}\label{fig_as}
\end{center}
\end{figure}
\begin{exmp}[a.s. convergence]\label{ex_as}
Let $\Omega=[0,1]$ be a continuous sample space with the uniform probability distribution. For $\omega\in[0,1]$, let us define the sequence of random variables $v^k(\omega)=\omega+\omega^k$ and the random variable $\bar v(\omega)=\omega$. Then, for all $\omega\in[0,1)$, $v^k(\omega)$ converges to $\bar v(\omega)$ as $k\to\infty$. Instead, if $\omega=1$, $v^k(\omega)=2$ for all $k\in\NN$ and $v^k(\omega)$ does not converge to $\bar v(\omega)$. However, $\PP[\omega\in[0,1)]=1$, hence $v^k(\omega)$ converges a.s. to $\bar v(\omega)$ as $k\to\infty$. In Figure \ref{fig_as}, we show how the distance from the limit point move toward zero with high probability increasing the number of iterations, i.e., $\PP[\lim_{k\to\infty}v^k(\omega)-\bar v(\omega)=0]=1$.

\end{exmp}

Let us recall some probabilistic and stochastic definitions that will be useful later on. We start with the definition of filtration.

\begin{defn}\label{def_marti}
Let $(\xi^k)_{k\in\NN}\subseteq\Omega$ be a sequence of random variables and let $ \mathcal F_k $, $k\in\NN$ be the $\sigma$-algebra of $ \Omega$, generated by the events prior to $k$, that is, $\mathcal{F}_{0} = \sigma\left(X_{0}\right)$ and 
$\mathcal{F}_{k} = \sigma\left(X_{0}, \xi_{1}, \xi_{2}, \ldots, \xi_{k}\right)$ for all $k \geq 1.$
Then $\mc F= (\mathcal F_k)_{k \in \NN}$ is called a filtration, if $\mathcal F_k \subseteq \mathcal F_\ell \subseteq \Omega$ for all $ k \leq \ell $. 
 \end{defn}
In words, a filtration is a family of $\sigma$-algebras non-decreasingly ordered that collects the history of $\xi^k$. Given a filtration, a subsequent important concept is that of martingale \cite[Chapter 7]{doob1953}, \cite[Section 1.9]{chung1990}, \cite[Section 4.1]{kushner2003}.

\begin{defn}\label{def_marti}
A sequence of random variables $(v^k)_{k \in \NN}$ is said to be a martingale adapted to $\mc F=(\mathcal F_k)_{k \in \NN}$ if it is integrable and for all $k\in\NN$,
$$\EEk{v^{k+1}}=v^k.$$
It is a supermartingale if for all $k\in\NN$
$$\EEk{v^{k+1}}\leq v^k,$$
and a submartingale if for all $k\in\NN$
$$\EEk{v^{k+1}}\geq v^k. $$
\end{defn}
These notions are the stochastic generalization of the notion of monotone (decreasing or increasing) sequences. Moreover, we note that every martingale is a submartingale and a supermartingale, while every sequence which is \textit{both} a submartingale and a supermartingale is also a martingale.
\begin{exmp}[Martingales]
Let $(x^k)_{k\in\NN}$ be the sequence generated by the fortune of a gambler after $k$ tosses of a fair coin. The gambler wins $1$ if the coin comes up head (with probability $p=\frac{1}{2}$) and loses $1$ otherwise. The expected fortune after the next toss is equal to the present fortune, i.e., $\EE[x^{k+1}]=x^k$, hence the sequence is a martingale. \\
Let us now consider the toss of a biased coin, with head coming up with probability $p\neq \frac{1}{2}$. If $p>\frac{1}{2}$, on average the gambler wins money, i.e., $\EE[x^{k+1}]\geq x^k$ and the sequence is a supermartingale. On the other hand, if $p<\frac{1}{2}$ the gambler loses and the sequence is a submartingale. See \cite{stroock2010,duflo2013,polyak1987,doob1953,borkar1995} for other examples.
 \end{exmp}
We conclude this section with the following result, due to Doob \cite[Theorem 7.4.1]{doob1953}, \cite[Lemma 2.2.7]{polyak1987}, \cite[Theorem 3.3.1]{borkar1995}.
\begin{thm}[Martingale convergence theorem]\label{teo_doob}
Suppose $(x^k)_{k \in \NN}$ is a nonnegative super-martingale which satisfies
$\sup _{k\in\NN} \EE[|x^k|]<\infty.$
Then, almost surely, there exists $\bar x\geq 0$ such that $\lim_{k\to\infty}x^k = \bar x$ and $\EE[|\bar x|]<\infty$.
\end{thm}

\subsection{Distance from a target set}

The basic idea for proving convergence of a sequence is that the distance from the solution should vanish or at least decrease at each iteration. This is particularly important when we consider vectors, i.e., when convergence results for sequences of real numbers cannot be applied directly.

The most used concept in this direction is that of F\'ejer monotone sequence. 
The term was coined in \cite{motzkin1954} but the concept was first proposed by F\'ejer in \cite{fejer1922}.
These processes have been widely studied in the literature \cite{combettes2001,combettes2015,bauschke2003,eremin1969,combettes2000} since they can be applied in solving classical problems as systems of equations or inequalities, operator equations with a priori information, equilibrium problems, and many others \cite{eremin2009,combettes2001,combettes2000,iusem2017,bot2020}. The key point is that one can take the target set to be the solution set of the problem of interest (even if it is unknown). Then, since the distance from the target decreases, the sequence will eventually reach (a point close to) the solution.

\begin{defn}\label{def_Fejer}
A sequence $(x^k)_{k \in \NN}$ is F\'ejer monotone with respect to a target set $\mc S\neq\varnothing$ if for every $\bar x\in \mc S$, it holds that for all $k\in\NN$
$$\norm{x^{k+1}-\bar x}\leq\norm{x^k-\bar x}. $$
\end{defn}
In words, Definition \ref{def_Fejer} states that the distance between the iterates and any point $\bar x\in S$ does not increase. 
\begin{exmp}[F\'ejer monotone sequence of numbers]
Let us consider the sequence $v^k=\frac{(-1)^k}{k}$. Though the sequence is oscillating, it is convergent to $\bar v=0$ and it is F\'ejer monotone with respect to $\mc S=\{0\}$.
 \end{exmp}
\begin{exmp}[F\'ejer monotone sequence of vectors]\label{ex_fejer_proj}
An example of a F\'ejer monotone sequence is the one generated by the projection (Definition \ref{def_projproxres}) onto a nonempty, closed and convex set $\mc C$ \cite{eremin2009,combettes2001,berg1995,combettes2000,gubin1967}, i.e.,
$$x^{k+1}=\op{proj}_{\mc C}(x^k).$$
The claim follows immediately from the fact that the projection operator is firmly nonexpansive \cite[Proposition 4.16]{bau2011}, hence nonexpansive (Definition \ref{def_lip}). In fact, any sequence generated by an iteration of the form $x^{k+1}=T(x^k)$ where $T$ is a nonexpansive operator is a F\'ejer monotone sequence \cite[Equation (2)]{combettes2000}.
 \end{exmp}

We remark that the diminishing distance from a target point does not necessarily imply convergence to such point. Specifically, we note that a F\'ejer monotone sequence $(x^k)_{k \in \NN}$ with respect to a nonempty set $\mc S$ may not converge even if the limit set is not empty. 
\begin{exmp}[Non-convergent F\'ejer monotone sequence]\label{fm_no_conv}
The sequence defined as $x^k=(-1)^{k} x_0$ for all $k\in\NN$ is F\'ejer monotone with respect to $\mc S=\{0\}$ but it does not converge for any $x_0 \notin \mc S$ \cite[page 9]{combettes2001}, \cite[page 1]{combettes2000}. 
 \end{exmp}

The notion of F\'ejer monotonicity can be extended in various directions \cite{berg1995,combettes2001,eremin1969,combettes2013,combettes2015,lin2018}. Here we recall only the concept of quasi-F\'ejer monotone sequence, first introduced in the stochastic case \cite{ermoliev1988,ermol1969} (see also Definition \ref{def_Fejer_stoc}) and later in several (deterministic) variants \cite{combettes2001,combettes2004,lin2018,berg1995}. 
\begin{defn}\label{def_qfm}
Let $\phi:\RR_{\geq0}\to\RR_{\geq0}$.
A sequence $(x^k)_{k \in \NN}\subseteq\RR^n$ is quasi-F\'ejer monotone with respect to a target set $\mc S\neq \varnothing$ if for every $\bar x \in \mc S$ there exists a nonnegative sequence $(\varepsilon^k)_{k\in\NN}$ such that $\sum_{k=0}^\infty\varepsilon^k<\infty$ and it holds that
$$\phi(\|x^{k+1}-\bar x\|) \leq\phi(\|x^k-\bar x\|)+\varepsilon^k \text{ for all } k\in\NN.  $$
\end{defn}

\begin{rem}
Definition \ref{def_qfm} is perhaps the most general definition of quasi-F\'ejer monotone sequence, as there are no restrictions on the function $\phi$. However, besides some general results (see, e.g., Proposition \ref{prop_vm} and Theorem \ref{theo_comb_var}), many convergence theorems hold for a given choice of the function, i.e., $\phi=|\cdot|$ or $\phi=|\cdot|^2$. For details, see Section \ref{subsec_Fejer} or \cite{alber1998,ermol1969,ermoliev1988,combettes2001,combettes2013}.
 \end{rem}

Next, we give a definition of F\'ejer monotone sequence in the stochastic case. Stochastic quasi-F\'ejer monotone sequences were first introduced in \cite{ermol1969} and later discussed in \cite{barty2007, combettes2015}. The interpretation is that the expected value of the distance from the target set is non-increasing, which reminds the definition of (super)martingale \cite{combettes2015,ermol1969}.

\begin{defn}\label{def_Fejer_stoc}
Let $\phi:\RR_{\geq0}\to\RR_{\geq0}$.
A sequence $(x^k)_{k \in \NN}$ of random variables is stochastic F\'ejer monotone with respect to a target set $\mc S\neq \varnothing$ if for every $\bar x \in \mc S$ it holds that, for all $k\in\NN$,
$$\EEk{\phi(\norm{x^{k+1}-\bar x})} \leq\phi(\norm{x^k-\bar x}).$$
It is called stochastic quasi-F\'ejer monotone relative to a target set $\mc S\neq \varnothing$ if for every $\bar x \in \mc S$ there exists a nonnegative sequence $(\varepsilon^k)_{k\in\NN}$ such that $\sum_{k=0}^\infty\varepsilon^k<\infty$ and it holds that, for all $k\in\NN$,
$$\EEk{\phi(\norm{x^{k+1}-\bar x})} \leq\phi(\norm{x^k-\bar x})+\varepsilon^k. $$
\end{defn}

Definitions \ref{def_qfm} and \ref{def_Fejer_stoc} hold true for any norm of choice, yet other metrics can be considered (see Remark \ref{remark_bregman}). Moreover, variable metrics have been considered as well \cite{combettes2013,vu2013,nguyen2016}.

\begin{defn}
Let $\beta\geq0$ and $\phi:\RR_{\geq0}\to\RR_{\geq0}$ and let $(W_k)_{k\in\NN}$ be a sequence in $\mc P_\beta$. A sequence $(x^k)_{k \in \NN}$ of random variables is quasi-F\'ejer monotone with respect to a target set $\mc S\neq \varnothing$ and relative to $(W_k)_{k\in\NN}$, if, given a nonnegative sequence $(\eta^k)_{k\in\NN}$ such that $\sum_{k=0}^\infty\eta^k<\infty$, for every $\bar x \in \mc S$ there exists a nonnegative sequence $(\varepsilon_k)_{k\in\NN}$ such that $\sum_{k=0}^\infty\varepsilon^k<\infty$ and for all $k\in\NN$
$$
\phi(\|x^{k+1}-\bar x\|_{W_{k+1}}) \leq(1+\eta_{k})\,\phi(\|x^k- \bar x\|_{W_k})+\varepsilon^k. 
$$
\end{defn}

There are many results on (stochastic, quasi) F\'ejer monotone sequences but they lie outside the scope of this survey. For a deeper insight on this topic we refer to \cite{combettes2001, bau2011,combettes2013,combettes2015,bau2015}. 

\begin{rem}\label{remark_bregman}
An important generalization of F\'ejer monotonicity is that of Bregman monotonicity \cite{facchinei2007,bauschke2003,bregman1967,nguyen2016,nguyen2017}. The concept has received a rising interest recently in the system and control community \cite{ananduta2021,bravo2018,benning2017,alacaoglu2021,mertikopoulos2018}. For the sake of completeness, we report here the definition, and later on we recall when some results hold also with the Bregman distance.

Let $\mc C\subseteq\RR^n$ be a closed convex set and let $f:\mc C\to\RR$ be a strictly convex continuous function which is continuously differentiable on $int\mc C$. The Bregman distance associated with $f$ is
\begin{equation}
D_{f}(x, y) = f(x)-f(y)-\langle \nabla f(y),x-y\rangle
\end{equation}
and it has the following geometric interpretation: $D_f (x, y)$ is the difference between $f (x)$ and the value at $x$ of the linearized approximation of $f(x)$ at $y$. $D_f (x, y)$ is nonnegative and it is zero if and only if $x = y$. 

We note that in general the Bregman distance is not a ``real" distance, since it may fail to satisfy, for instance, the triangular inequality.

An example of a Bregman function is $f=\normsq{\cdot}$ whose associated distance is $D_f(x, y) =\|x-y\|^{2} / 2$.
Another example is given by $g(x) = \sum_{i=1}^{n} x_{i} \log x_{i}$ with the convention that $0 \log 0=0$. The associated distance is $D_{g}(x, y)=\sum_{i=1}^{n}(x_{i} \log \frac{x_{i}}{y_{i}}+y_{i}-x_{i})$ \cite[Example 12.7.4]{facchinei2007}, i.e., the Kullback--Leibler divergence \cite{kullback1951,kullback1997}, widely used in machine learning and generative adversarial networks \cite{goodfellow2014,goodfellow2016}.

A sequence $(x^k)_{k \in \mathbb{N}}$ in $\mc C$ is Bregman monotone with respect to a set $\mc S$ if the following conditions hold:
\begin{itemize}
\item[(i)] $\mc S \cap \mc C\neq \varnothing$,
\item[(ii)] $(x^k)_{k \in \mathbb{N}}$ lies in $int (\mc C)$,
\item[(iii)] for every $ \bar x \in \mc S \cap \mc C$, $ D_f(\bar x, x^{k+1}) \leq D_f(\bar x, x^k)$ for all $k \in \mathbb{N}$. 
\end{itemize}
\end{rem}

\section{Convergence of deterministic sequences}\label{sec_det}
In this section, we walk through a number of convergence results for deterministic sequences of real numbers. When possible, we propose first the most general result and then show its consequences. We start with some results on F\'ejer monotone sequences and then move to general sequences of real numbers.

\subsection{F\'ejer monotone convergent sequences}\label{subsec_Fejer}

The first result we present is related to the concept of F\'ejer monotone sequences and it was originally proposed in \cite{bau2011}. Parts of this result are also in \cite[Theorems 2.7 and 2.10]{berg1995} while in \cite[Propositions 1--4]{combettes2000} a distinction between strong and weak convergence is made. Other properties of F\'ejer monotone sequences can be found in \cite{bau2011,berg1995,alber1998,combettes2000, combettes2001} and reference therein.

\begin{prop}[Proposition 5.4, \cite{bau2011}]\label{prop_bau}
Let $\mc X$ be a nonempty subset of $\RR^n$ and let $(x^k)_{k \in \mathbb{N}}$ be a sequence in $\RR^n$. Suppose that $(x^k)_{k \in \mathbb{N}}$ is Fejér monotone with respect to $\mc X$. Then, the following statements hold:

\begin{itemize}
\item[(i)] $(x^k)_{k \in \mathbb{N}}$ is bounded;
\item[(ii)] For every $\bar x \in \mc X,(\|x^k-\bar x\|)_{k \in \mathbb{N}}$ converges;
\item[(iii)] $(d_{\mc X}(x^k))_{k \in \mathbb{N}}$ is decreasing and converges;
\item[(iv)] $\left\|x^{k+m}-x^k\right\| \leq 2 d_{\mc X}(x^k)$ for all $m,k \in \mathbb{N}$;
\end{itemize}
\end{prop}
\begin{proof}
The statements follow from the definition of F\'ejer monotone sequence (Definition \ref{def_Fejer}).
 \end{proof}

\begin{rem}\label{remark_qfm}
A similar result holds also for quasi-F\'ejer sequences \cite[Proposition 3.3]{combettes2001}, \cite[Proposition 1]{alber1998}. However, in such a case it is not possible to prove that the distance from the target set is decreasing as in Proposition \ref{prop_bau}(iii). 

Formally, let $\mc X\subseteq \RR^n$ be nonempty closed convex and let $(x^k)_{k \in \mathbb{N}}$ be a sequence in $\RR^n$. Suppose that $(x^k)_{k \in \mathbb{N}}$ is quasi-Fejér monotone with respect to $\mc X$. Then, the following statements hold:
\begin{itemize}
\item[(i)] $(x^k)_{k \in \mathbb{N}}$ is bounded;
\item[(ii)] For every $\bar x \in \mc X,(\|x^k-\bar x\|)_{k \in \mathbb{N}}$ converges. 
\end{itemize}
\end{rem}

We note that having convergence of the sequence as in Proposition \ref{prop_bau}(ii) does not necessarily mean that the sequence $(x^k)_{k\in\NN}$ converges to a point in $\mc X$ (see Examples \ref{fm_no_conv}). On the other hand, the latter result can be obtained under slightly stronger assumptions (see also Examples \ref{ex_lemmi}).

\begin{thm}[Theorem 3.8, \cite{combettes2001}]\label{theo_qfm_comb}
Let $\mc X$ be a nonempty set and let $(x^k)_{k \in \NN}$ be a sequence in $\RR^n$. Suppose that $(x^k)_{k \in \NN}$ is quasi-F\'ejer monotone with respect to $\mc X$. Then, $(x^k)_{k \in \NN}$ converges to a point in $\mc X$ if and only if every sequential cluster point of $(x^k)_{k \in \NN}$ belongs to $\mc X$. 
\end{thm}
\begin{proof}
Necessity is straightforward. Sufficiency follows from Remark \ref{remark_qfm} (specifically from \cite[Proposition 3.3]{combettes2001}).
 \end{proof}
\begin{rem}
Since Theorem \ref{theo_qfm_comb} holds for quasi-F\'ejer monotone sequences, it holds also for F\'ejer monotone ones \cite[Theorem 5.5]{bau2011}. In this case, the proof follows by the fact that for every $x\in \mc X$ the sequence $(\|{x^k-x}\|)_{k\in\NN}$ converges by Proposition \ref{prop_bau} and that if every sequential cluster point $x$ belongs to $\mc X$, then the sequence converges to a point in $\mc X$ by Lemma \ref{lemma247}. The result in Theorem \ref{theo_qfm_comb} has been obtained many times in the literature, for weak and strong convergence \cite{gubin1967,alber1998,combettes2000,combettes2001}, but it seems to originate in \cite{schaefer1957}.
 \end{rem}

\begin{rem}
Under suitable conditions, convergence results as in Proposition \ref{prop_bau} and Theorem \ref{theo_qfm_comb} can be obtained also for Bregman monotone sequences \cite[Proposition 4.1 and Theorem 4.11]{bauschke2003}.
 \end{rem}

The following result is known as the Opial Lemma \cite{opial1967} and it can be found in many works and with different applications \cite{malitsky2020siam,malitsky2019,bot2020gan,abbas2014,abbas2015,bau2011,mainge2007}, since it often relate to convergence of sequences generated by nonexpansive operators \cite{opial1967, naraghirad2020,peypouquet2009} (see also Example \ref{ex_fejer_proj}).
We here show a proof which follows from some results in \cite{bau2011} and we report the discrete time formulation \cite{boct2016,csetnek2019,attouch2019}, but it can be found also in continuous time \cite{csetnek2019,bot2016,attouch2019}. For a different proof see \cite{peypouquet2009,opial1967}.

\begin{lem}[Opial Lemma]\label{lemma_opial}
Let $(x^k)_{k\in\NN}$ be a bounded sequence and let $\mc X\subseteq\RR^n$. If
\begin{enumerate}
\item for all $z\in\mc X$ $\lim_{k\to\infty}||x^k-z||$ exists;
\item every sequential cluster point of $(x^k)_{k\in\NN}$ is in $\mc X$ as $k\to\infty$;
\end{enumerate}
then, $(x^k)_{k\in\NN}$ is convergent to a point in $\mc X$.
\end{lem}
\begin{proof}
Since the sequence $(x^k)_{k \in \NN}$ is bounded, it has at least one sequential cluster point. We show that, under this assumption, there cannot be two. The proof follows by contradiction. Suppose that $\bar x$ and $\bar y$ are two sequential cluster points, that is, $x_{k_n}\to \bar x$ and $x_{k_l}\to \bar y$, for $n, l\in\NN$. Since $\bar x$ and $\bar y$ are sequential cluster points, the sequences $(\|x^k-\bar x\|)_{k\in\NN}$ and $(\|x^k-\bar y\|)_{k\in\NN}$ converge. Moreover, it holds that, for all $k\in\NN$
$$2\langle x^k,\bar x-\bar y\rangle=\normsq{x^k-\bar y}-\normsq{x^k-\bar x}+\normsq{\bar x}-\normsq{\bar y}.$$
Therefore, $\langle x^k,\bar x-\bar y\rangle$ converges to some point $w$. Taking the limit along $x_{k_n}$ and $x_{k_l}$ we have
$$w=\langle \bar x,\bar x-\bar y\rangle=\langle \bar y,\bar x-\bar y\rangle.$$
It follows that $\normsq{\bar x-\bar y}=0$ hence $\bar x=\bar y$.  
 \end{proof}

\begin{rem}
The Opial Lemma provides a powerful tool to derive convergence of an iterative process. In fact, condition 2. has been already mentioned in many previous results in this survey. Interestingly, similar results can be extended to the Bregman distance \cite{naraghirad2020,naraghirad2014,huang2011}.
\end{rem}

\subsection{Convergent sequences of real numbers}
%
We now introduce a number of results on sequences of real numbers. 
We note that even if the following results are for general sequences of real numbers, their importance for system theory lies on the fact that they can be paired with (quasi) F\'ejer monotonicity (see Remark \ref{remark_Fejer}). In Table \ref{table_lemmi_det}, we summarize the results presented in this section, with emphasis on the auxiliary sequences that may affect convergence.
\begin{table*}
\begin{center}
\begin{tabular}{ccccccc}
\toprule
& Seq($k+1$)& & Coeff. &Seq($k$) & Negative & Noise\\
& $v^{k+1}$ & $\leq$ & $C^k$ & $v^k$ & $-\theta^k$ & $+\varepsilon^k$\\
\midrule
Lemma \ref{lemma_comb} & NN & & $\gamma$ && \cmark & \cmark\\
Lemma \ref{lemma_det_rs} & NN && $(1+\delta^k)$ && \cmark &\cmark\\
Corollary \ref{cor_mali1} & NN && 1 && \cmark & \xmark\\
Corollary \ref{cor_polyak} & NN && $(1+\delta^k)$ && \xmark &\cmark\\
Lemma \ref{lemma_polyak1} & Real  && $\gamma^k$ && \xmark & \cmark\\
Lemma \ref{lemma_neg_rs} & NN && $(1-\delta^k)$ &&\cmark & \cmark\\
Lemma \ref{lemma_xu03} & NN & & $(1-\delta^k)$ && \xmark & $\delta^k\beta^k$\\
Lemma \ref{lemma_xu02} & NN && $(1-\delta^k)$ && \xmark& $\delta^k\beta^k+\varepsilon^k$\\
Corollary \ref{cor_rem} & NN && $(1-\delta^k)$ && \xmark & $\delta^k(\beta^k+\eta^k)$\\
Corollary \ref{cor_qin} & NN & & $(1-\delta^k)$ && \xmark& \cmark\\
Corollary \ref{cor_liu} &  NN && $(1-\delta^k)$ && \xmark& $\eta^k+\varepsilon^k$\\
Proposition \ref{prop_alber} & NN & & 1 && \xmark & $a\beta^k$\\
Lemma \ref{lemma_he} & NN && $(1-\delta^k)$ && \xmark & $\delta^k \gamma^k+\beta^k$\\
&&&1&&\cmark&\cmark\\
Lemma \ref{lemma_meno} & Real & & $(1+\delta^k)$ && $\delta^kv^{k-1}$ & \cmark\\
Lemma \ref{lemma_rate} & NN && $1/\gamma$&& $\beta^{k+1}/\gamma$ & $\beta^k/\gamma$\\
\bottomrule
\end{tabular}
\caption{Convergence results for deterministic sequences of real numbers divided by their form. In the first line, the most general inequality is presented. NN stands for a sequence of \textit{nonnegative} real numbers, while \cmark (\xmark) indicates if the inequality in the corresponding lemma contains (or not) a term of that column type. $C^k$ is a general  ``coefficient", whose specific form can be retrieved from the column.}\label{table_lemmi_det}
\end{center}
\end{table*}

Let us note that, in the first line of Table \ref{table_lemmi_det}, $C^k$ is a coefficient which, depending on the form, represents the level of expansion or contraction, $\varepsilon^k$ can be seen as an additive noise and $\theta^k$ is a ``negative term", because of the minus sign, which decreases the value of the sequence $v^k$. For a graphical interpretation of the effects of those sequences, we also refer to Figure \ref{fig_lemma} later on, which is specifically related to Lemma \ref{lemma_det_rs}.

The first lemma that we report is widely used and it has a number of consequences that are widely used as well. We do not include the proof since it is very similar to the proof of the forthcoming Lemma \ref{lemma_neg_rs}.

\begin{lem}[Lemma 3.1, \cite{combettes2001}]\label{lemma_comb}
Let $\gamma\in (0,1]$ and let $(v^k)_{k\in\NN}$, $(\theta^k)_{k\in\NN}$ and $(\varepsilon^k)_{k\in\NN}$ be nonnegative sequences such that $\sum_{k=0}^\infty\varepsilon^k<\infty$ and
\begin{equation}\label{eq_combette}
v^{k+1}\leq \gamma v^k-\theta^k+\varepsilon^k \text{ for all } k\in\NN.
\end{equation}
Then, the following statements hold:
\begin{itemize}
    \item[(i)] $(v^k)_{k\in\NN}$ is bounded;
    \item[(ii)] $(v^k)_{k\in\NN}$ converges;
    \item[(iii)] $\sum_{k=0}^\infty\theta^k<\infty$;
    \item[(iv)] If $\gamma\neq1$, then $\sum_{k=0}^\infty v^k<\infty$.
\end{itemize}
\end{lem}

\begin{rem}\label{remark_Fejer}
If $v^k=\norm{x^k-\bar x}$, for some sequence $(x^k)_{k \in \NN}$ and a given $\bar x\in\mc X$, having that $(v^k)_{k\in\NN}$ satisfies the inequality \eqref{eq_combette} implies that $(x^k)_{k \in \NN}$ is a quasi-F\'ejer monotone sequence relative to the set $\mc X$.\\
We also note that the function $V(x^k)=\norm{x^k-\bar x}=v^k$ can also be seen as a decreasing Lyapunov function associated to the sequence that tends toward zero \cite[Section 2.2]{polyak1987}.
 \end{rem}
\begin{rem}\label{remark_kannan}
Ffor a specific choice of the noise term, the following result can be proven \cite[Lemma 3.3]{kannan2012}.
Suppose 
$$v^{k+1} \leq \gamma v^{k}+\eta^{k} \beta,\text{ for all } k\in\NN$$ 
where $\gamma \in(0,1)$, $(\eta^{k})_{k\in\NN}$ is a decreasing positive sequence such that $\sum_{k=0}^{\infty}(\eta^{k})^{2}<\infty,$ and let $0 \leq v^{k} \leq \bar{v}<\infty$ for all $k\in\NN$. Then, $\sum_{k=1}^{\infty} \eta^{k} v^{k}<\infty$.
 \end{rem}

The next lemma is a consequence and a generalization of Lemma \ref{lemma_comb}. It has its stochastic counterpart in the well know Robbins--Siegmund Lemma (Lemma \ref{lemma_rs}) \cite{RS1971}. It is taken from \cite{bau2011} yet here we provide a different proof. For a graphical interpretation, we refer to Figure \ref{fig_lemma}.

\begin{lem}[Lemma 5.31, \cite{bau2011}]\label{lemma_det_rs}
Let $(v^k)_{k\in\NN}$, $(\theta^k)_{k\in\NN}$, $(\varepsilon^k)_{k\in\NN}$ and $(\delta^k)_{k\in\NN}$ be nonnegative sequences such that $\sum_{k=0}^\infty\varepsilon^k<\infty$ and $\sum_{k=0}^\infty\delta^k<\infty$ and
\begin{equation}\label{eq_det_rs}
v^{k+1}\leq (1+\delta^k)v^k-\theta^k+\varepsilon^k, \text{ for all } k \in \NN.
\end{equation}
Then, $\sum_{k=0}^\infty\theta^k<\infty$ and $(v^k)_{k \in \NN}$ is bounded and converges to a nonnegative variable.
\end{lem}
\begin{center}
\begin{figure*}[t]
\centering
\includegraphics[scale=.3]{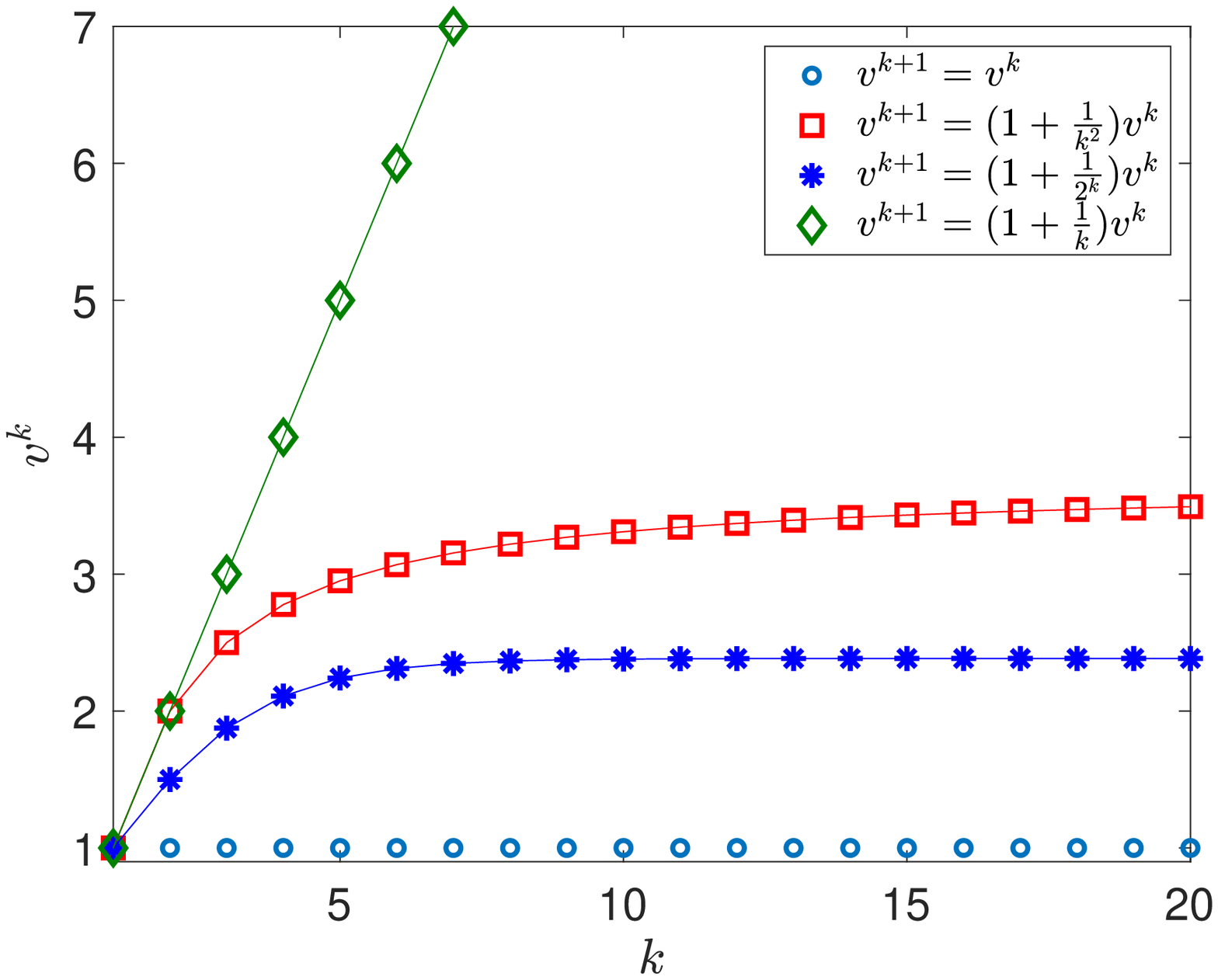}
\includegraphics[scale=.3]{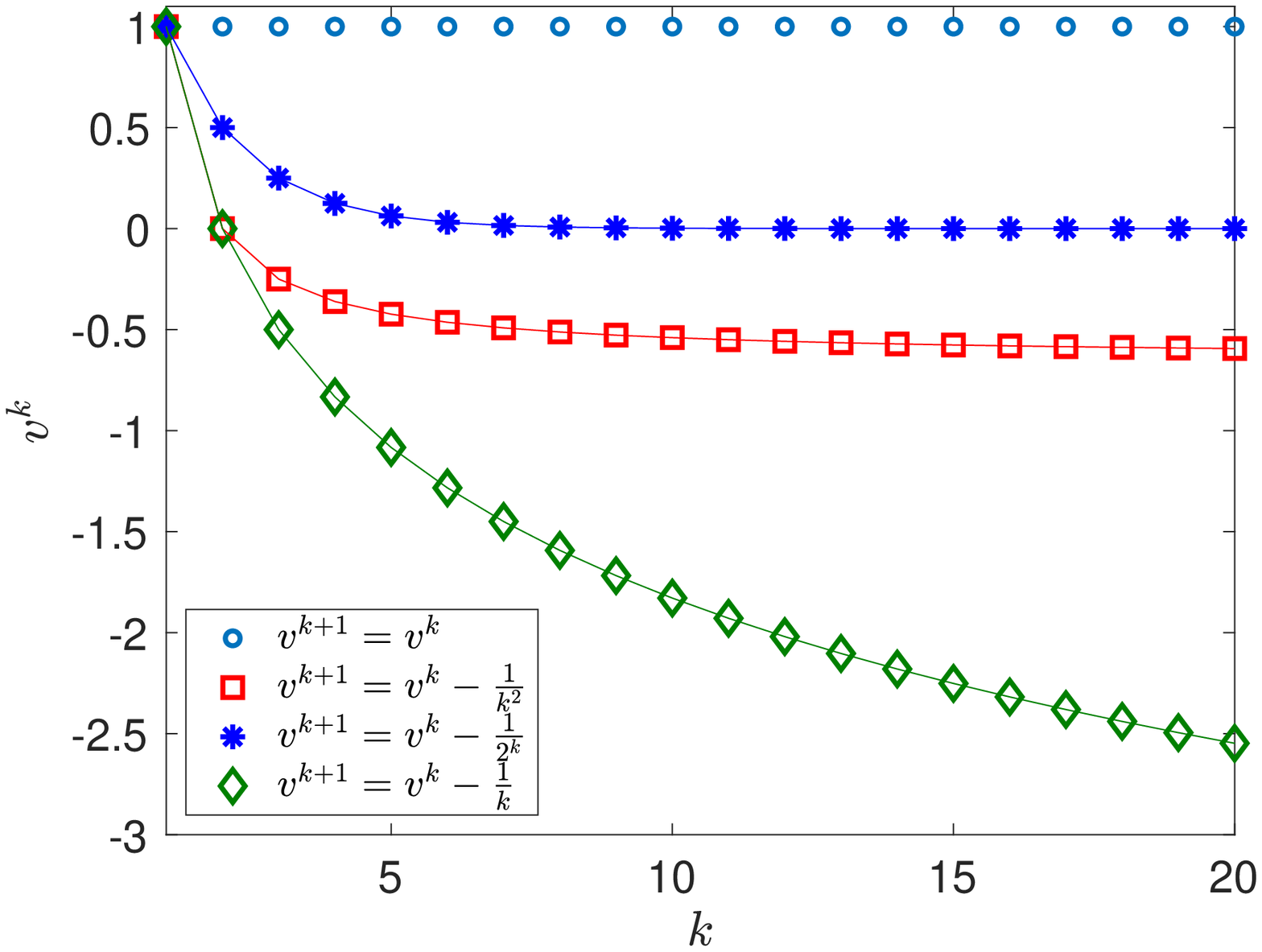}
\includegraphics[scale=.3]{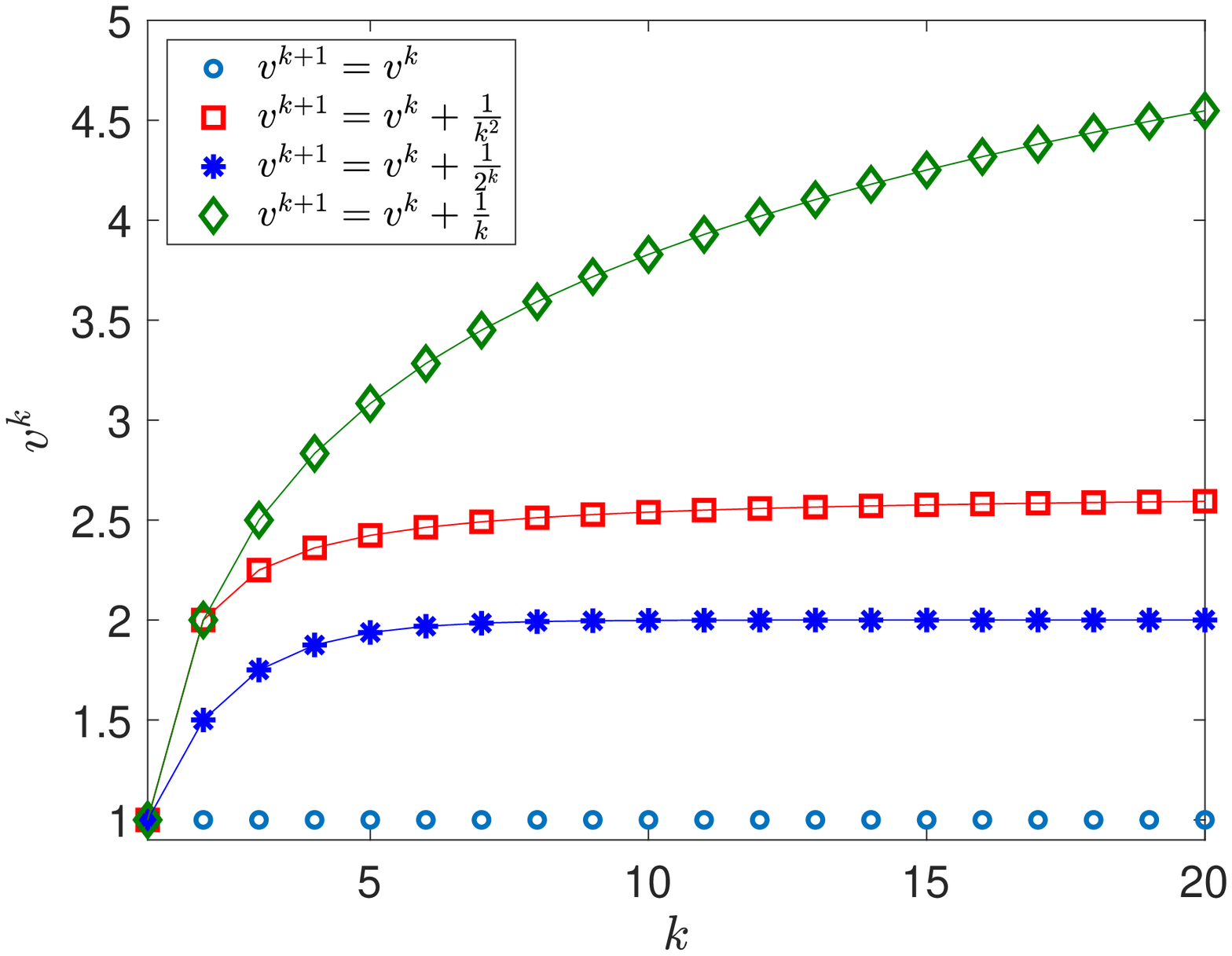}
\includegraphics[scale=.3]{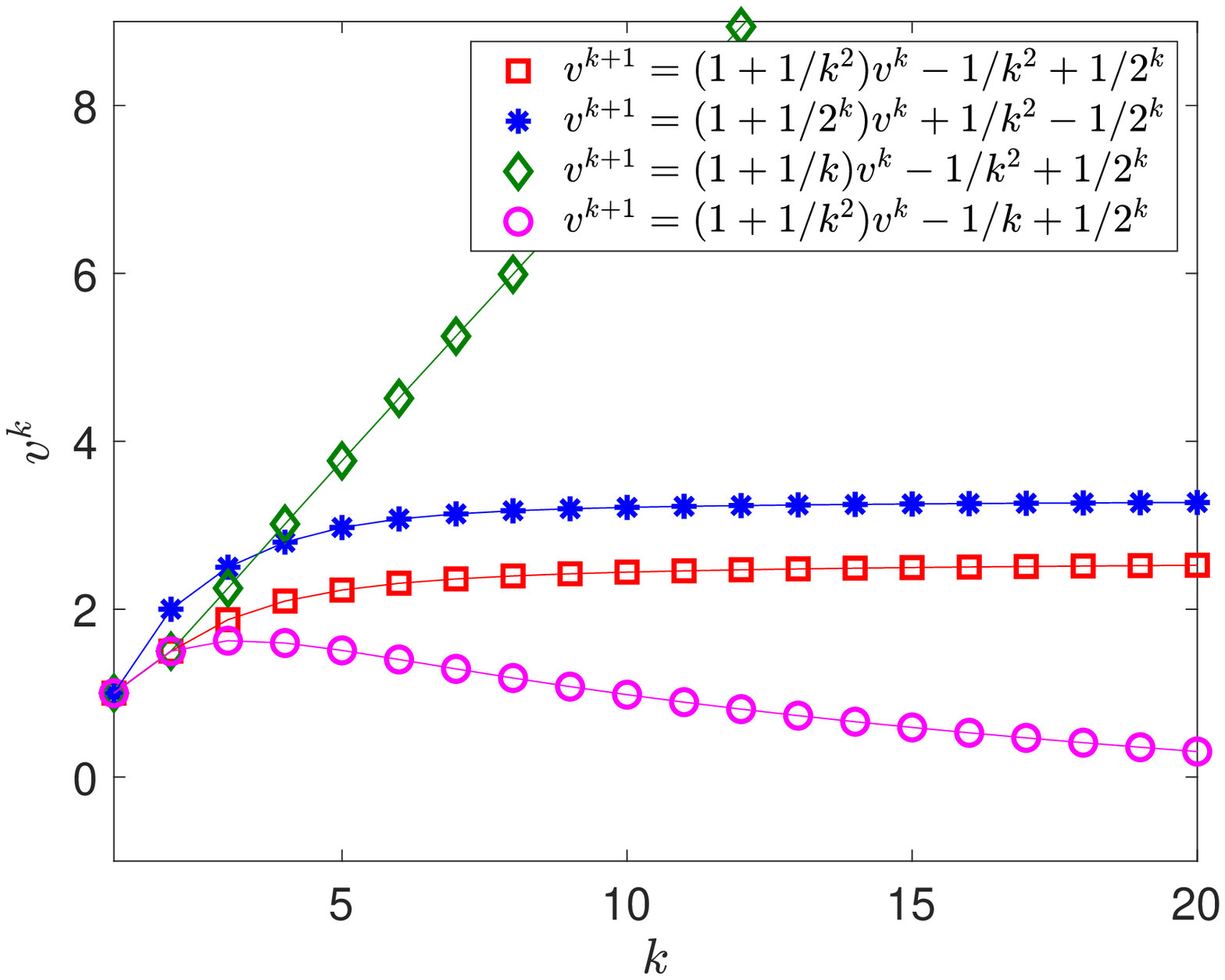}
\caption{Influence of the auxiliary sequences of Lemma \ref{lemma_det_rs} in the behavior of the convergent sequence $(v^k)_{k\in\NN}$. If the sequences are non-summable (green lines in the first three plots, green and magenta in the last), then convergence does not necessarily hold.}\label{fig_lemma}
\end{figure*}
\end{center}
\begin{proof}
Define
$\beta^k=\prod_{i=1}^k(1+\delta^i)$ and note that
$\beta^k$ converges to some $\bar \beta$ since $(\delta^k)_{k\in\NN}$ is summable. Moreover, it holds that
$$1+\delta^k=\frac{\beta^k}{\beta^{k-1}}$$
and, for all $k \in \NN$
$$v^{k+1}\leq\frac{\beta^k}{\beta^{k-1}} v^k+\varepsilon^k-\theta^k.$$
Since $\beta^k> 0$ for all $k \in \NN$, we have
$$\frac{v^{k+1}}{\beta^k}\leq\frac{v^k}{\beta^{k-1}}+\frac{\varepsilon^k}{\beta^k}-\frac{\theta^k}{\beta^k}.$$
Now, let
$$\tilde v^k=\frac{v^k}{\beta^{k-1}},\;\;\tilde \varepsilon^k=\frac{\varepsilon^k}{\beta^k},\;\;\tilde \theta^k=\frac{\theta^k}{\beta^k}$$
and rewrite the inequality as
$$\tilde v^{k+1}\leq \tilde v^k+\tilde \varepsilon^k-\tilde \theta^k.$$
Note that $\tilde v^k$, $\tilde \varepsilon^k$ and $\tilde \theta^k$ are nonnegative and $\sum_{k=1}^\infty\tilde \varepsilon^k \leq\sum_{k=1}^\infty\varepsilon^k<\infty$, hence we can apply Lemma \ref{lemma_comb}. It follows that $\tilde v^k$ is bounded by $\bar\alpha$ and convergent to some $\bar v$ and that $\sum_{k=1}^\infty\tilde \theta^k<\infty$. Therefore $v^k$ is convergent, i.e.,
$$v^k=\frac{v^k}{\beta^{k-1}}\beta^{k-1}=v^k_1\beta_{k_1}\to\bar\alpha\bar\beta \text{ as } k\to\infty,$$
and bounded 
$$\frac{v^k}{\beta^{k-1}}<A \Rightarrow v^k<A\beta^{k-1}\to A\beta_\infty  \text{ as } k\to\infty.$$
Since 
$\theta^k=\tilde\theta^k\beta^k<\tilde \theta^k\beta^\infty$ for all $k\in\NN$,
we conclude that $(\theta^k)_{k\in\NN}$ is summable.  
\end{proof}

We note that there is a slight difference between Lemma \ref{lemma_comb} and Lemma \ref{lemma_det_rs}. Specifically, in the former, the sequence converges if the coefficient $\gamma$ is in the interval $(0,1]$ while in Lemma \ref{lemma_det_rs} the coefficient can be taken larger than 1 and time varying.

The following results are immediate consequences of Lemmas \ref{lemma_comb} and \ref{lemma_det_rs}. Let us start with removing the noise term.

\begin{cor}[Lemma 2.8, \cite{malitsky2015}]\label{cor_mali1} 
Let $(v^k)_{k\in\NN}$ and $(\theta^k)_{k\in\NN}$ be nonnegative sequences such that 
$$v^{k+1}\leq v^k-\theta^k, \text{ for all } k\in\NN.$$
Then, $(v^k)_{k\in\NN}$ is bounded and $\lim\limits_{k\to\infty}\theta^k=0$.
\end{cor}

\begin{proof}
It follows from Lemmas \ref{lemma_comb} and \ref{lemma_det_rs} by taking $\delta^k$ and $\varepsilon^k$ equal to 0. 
 \end{proof}

Similarly, this result from \cite{polyak1987} can be obtained as a consequence of Lemma \ref{lemma_det_rs} by removing the negative term.
\begin{cor}[Lemma 2.2.2, \cite{polyak1987}]\label{cor_polyak}
Let $(v^k)_{k\in\NN}$, $(\varepsilon^k)_{k\in\NN}$ and $(\delta^k)_{k\in\NN}$ be nonnegative sequences such that
$$v^{k+1} \leq\left(1+\delta^k\right) v^k+\varepsilon^k, \text{ for all } k\in\NN$$
and $\sum_{k=0}^{\infty} \delta^k<\infty$, $\sum_{k=0}^{\infty} \varepsilon^k<\infty$. Then $v^k$ converges to some $\bar v\geq 0$.
\end{cor}
\begin{proof}
It follows from Lemma \ref{lemma_det_rs} by taking $\theta^k=0$ for all $k\in\NN$. See \cite{polyak1987} for a different proof.
 \end{proof}

%
%
%

Concerning the coefficient sequence, other options can be considered.
In the next result, the coefficient should be strictly smaller than 1, compared to Lemma \ref{lemma_det_rs}, but need not be constant as in Lemma \ref{lemma_comb}.
\begin{lem}[Lemma 2.2.3, \cite{polyak1987}]\label{lemma_polyak1}
Let $(v^k)_{k\in\NN}$ be a sequence of real numbers such that
$$v^{k+1} \leq \gamma_{k} v^{k}+\varepsilon^k,\text{ for all } k\in\NN,$$
where $(\varepsilon^k)_{k\in\NN}$ and $(\gamma_k)_{k\in\NN}$ are nonnegative sequences such that
\begin{itemize}
\item[1.] $0 \leq \gamma_{k}<1$ 
\item[2.] $\sum_{k=0}^{\infty}\left(1-\gamma_{k}\right)=\infty,$ 
\item[3.] $\lim_{k\to\infty}\frac{\varepsilon^k}{1-\gamma_{k}} = 0$.
\end{itemize}
Then, $lim_{k\to\infty} v^k=\bar v\leq 0$. Moreover, if $v^k>0$ then $\lim_{k\to\infty}v^k=0$.
\end{lem}
\begin{proof}
By definition, given $\varepsilon>0$ there exists $k_0\in\NN$ such that
$$\frac{\varepsilon^k}{(1-\gamma_k)}\leq \varepsilon \quad \text{ for all } k\geq k_0.$$
Then
$$\begin{aligned}
v^{k+1}&\leq \gamma_kv^k+\varepsilon^k\\
&\leq \gamma_kv^k+(1-\gamma_k)\varepsilon\\
&\leq \gamma_k\gamma_{k-1}v^{k-1}+\left[\gamma_k(1-\gamma_{k-1})+(1-\gamma_k)\right]\varepsilon\\
&\leq\prod_{i=k_0}^k\gamma_iv^{k_0}+\varepsilon\left(1-\prod_{j=k_0}^k\gamma_i\right).
\end{aligned}$$
Note that  $\sum_{k=0}^{\infty}\left(1-\gamma_{k}\right)=\infty$ implies that $\prod_{k=0}^\infty \gamma_k=0$ therefore taking the $\limsup$ as $k\to\infty$ leads to $\limsup v^k\leq\varepsilon$, which proves the claim.  
 \end{proof}

\begin{rem}\label{remark_kannan}
In \cite[Lemma 2.1]{kannan2012}, the result in Lemma \ref{lemma_polyak1} is proven also for a different condition than 1., i.e., 
\begin{itemize}
\item[1.*] there exists $\bar k\in\NN$ such that $0<\gamma_k<1$ for all $k\geq\bar k$ and $q_k<\infty$ for all $k\leq\bar k$.
\end{itemize}
The proof follows considering the shifted process starting from $\bar k$ and using Lemma \ref{lemma_polyak1} on the resulting sequence.
%
 \end{rem}

Many of the previous results have the coefficient $(1+\delta^k)$, therefore, we now consider what happens if we change it to $(1-\delta^k)$ (see also Figure \ref{fig_neg} for a graphical interpretation).
This might be a special case of Lemma \ref{lemma_comb} but, in some cases, it allows to study convergence to zero (see Remark \ref{remark_tozero}), which relates to the standard Lyapunov based approach for stability analysis. In fact, we have already had a glimpse of the effect of a coefficient smaller than 1 in Lemma \ref{lemma_comb}(iv) and Lemma \ref{lemma_polyak1} and its connection with Lyapunov analysis (Remark \ref{remark_Fejer}).

The first result of this type extends the previous lemmas to this case. This result is new as we provide a proof that does not follow from previous results.

\begin{lem}\label{lemma_neg_rs}
Let $(v^k)_{k\in\NN}$, $(\theta^k)_{k\in\NN}$, $(\varepsilon^k)_{k\in\NN}$ and $(\delta^k)_{k\in\NN}$ be nonnegative sequences such that $\sum_{k=0}^\infty\varepsilon^k<\infty$, $\sum_{k=0}^\infty\delta^k<\infty$, $\delta^k\in[0,1)$ for all $k\in\NN$ and
$$v^{k+1}\leq (1-\delta^k)v^k+\varepsilon^k-\theta^k \text{ for all }k \in \NN.$$
Then, $(v^k)_{k \in \NN}$ is bounded and converges to some $\bar v\geq 0$ and $\sum_{k=0}^\infty\theta^k<\infty$.
\end{lem}
\begin{proof}
To prove that $(v^k)_{k\in\NN}$ is bounded, let $\varepsilon=\sum_{k=0}^\infty\varepsilon^k$. Then,
$$0\leq v^{k+1}\leq\prod_{i=0}^k(1-\delta^i)v^0+\sum_{i=0}^k\prod_{j=0}^{k-i}(1-\delta_j)\varepsilon_i\leq v^0+\varepsilon.$$
Therefore $v^k\in[0,v^0+\varepsilon]$ and the first claim is proven. Now we prove convergence. Let $\bar v= \liminf v^k\in[0,v^0+\varepsilon]$. Then there exists a subsequence $v^{k_n}$ such that $\lim v^{k_n}=\bar v$. Then, for every $\eta>0$ there exists $n_0\in\NN$ such that $v^{k_{n_0}}\leq\bar v+\eta/2$. Since $\sum_{k=0}^\infty\varepsilon^k<\infty$, there exists $n_1$ such that $\sum_{m>n_1}\varepsilon^m\leq\eta/2.$ Set $n=\max\{n_0,n_1\}$, then, iterating, for every $k\geq k_n$
$$v^k\leq v^{k_n}+\sum_{m\geq n}\varepsilon^m\leq\frac{\eta}{2}+\bar v+\frac{\eta}{2}=\bar v+\eta.$$
Hence, $\limsup v^k\leq\liminf v^k+\eta$ and, since $\eta$ can be arbitrarily small, $(v^k)_{k\in\NN}$ converges to $\bar v$.
Lastly, we show that $(\theta^k)_{k\in\NN}$ is summable. Since
$$\theta^k\leq(1-\delta^k)v^0-v^{k+1}+\varepsilon^k,$$
we can do a telescopic sum to obtain
$$\sum_{k=1}^\infty\theta^k\leq(1-\delta_0)v^0-v^{k+1}+\sum_{k=1}^\infty\varepsilon^k\leq v^0+\varepsilon.  $$
\end{proof}
\begin{center}
\begin{figure*}[t]
\centering
\includegraphics[scale=.3]{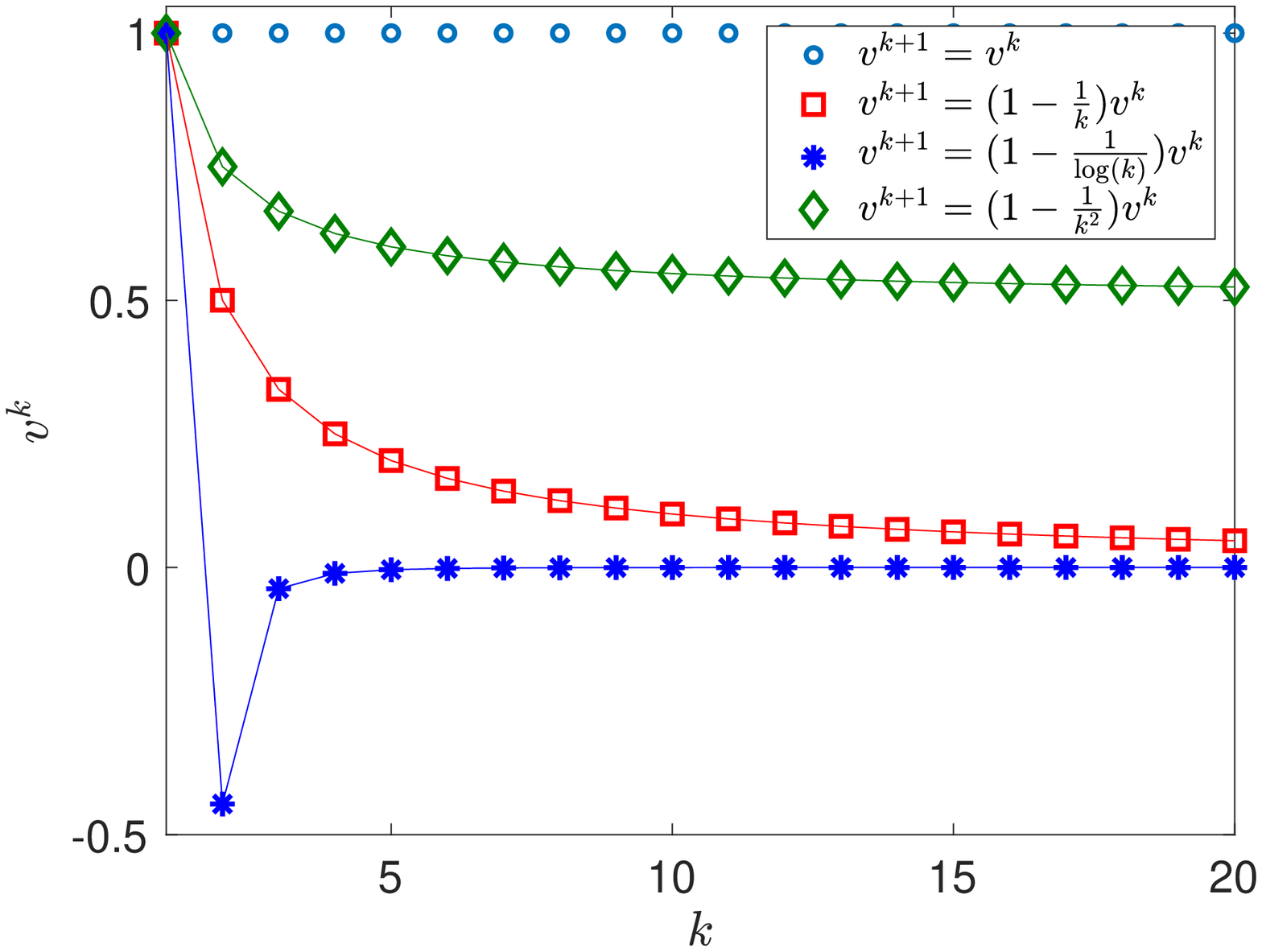}
\includegraphics[scale=.3]{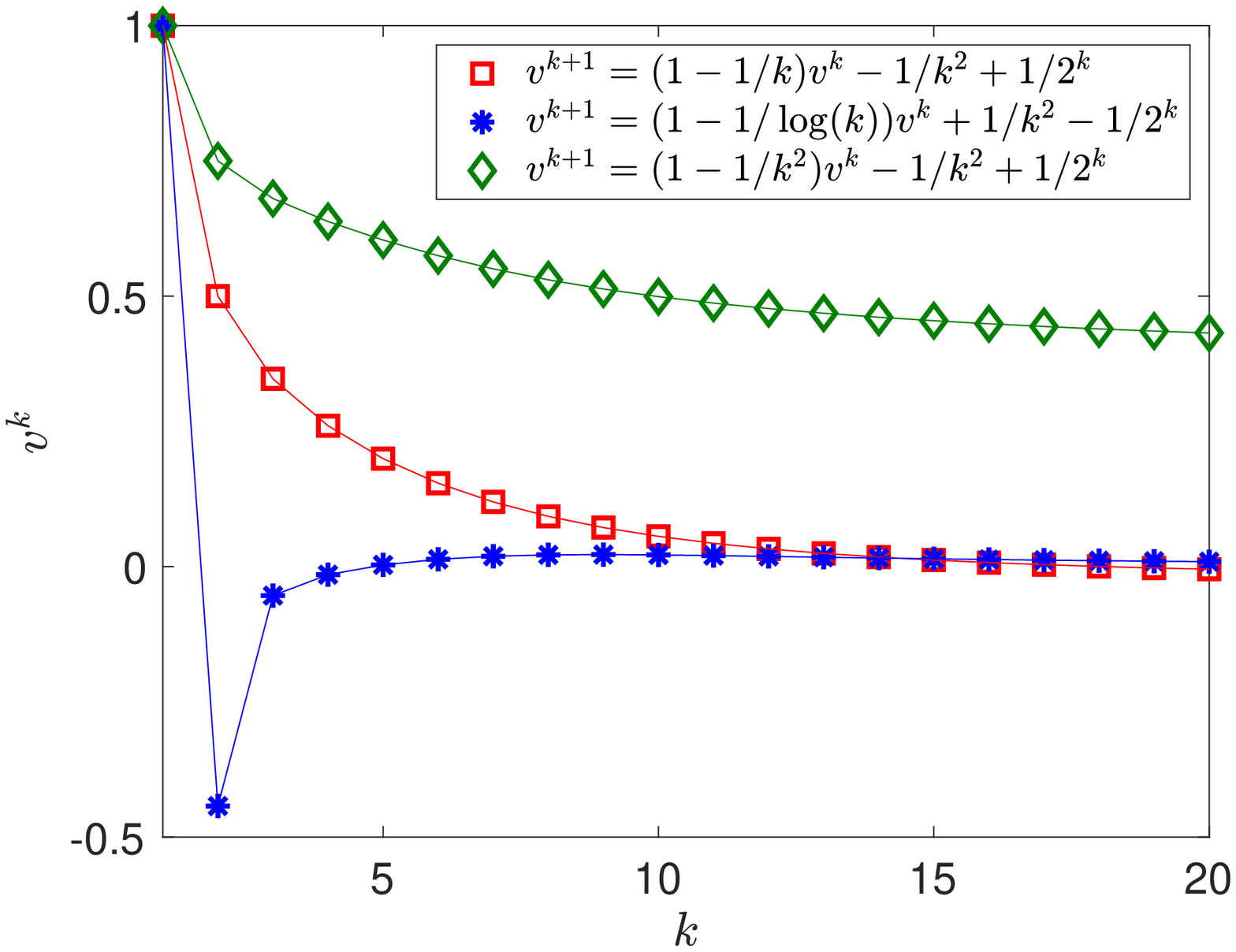}
\caption{Influence of the the coefficient $(1-\delta^k)$ in the behavior of the convergence of a sequence $v^{k+1}=(1-\delta^k)v^k-\theta^k+\varepsilon^k$. We note that if the sequence $\delta^k$ is summable (green lines in the plots), the convergence to zero is not guaranteed.}\label{fig_neg}
\end{figure*}
\end{center}

The following lemmas are taken from various works \cite{qin2008,xu2002,xu2003,xu1998} and they are quite similar to each other. We here establish the relations and difference between them. 
Let us remark that in the following results, the sequence $\delta^k$ in the coefficient is not summable, i.e., from now on $\sum_{k=1}^\infty\delta^k=\infty$. The advantage of this choice is that convergence to zero can be obtained, as shown in Figure \ref{fig_neg}.

The first result considers real (not only positive) noise sequences $\varepsilon^k$ and it also provides two alternative conditions on the auxiliary sequences.
\begin{lem}[Lemma 2.1, \cite{xu2003}]\label{lemma_xu03}
Let $(v^k)_{k\in\NN}$ be a sequence of nonnegative real numbers such that
$$v^{k+1} \leq\left(1-\delta^k\right) v^k+\delta^k \beta^k,\text{ for all } k\in\NN$$
where $(\delta^k)_{k\in\NN}$ and $(\beta^k)_{k\in\NN}$ are sequences of real numbers such that:
\begin{itemize}
\item[1.] $\delta^k\in [0,1]$ and $\sum_{k=0}^{\infty} \delta_{k}=\infty$, or equivalently, $\prod_{k=0}^{\infty}\left(1-\delta^k\right)=0$,
\item[2a.] $\lim \sup _{k \to \infty} \beta^k \leq 0$,
\item[2b.] $\sum_{k=0}^\infty \delta^k \beta^k<\infty$.
\end{itemize}
Then, $\lim_{k\to\infty} v^k=0$. 
\end{lem}
\begin{proof}
If $1.$ and $2a.$ hold, then the result can be proven with the same arguments as the proof of Lemma \ref{lemma_polyak1} by setting $\varepsilon^k=\delta^k\beta^k$. On the other hand, if $1.$ and $2b.$ hold, we have for all $k>m$
$$\begin{aligned}
v^{k+1}&\leq (1-\delta^k)v^k+\delta^k \beta^k\\
&\leq\prod_{i=m}^k(1-\delta^i)v^{m}+\sum_{i=m}^k\delta^i\beta^i.
\end{aligned}$$
Taking the limit for $k\to\infty$ and $m\to\infty$ we have $\limsup v^k\leq 0$.  
 \end{proof}
We note that if we set $\gamma_k=1-\delta^k$ and $\varepsilon^k=\delta^k\beta^k$, we obtain the same statement as Lemma \ref{lemma_polyak1}. Moreover, condition $2b.$ provides an alternative assumption, similar to most results in the literature. 

\begin{rem}\label{remark_tozero}
Convergence to zero is of particular interest in combination with a F\'ejer-like property. Specifically, if $v^k=\norm{x^k-\bar x}$ for some sequence $(x^k)_{k\in\NN}$ and $\bar x\in\mc X$, Lemma \ref{lemma_xu03} states that $\lim_{k\in\NN}\norm{x^k-\bar x}=0$, i.e., $\lim_{k\to\infty}x^k= \bar x$. We note, however, that the sequence is not quasi-F\'ejer monotone because of the term $\delta^k$ which cannot be 0 (contrary to Remark \ref{remark_Fejer}). We refer to Section \ref{sec_app_det} for more details. 
 \end{rem}

Assumption $2a.$ in Lemma \ref{lemma_xu03} is used in a previous paper by the same authors \cite[Lemma 2.5]{xu2002}. Since also Assumption $2b.$ in Lemma \ref{lemma_xu03} can be used to prove the following result, let us extend \cite[Lemma 2.5]{xu2002} next.

\begin{lem}[Extension of Lemma 2.5, \cite{xu2002}]\label{lemma_xu02}
Let $(v^k)_{k\in\NN}$ be a sequence of nonnegative real numbers satisfying
$$v^{k+1} \leq\left(1-\delta^k\right) v^k+\delta^k \beta^k+\varepsilon^k,\text{ for all } k\in\NN$$
where $(\delta^k)_{k\in\NN}$, $(\beta^k)_{k\in\NN}$ and $(\varepsilon^k)_{k\in\NN}$ satisfy the following conditions: 
\begin{enumerate}
\item[1.] $\delta^k\in[0,1]$, $\sum_{k=0}^\infty \delta^k=\infty$, or equivalently, $\prod_{k=1}^{\infty}(1-\delta^k)=0$,
\item[2a.] $\lim \sup _{k \to \infty} \beta^k \leq 0$,
\item[2b.] $\sum_{k=0}^\infty \delta^k \beta^k<\infty$,
\item[3.] $\varepsilon^k \geq 0$ and $\sum_{k=0}^\infty \varepsilon^k<\infty$.
\end{enumerate}
Then, $\lim_{k\to\infty} v^k=0.$
\end{lem}
\begin{proof}
The proof is similar to the proof of Lemma \ref{lemma_xu03}. Since $\varepsilon^k$ is summable, $\sum_{k=k_0}^\infty\varepsilon^k<\varepsilon$ for some $k_0\in\NN$ and $\varepsilon>0$ arbitrarily small, and
$$\begin{aligned}
v^{k+1}&\leq \prod_{i=k_0}^k(1-\delta^i)v^{k_0}+\varepsilon\left(1-\prod_{i=k_0}^k(1-\delta^i)\right)+\sum_{i=k_0}^k\varepsilon^i\\
&\leq2\varepsilon. 
\end{aligned}$$
\end{proof}

A particular case of Lemma \ref{lemma_xu02} is proposed in \cite{lei2020cdc} as a consequence of \cite[Theorem 3.3.1]{chen2006}. Let us note that the assumptions in the following result imply those in Lemma \ref{lemma_xu02} which is, in turn, more general.
\begin{cor}[Proposition 3, \cite{lei2020cdc}]\label{cor_rem}
Let $(v^k)_{k\in\NN}$ be a nonnegative sequence such that
$$v^{k+1}\leq(1-\delta_k)v^k+\delta^k(\beta^k+\eta^k)$$
where $(\delta^k)_{k\in\NN}$, $(\beta^k)_{k\in\NN}$, $(\eta^k)_{k\in\NN}$ are nonnegative sequences such that
\begin{itemize}
\item[1.] $\sum_{k=1}^\infty\delta^k=\infty$ and $\lim_{k\to\infty}\delta^k=0$
\item[2.] $\lim_{k\to\infty}\beta^k=0$
\item[3.] $\sum_{k=1}^\infty\delta^k\eta^k<\infty$
\end{itemize}
Then, $\lim_{k\to\infty} v^k=0$.
\end{cor}
\begin{proof}
The sequence satisfies the assumptions of Lemma \ref{lemma_xu02}, hence the result holds.
\end{proof}

A very recent result of this type is a consequence of both Lemma \ref{lemma_polyak1} and Lemma \ref{lemma_xu03}.
\begin{cor}[Lemma 1.1, \cite{qin2008}]\label{cor_qin}
Assume that $(v^k)_{k\in\NN}$ is a sequence of nonnegative real numbers such that
$$v^{k+1}\leq\left(1-\delta^k\right) v^k+\varepsilon^k,\text{ for all } k\in\NN$$
where $(\delta^k)_{k\in\NN}$ and $(\varepsilon^k)_{k\in\NN}$ are sequences such that
\begin{enumerate}
\item[1.] $\delta^k\in(0, 1)$ and $\sum_{k=1}^\infty \delta^k=\infty$,
\item[2a.] $\lim \sup _{k \to \infty} \frac{\varepsilon^k}{\delta^k} \leq 0$ ,
\item[2b.] $\sum_{k=1}^{\infty}|\varepsilon^k|<\infty$.
\end{enumerate}
Then, $\lim_{k \to \infty} v^k=0$.
\end{cor}
\begin{proof}
Suppose that $2a.$ holds. Then, the proof follows from Lemma \ref{lemma_polyak1} by setting $\gamma_k=(1-\delta^k)$. When $2b.$ holds instead, the proof follows applying Lemma \ref{lemma_xu03} by defining $\varepsilon^k=\delta^k\frac{\varepsilon^k}{\delta^k}=\delta^k\beta^k$.
 \end{proof}

A consequence of Corollary \ref{cor_qin} is the following result that presents a slightly different notation.
\begin{cor}[Lemma 3, \cite{xu1998}]\label{cor_liu}
Let $(v^k)_{k\in\NN}$, $(\eta^k)_{k\in\NN}$, $(\varepsilon_k)_{k\in\NN}$ and $(\delta^k)_{k\in\NN}$ be nonnegative real sequences such that 
$$v^{k+1} \leq(1-\delta^k) v^k+\eta^k+\varepsilon^k \text{ for all } k\in\NN$$
and such that
\begin{enumerate}
\item[1.] $\delta^k \in[0,1]$ and $\sum_{k=0}^{\infty} \delta^k=\infty$,
\item[2.] $\varepsilon^k=o(\delta^k),$
\item[3.] $\sum_{k=0}^{\infty} \eta^k<\infty .$ 
\end{enumerate}
Then, $\lim _{k \rightarrow \infty} v^k=0 .$
\end{cor}
\begin{proof}
We note that $\varepsilon^k=o(\delta^k)$ is equivalent to $\lim_{k\to\infty}\varepsilon^k/\delta^k=0$. Then, the result follows by applying Corollary \ref{cor_qin} and Lemma \ref{lemma_xu02}. 
 \end{proof}

\begin{rem}
A result similar to the last corollaries where the boundedness of the sequence is shown, is presented also in \cite[Lemma 2.5]{cholamjiak2018} and reads as follows.\\
Let $(v^k)_{k\in\NN}$ and $(\eta^k)_{k\in\NN}$ be sequences of nonnegative real numbers such that $\sum_{k=1}^\infty \eta^k<\infty$ and such that
$$
v^{k+1} \leq(1-\delta^k) v^k+\eta^k+\varepsilon^k, \text{ for all }k\in\NN
$$
where $(\delta^k)_{k\in\NN}\subseteq(0,1)$ and $(\varepsilon^k)_{k\in\NN}$ is a sequence of real numbers. Then, the following results hold:
\begin{itemize}
\item[(i)] If $\varepsilon^k \leq \delta^k M$ for some $M \geq 0$, then $(v^k)$ is a bounded sequence.
\item[(ii)] If $\sum_{k=1}^\infty \delta^k=\infty$ and $\lim \sup _{n \rightarrow \infty} \frac{\varepsilon^k}{\delta^k} \leq 0$, then $\lim _{n \rightarrow \infty} v^k=0$.
\end{itemize}
We note that (ii) is a consequence of Lemma \ref{lemma_xu02} or Corollaries \ref{cor_qin} and \ref{cor_liu}.
\end{rem}

We now consider three results whose conditions for convergence are more involved than the results proposed until now \cite{alber1998,cholamjiak2018,he2013,mainge2008}. 
The first one is proposed in \cite{alber1998}. It allows for non-summable additive noise but requires a condition that couple the sequences involved.
\begin{prop}[Proposition 2, \cite{alber1998}]\label{prop_alber}
Let $(v^k)_{k\in\NN}$ and $(\beta^{k})_{k\in\NN}$ be two nonnegative sequences such that $\sum_{k=0}^{\infty} \beta_{k}=\infty$ and $\sum_{k=0}^{\infty} \beta_{k} v^{k}<\infty$. Then:
\begin{itemize}
\item[(i)] there exists a subsequence $(v^{k_n})$, $n\in\NN$ such that $\lim _{k \rightarrow \infty} v^{k_n}=0$.
\item[(ii)] Moreover, if there exists $a>0$ such that
$$v^{k+1}\leq v^{k} + a \beta^{k}, \text{  for all }k\in\NN$$ 
then $\lim _{k \rightarrow \infty} v^{k}=0.$
\end{itemize}
\end{prop}
\begin{proof}
Both claims can be proven by contradiction. See \cite{alber1998} for more details.
 \end{proof}

In the next result, the sequence should satisfy two interdependent inequalities \cite{cholamjiak2018,he2013}.
\begin{lem}[Lemma 7, \cite{he2013}]\label{lemma_he}
Let $(v^k)_{k\in\NN}$ and $(\eta^k)_{k\in\NN}$ be nonnegative sequences of real numbers, let $(\delta^k)_{k\in\NN}\subseteq(0,1)$ and let $(\gamma^k)_{k\in\NN}$, $(\varepsilon^k)_{k\in\NN}$, and $(\beta^k)_{k\in\NN}$ be three sequences of real numbers such that
\begin{equation}\label{eq_he}
\begin{aligned}
v^{k+1} \leq &(1-\delta^k) v^k+\delta^k \gamma^k+\beta^k,  \text{ and}, \\
v^{k+1} \leq &v^k-\eta^k+\varepsilon^k, \text{ for all } k\in\NN 
\end{aligned}
\end{equation}
and such that
\begin{itemize}
\item[1.] $\sum_{k=0}^{\infty} \delta^k=\infty$
\item[2.] $\lim _{k \rightarrow \infty} \varepsilon^k=0$
\item[3.] $\lim _{k \rightarrow \infty} \eta^{k_n}=0$ implies that $\lim \sup _{k \rightarrow \infty} \gamma^{k_{n}} \leq 0$ for any
subsequence $(k_{n}) \subset(k)$
\item[4.] $\lim \sup _{n \rightarrow \infty}\frac{\beta^k}{\delta^k} \leq 0$.
\end{itemize}
Then, $\lim _{n \rightarrow \infty} v^k=0$.
\end{lem}
\begin{proof}
The proof is divided in two cases.\\
Case 1: $(v^k)_{k\in\NN}$ is eventually decreasing and the result follows from Lemma \ref{lemma_xu02}.\\
Case 2: $(v^k)_{k\in\NN}$ is not eventually decreasing. Then, there exists $k_0$ such that $v^{k_0}\leq v^{k_0+1}$. Let $J_{k}=\{k_{0} \leq n \leq k: v_{n} \leq v_{n+1}\}$, $k>k_{0}$ and let $\tau(k)=\max J_k$. Then, $\tau(k)\to\infty$ as $k\to\infty$ by definition. It follows that $v^k\leq v^{\tau(k)+1}$. Then, using \eqref{eq_he} and the assumptions it follows that 
$$v^{\tau(k)}\leq\gamma^{\tau(k)}+\frac{\beta^{\tau(k)}}{\delta^{\tau(k)}},$$
and $\op{limsup}_{k\to\infty} v^{\tau(k)}\leq0$. Hence, $\lim_{k\to\infty}v^k= 0$. For more details we refer to \cite{he2013}.
\end{proof}
\begin{rem}
By removing $\beta^k$ in Equation \eqref{eq_he}, the result holds as a particular case of Lemma \ref{lemma_he} and can be proven similarly, using Lemma \ref{lemma_xu03} instead of Lemma \ref{lemma_xu02}.
\end{rem}
The next result, instead, uses the sequence at two steps backwards.
\begin{lem}[Lemma 2.2, \cite{mainge2008}]\label{lemma_meno}
Let $(v^k)_{k\in\NN}$ and $(\varepsilon^k)_{k\in\NN}$ be nonnegative sequences such that:
\begin{equation}\label{eq_meno}
v^{k+1}-v^k \leqslant \delta^k(v^k-v^{k-1})+\varepsilon^k,\text{ for all } k\in\NN
\end{equation}
and such that
\begin{itemize}
\item[1.] $\sum_{k=1}^\infty \varepsilon^k<\infty$
\item[2.] $(\delta^k) \subset[0, \delta]$, where $\delta \in[0,1)$.
\end{itemize}
Then $(v^k)_{k\in\NN}$ converges and $\sum_{k=1}^\infty[v^{k+1}-v^k]_{+}<\infty$, where $[t]_{+}:=\max \{t, 0\}$ (for any $t \in \mathbb{R}$).
\end{lem}
\begin{proof}
Let $u^k=v^k-v^{k-1}$. Then, $[u^{k+1}]_+\leq\delta[u^k]_++\varepsilon^k$ and $([u^k]_+)_{k\in\NN}$ is bounded. It follows that the sequence $(w^k=v^k-\sum_{j=1}^k[u_j]_+)_{k\in\NN}$ is bounded and non increasing, hence, convergent. Hence, $(v^k)_{k\in\NN}$ is convergent.
\end{proof}

\begin{rem}
The result can be extended to the case with a negative term, namely, Equation \eqref{eq_meno} becomes
$$v^{k+1}-v^k \leqslant \delta^k(v^k-v^{k-1})+\varepsilon^k-\theta^k,\text{ for all } k\in\NN$$
where $\theta^k$ is a nonnegative sequence. The conclusions are the same as Lemma \ref{lemma_meno} and, moreover, it holds that $\sum_{k=1}^\infty \theta^k<\infty$ \cite[Lemma 2]{boct2016}.
\end{rem}

\begin{rem}
The inequality in Equation \eqref{eq_meno} can be rewritten as
$$v^{k+1}\leq(1+\delta^k)v^k-\delta^kv^{k-1}+\varepsilon^k,\text{ for all } k\in\NN$$
which is similar to the form of the results presented until now. However, we note that in Lemma \ref{lemma_meno}, the sequence $v^k$ need not be nonnegative.
\end{rem}

We conclude this section with the following result on the convergence rate which guarantees convergence to zero.
However, the study of the convergence rates lays outside the scopes of this survey. For similar results, we refer to \cite[Lemma 2.9]{malitsky2015}, \cite[Lemma 3]{lei2018} and, more generally, to \cite{polyak1987}. 

\begin{lem}[Lemma 2.7, \cite{malitsky2018fbf}] \label{lemma_rate}
Let $(v^{k})_{k \in \mathbb{N}}$ and $(\beta^{k})_{k \in \mathbb{N}}$ be two nonnegative sequences of real numbers. Suppose there exist constants $\gamma>1$ and $\delta>0$ such that
$$
\gamma v^{k+1}+\beta^{k+1} \leq v^{k}+\beta^{k} \quad \text { and } \quad \delta \beta^{k} \leq v^{k} \text{ for all } k \in \mathbb{N}
$$
Then $(v^{k})_{k \in \mathbb{N}}$ and $(\beta^{k})_{k \in \mathbb{N}}$ converge to zero with $R$-linear rate.
\end{lem}

\begin{table*}[t]
\begin{center}
\begin{tabular}{cccccc}
\toprule
& Seq($k+1$)& & Coeff. Seq($k$) & Negative & Noise\\
& $\EE[v^{k+1}]$ & $\leq$ & $C^kv^k$ & $-\theta^k$ & $+\varepsilon^k$\\
\midrule
Lemma \ref{lemma_rs} & NN & & $(1+\delta^k)$ & \cmark & \cmark\\
Corollary \ref{lemma_glad} & NN && $(1+\delta^k)$ & \xmark &\cmark\\
Corollary \ref{cor_poggio} & NN && 1 & \cmark & \xmark\\
Corollary \ref{cor_duflo} & NN && $1$ & \cmark &\cmark\\
Lemma \ref{lemma_rs_comb} & NN & & $\gamma^k$ & \cmark & \cmark\\
Lemma \ref{lemma_fake_rs} & NN && $(1-\delta^k)$ & \xmark & \cmark\\
\bottomrule
\end{tabular}
\caption{Convergence results for stochastic sequences of real numbers divided by their form. In the first line, the most general inequality is presented. NN stands for a sequence of \textit{nonnegative} real numbers, while \cmark (\xmark) indicates if the inequality in the corresponding lemma contains (or not) a term of that column type. $C^k$ is a general  ``coefficient", whose specific form can be retrieved from the column.}\label{table_lemmi_stoc}
\end{center}
\end{table*}

\section{Convergence of stochastic sequences}\label{sec_stoc}

In this section, we report the convergence results available for sequences of random variables, summarized in Table \ref{table_lemmi_stoc}.
We recall that the probability space is $(\Omega, \mc F,\PP)$ where $\Omega$ is the sample space, $\mc F$ is the event space, and $\PP$ is the probability function defined on the event space. The symbol $\EE$ indicates the associated expected values. We also recall that $\mc F=(\mc F_k)_{k\in\NN}$ is a filtration.

\subsection{Convergent sequences of random variables}

Firstly, we recall some results on convergent random sequences. We start with a result by Robbins and Siegmund, first appeared in \cite{RS1971}, which is the most used in the stochastic literature. In Figure \ref{fig_lemma_rs}, we provide a graphical interpretation.

\begin{lem}[Robbins--Siegmund Lemma]\label{lemma_rs}
Let $(v^k)_{k\in\NN}$, $(\theta^k)_{k\in\NN}$, $(\varepsilon^k)_{k\in\NN}$ and $(\delta^k)_{k\in\NN}$ be nonnegative sequences such that $\sum_{k=0}^\infty\varepsilon^k<\infty$, $\sum_{k=0}^\infty\delta^k<\infty$ and
\begin{equation}\label{eq_rs}
\EE[v^{k+1}|\mc F_k]\leq (1+\delta^k)v^k+\varepsilon^k-\theta^k  \text{ a.s., for all } k\in\NN
\end{equation}
Then, $\sum_{k=0}^\infty \theta^k<\infty$ and $(v^k)_{k\in\NN}$ converges a.s. to a nonnegative random variable.
\end{lem}
\begin{center}
\begin{figure*}[t]
\centering
\includegraphics[width=\textwidth]{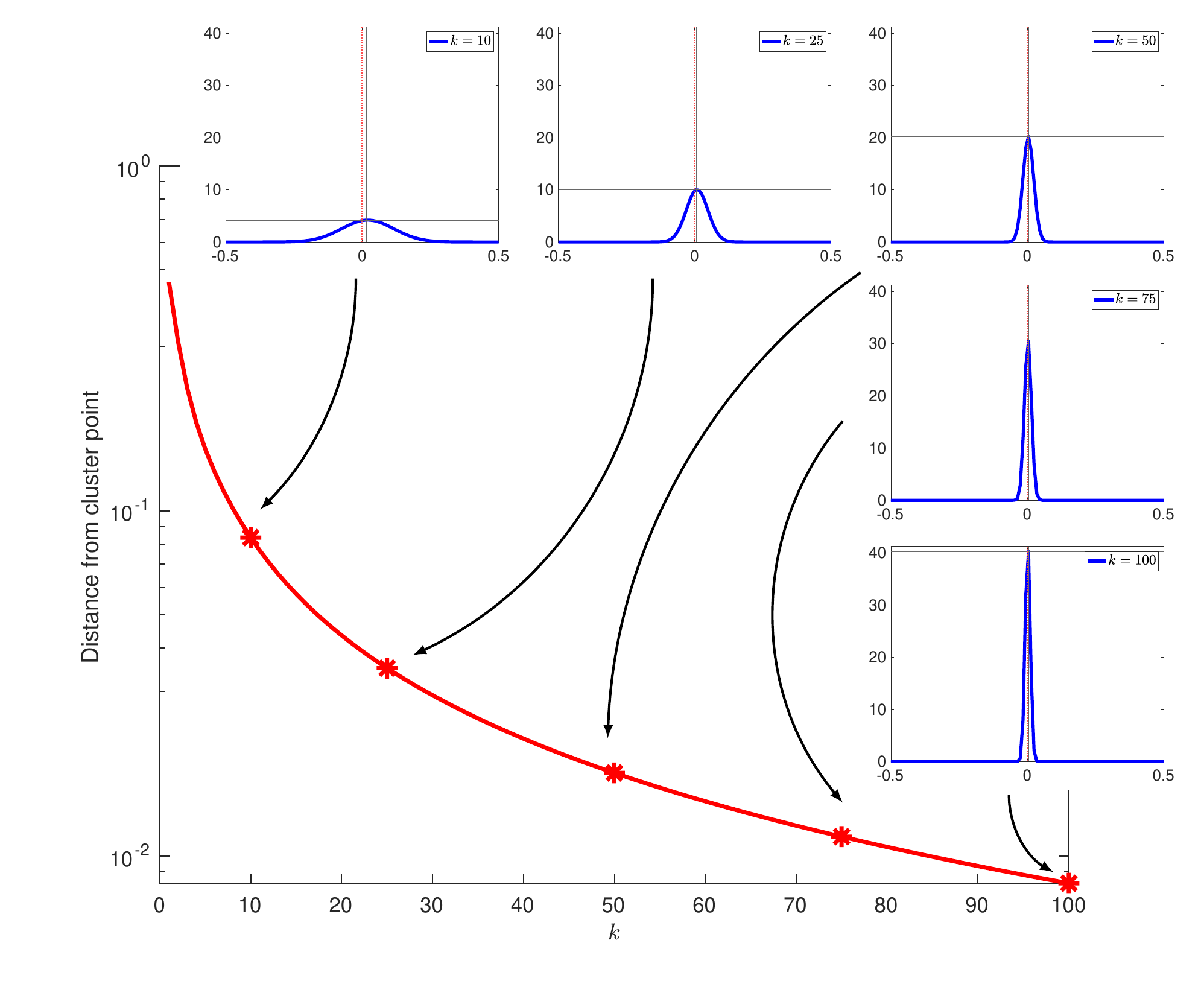}
\caption{The big plot shows how the average distance from a cluster point of the sequence $\EE[v^{k+1}|\mc F_k]= (1+\delta^k)v^k+\varepsilon^k-\theta^k$, generated by Robbins--Siegmund Lemma (Lemma \ref{lemma_rs}), goes to zero as the number of iterations increases.
The small plots show how the distribution of the distance from a cluster point varies with the iteration towards a probability distribution centered at 0, i.e., the sequence converges a.s..}\label{fig_lemma_rs}
\end{figure*}
\end{center}
\begin{proof}
The proof follows by rewriting the sequence as in Lemma \ref{lemma_det_rs}. Then, it is possible to show that the sequence
$$y^n=\tilde v^k-\sum_{k=0}^{n-1}(\tilde\varepsilon^k-\tilde\theta^k)$$
is a supermartingale. The claim then follows by the Martingale Convergence Theorem (Theorem \ref{teo_doob}). See \cite{RS1971} for technical details.
 \end{proof}
 \begin{rem}\label{remark_theta}
 Besides the convergence of the sequence $(v^k)_{k\in\NN}$, it is of particular interest also the fact that the sequence $(\theta^k)_{k\in\NN}$ is summable. Specifically, this result can be used to obtain more information once related with a (quasi-)F\'ejer property. In the stochastic case, this term is particularly useful, compared to the deterministic case, because there are not as many results on stochastic F\'ejer monotone sequences and the techniques available for the deterministic case cannot be used here. We refer to the application sections to see how the negative term is exploited.
  \end{rem}


The following results are consequences of Robbins--Siegmund Lemma. 
The first one is attributed to Gladyshev \cite{polyak1987,ljung2012}. In fact, it came implicitly in a work by Gladyshev \cite{glad1965} in which the author provides a proof of the convergence of Robbins--Monro algorithm \cite{robbins1951,ljung2012}. Even if it was published prior than the result by Robbins-Siegmund, it is a particular case of Lemma \ref{lemma_rs} \cite[Application 2]{RS1971}. 

\begin{cor}[Gladyshev \cite{glad1965}, Lemma 2.2.9, \cite{polyak1987}]\label{lemma_glad}
Let $(v^k)_{k\in\NN}$ be a nonnegative sequence of random variables. Let $\EE[v^0]<\infty$, and let $(\delta^k)_{k\in\NN}$ and $(\varepsilon^k)_{k\in\NN}$ be such that $\sum_{k=0}^\infty\delta^k<\infty$ and $\sum_{k=0}^\infty\varepsilon^k<\infty$
\begin{equation}\label{eq_glad}
\EE[v^{k+1}|\mc F_k]\leq(1+\delta^k)v^k+\varepsilon^k \text{ for all }k\in\NN
\end{equation}
Then $\lim_{k\to\infty}v^k= \bar v\geq 0$ a.s. where $\bar v$ is a random variable.
\end{cor}
\begin{proof}
It follows from Lemma \ref{lemma_rs} letting $\theta_k=0$. Different proofs can be found in \cite[Application 2]{RS1971}, \cite[Lemma 2.2.9]{polyak1987} or \cite[Theorem 3.3.6]{borkar1995}.
 \end{proof}
\begin{rem}
In \cite[Lemma 2.2]{barty2007}, it is shown that if Equation \eqref{eq_glad} holds, then the sequence $(v^k)_{k\in\NN}$ is bounded.
 \end{rem}
Similarly to Lemma \ref{lemma_comb} and Lemma \ref{lemma_det_rs} in the deterministic case, many results can be obtained removing or changing the sequences in \eqref{eq_rs}. In fact, the next corollary is straightforward from Lemma \ref{lemma_rs}.

\begin{cor}[Theorem B.2, \cite{poggio2011}]\label{cor_poggio}
Let $(v^k)_{k\in\NN}$ and $(\theta^k)_{k\in\NN}$ be positive sequences adapted to $\mc F=(\mc F_k)_{k\in\NN}$ and let
$$\EE[v^{k+1} |\mc F_k]\leq v^k-\theta^k \text{ for all } k\in\NN.$$
Then, $(v^k)_{k\in\NN}$ converges a.s. to a finite random variable $\bar v$ and $\sum_{k=0}^\infty\theta^k<\infty$.
\end{cor}

\begin{proof}
It follows from Robbins--Siegmund Lemma by taking $\varepsilon^k=0$. For a different proof, see \cite{poggio2011}.
 \end{proof}

Interestingly, we note that besides convergence of the sequence, there is additional information to be derived from Robbins--Siegmund Lemma. For instance, the next corollary is used in \cite{duflo2013} to prove the Law of Large Numbers for martingales \cite[Theorem 1.3.15]{duflo2013} (see also Section \ref{sec_app_stoc}). 

\begin{cor}[Corollary 1.3.13, \cite{duflo2013}]\label{cor_duflo}
Let $(v^k)_{k\in\NN}$, $(\theta^k)_{k\in\NN}$ and $(\varepsilon^k)_{k\in\NN}$ be positive sequences adapted to $\mc F=(\mc F_k)_{k\in\NN}$ and let $(a_k)_{k\in\NN}$ be a strictly positive, increasing sequence adapted to $\mc F$ such that
$$\EE[v^{k+1} |\mc F_k]\leq v^k+\varepsilon^k-\theta^k \text{ for all } k\in\NN.$$
Then, if $\sum_{k=1}^\infty a_k^{-1}\varepsilon^k<\infty$ the following hold a.s.:
\begin{itemize}
    \item[(i)] $\sum_{k=1}^\infty a_k^{-1}(v^{k+1}-v^k)$ converges and $\sum_{k=1}^\infty a_k\theta^k<\infty$;
    \item[(ii)] $(v^k)_{k\in\NN}$ converges if $(a_k)_{k\in\NN}$ is convergent;
    \item[(iii)] $\lim_{k\to\infty}a_k^{-1}v^k= 0$ and $\lim_{k\to\infty}a_k^{-1}v^{k+1}= 0$ if $(a_k)_{k\in\NN}$ is divergent.
\end{itemize}
\end{cor}
\begin{proof}
Let
$$u^k=\sum_{i=1}^kv^i\left(a_{i-1}^{-1}-a_i^{-1}\right)+v^ka_k^{-1}.$$
Then we can apply Robbins--Siegmund Lemma to the inequality
$$\EEk{u_{k+1}}\leq u^k+a_k^{-1}(\varepsilon^k-\theta^k)$$
and conclude the proof. For technical details, we refer to \cite{duflo2013}.  
 \end{proof}

The following proposition explicitly connects stochastic quasi-F\'ejer monotone sequences to Robbins--Siegmund Lemma.

\begin{prop}[Proposition 2.3, \cite{combettes2015}]\label{prop_gen_rs}
Let $\mc X\subseteq\RR^n$ be nonempty and closed, let $\phi:\:\RR_{\geq0} \to \RR_{\geq0}$ be a strictly increasing function such that $\lim_{t \to \infty} \phi(t)=\infty$, and let $(x^k)_{k\in\mathbb{N}}$ be a sequence of random variables. Let $\mc C(x^k)$ be the set of sequential cluster points of $(x^k)_{k \in \NN}$.
Suppose that, for every $\bar x\in\mc X$, there exist $(\delta^k)_{k\in\mathbb{N}}$, $(\theta^k)_{k\in \mathbb{N}}$, and $(\varepsilon^k)_{k\in\mathbb{N}}$ positive sequences such that $\sum_{k=0}^\infty\delta^k<\infty$, $\sum_{k=0}^\infty \varepsilon^k<\infty$ and
$$\begin{aligned}
\mathbb{E}(&\phi(||x^{k+1}-\bar x||)|\mc{F}_n)+\theta^k(\bar x) \\
&\leq(1+\delta^k(\bar x))\phi(||x^k-\bar x||)+\varepsilon^k(\bar x) \text{ a.s. for all }k\in\NN
\end{aligned}$$
Then the following hold:
\begin{enumerate}
    \item[(i)] $\sum_{k=1}^\infty \theta^k< \infty$ a.s.;
    \item[(ii)] $(x^k)_{k\in\mathbb{N}}$ is bounded a.s.;
    \item[(iii)] $(||x^k-\bar x||)_{k\in \mathbb{N}}$ converges a.s.;
    \item[(iv)] Let $\mc C(x^k) \subset \mc X$ a.s., then $(x^k)_{k\in\mathbb{N}}$ converges a.s..
\end{enumerate}
\end{prop}
\begin{proof}
(i) It follows from Lemma \ref{lemma_rs}.\\
(ii) Let $v^k=\norm{x^k-\bar x}$. Then, $\lim_{k\to\infty}\phi(v^k)= \bar v\in\RR_{\geq0}$ by Lemma \ref{lemma_rs}. Since $\lim_{t\to\infty}\phi(t)=\infty$, $v^k$ is bounded and, therefore, also $x^k$ is bounded.\\
(iii) It follows from (ii), for more details see \cite[Proposition 2.3]{combettes2015}.\\
(iv) Let $\bar x,\bar y\in\mc C(x^k)$. Then, there exist two subsequences $(x_{k_n})$ and $(x_{k_m})$ such that $x_{k_n}\to \bar x$ and $x_{k_m}\to\bar y$ as $k\to \infty$. By (iii) the sequences $(\norm{x^k-\bar x})_{k\in\NN}$ and $(\|x^k-\bar y\|)_{k\in\NN}$ converge and it holds that $\langle x^k,\bar x-\bar y\rangle\to \rho$ for some $\rho\in\RR^n$. Then, $\langle \bar x,\bar x-\bar y\rangle= \rho$, $\langle \bar y,\bar x-\bar y\rangle= \rho$ and
$$0=\langle \bar x, \bar x-\bar y\rangle-\langle \bar y,\bar x-\bar y\rangle=\|\bar x-\bar y\|^{2}.$$ 
Therefore, $\bar x=\bar y$ and $\lim_{k\to\infty}x^k=\bar x$.
 \end{proof}

\begin{rem}
A specific case of Proposition \ref{prop_gen_rs} was also presented in \cite[Lemma 2.3]{barty2007} without the negative term $\theta^k$, using $\phi=|\cdot|^2$ and setting $\delta^k=0$.\\
More generally, an analogous result holds with $\phi=|\cdot|^p$, $p>0$ \cite[Lemma 2.2]{combettes2019}.
 \end{rem}

Analogously to the deterministic case, also for sequences of random variables, we can find results for sequences with a coefficient strictly smaller than 1. This is the case of the following results.

\begin{lem}[Lemma 2.1, \cite{combettes2019}]\label{lemma_rs_comb}
Let $(v^k)_{k \in \mathbb{N}}$, $(\theta^k)_{k \in \mathbb{N}}$ and $(\varepsilon^k)_{k \in \mathbb{N}}$ be sequences of nonnegative random variables and suppose that there exists a nonnegative sequence $(\gamma^k)_{k \in \mathbb{N}}$
such that $\lim_{k\to\infty} \gamma^k<1$ and
$$
\EE[v^{k+1} \mid \mathcal{F}_{k}] \leq \gamma^k v^k+\varepsilon^k -\theta^k\text{ for all } k\in\NN.
$$
Moreover, let $\EE[v^0]<\infty$ and $\sum_{k=1}^\infty \EE[\varepsilon^k]<\infty .$ Then $\sum_{k=1}^\infty \EE[v^k]<\infty$ and
$\sum_{k=1}^\infty \EE[\theta^k]<\infty .$
\end{lem}
\begin{proof}
The proof follows with arguments similar to Lemma \ref{lemma_xu02} and Lemma \ref{lemma_neg_rs} but it can be proven also as a consequence of Lemma \ref{lemma_rs}. For technical details we refer to \cite{combettes2019}.
\end{proof}

We conclude this section with a lemma that is quite popular in the literature \cite{yousefian2017, koshal2013,polyak1987} and cited along with Robbins--Siegmund Lemma. It is the stochastic counterpart of Lemma \ref{lemma_polyak1} even if it has a slightly different notation.

\begin{lem}[Lemma 2.2.10, \cite{polyak1987}]\label{lemma_fake_rs}
Let $(v^k)_{k\in\NN}$ be a sequence of nonnegative random variables such that $\EE[v^0]< \infty$ and let $(\delta^k)_{k\in\NN}$ and $(\varepsilon^k)_{k\in\NN}$ be deterministic nonnegative sequences such that $0\leq \delta^k\leq 1$ for all $k\in\NN$, $\sum_{k=0}^\infty\delta^k=\infty$, $\sum_{k=0}^\infty\varepsilon^k<\infty$, $\lim_{k\to\infty}\frac{\varepsilon^k}{\delta^k}=0$ and 
$$\EE[v^{k+1}|\mc F_k]\leq (1-\delta^k)v^k+\varepsilon^k \text{ a.s., for all }k\in\NN.$$
Then, $\lim_{k\to\infty} v^k = 0$ a.s..
\end{lem}
\begin{proof}
The proof follows by applying Lemma \ref{lemma_polyak1} to
\begin{equation}\label{eq_fake_rs}
\EE[v^{k+1}]\leq(1-\delta^k)\EE[v^k]+\varepsilon^k
\end{equation}
and showing that
$$u^k=v^k-\sum_{i=k}^\infty\varepsilon_i$$
is a supermartingale. Then, the claim follows by the Martingale Convergence Theorem (Theorem \ref{teo_doob}). See \cite{polyak1987} for technical details.  
 \end{proof}

\begin{rem}
To retrieve the same form of Lemma \ref{lemma_polyak1}, one can refer to \cite[Lemma 1(a)]{lei2020br} where the following statement is provided. \\
Let $(v_{k})_{k \in\NN}$ and $(\varepsilon^k)_{k\in\NN}$ be sequences of nonnegative random variables such that 
$$\mathbb{E}[v^{k+1} \mid \mathcal{F}_k] \leq \gamma^{k} v^{k}+\varepsilon^{k} \text{ for all } k\in\NN,$$ 
$\mathbb{E}[v_{0}]<\infty$, $0 \leq \gamma^{k}<1$, $\sum_{k=0}^{\infty}(1-\gamma^{k})=\infty$, $\sum_{k=0}^{\infty} \varepsilon^{k}<\infty$, and $\lim _{k \rightarrow \infty} \frac{\varepsilon^{k}}{1-\gamma^{k}}=0 .$ Then,
$\lim_{k\to\infty} v_{k} = 0$ a.s..
\end{rem}

\begin{rem}
Convergence to zero in Lemma \ref{lemma_fake_rs} can also be derived from Lemma \ref{lemma_rs} by exploiting the properties of the negative term in Equation \eqref{eq_fake_rs}. In fact, from Lemma \ref{lemma_rs} and Equation \eqref{eq_fake_rs} we have that $\sum_{k=1}^\infty\delta^kv^k<\infty$ and since $\delta^k$ is not summable, it must be $\lim_{k\to\infty}v^k= 0$.
\end{rem}

\section{Convergence with variable metric}\label{sec_vm}

Let us consider in this section the more general setting with variable metric, i.e., cases in which the metric is allowed to change at each iteration. Applications of these results involve theoretical problems as monotone inclusions \cite{combettes2014,vu2013}, as well as inverse problems \cite{combettes2013}, convex feasibility problems \cite{nguyen2016,combettes2013} and constrained convex minimization \cite{cui2019vm}.  All the results in this section concern F\'ejer properties and we consider mostly the deterministic case. The first result that we propose is an extension of Proposition \ref{prop_bau}. 

\begin{prop}[Proposition 3.2, \cite{combettes2013}]\label{prop_vm} 
Let $\beta>0$ and let $(W_k)_{k \in \mathbb{N}}$ be in $\mathcal{P}_{\beta}$. Let $\phi:\RR_{\geq0}\to\RR_{\geq0}$ be strictly increasing and such that $\lim_t\phi(t)=\infty$. Let $\mc S\subseteq\RR^n$ be nonempty, and $(x^k)_{k\in\NN}$ be a quasi-F\'ejer monotone sequence in $\RR^n$ with respect to $\mc S$ relative to $(W_k)_{k\in\NN}$. Then the following hold:
\begin{itemize}
\item[(i)] $(x^k)_{k\in\NN}$ is bounded;
\item[(ii)] Let $\bar x\in \mc S$. Then $(\|x^k-\bar x\|_{W_k})_{k\in\NN}$ converges. 
\end{itemize}
\end{prop}
\begin{proof}
(i) follows by the fact that $(W_k)_{k\in\NN}$ is in $\mathcal P_\beta$. (ii) follows from Corollary \ref{cor_polyak} and by showing that there cannot be two cluster points.
 \end{proof}

\begin{rem}
A similar result holds also for quasi-Bregman monotone sequences and it can be proven analogously by applying Corollary \ref{cor_polyak} \cite{nguyen2017,nguyen2016}.
 \end{rem}

Analogously to Section \ref{sec_det}, under stronger assumptions, we can obtain stronger convergence results. In fact, the next result is a generalization of Theorem \ref{theo_qfm_comb}.

\begin{thm}[Theorem 3.3, \cite{combettes2013}] \label{theo_comb_var}
Let $\beta>0$ and let $(W_k)_{k \in \mathbb{N}}$ and W be operators in $\mathcal{P}_{\beta}$ such that $W_k \to W$ pointwise. Let $\phi:\RR_{\geq0}\to\RR_{\geq0}$ be strictly increasing and such that $\lim_t\phi(t)=\infty$. Let $\mc S\subseteq\RR^n$ be nonempty and let $(x^k)_{k\in\NN}$ be a quasi-F\'ejer monotone sequence in $\RR^n$ with respect to $\mc S$ and relative to $(W_k)_{k\in\NN}$. Then, $(x^k)_{k\in\NN}$ converges to a point in $\mc S$ if and only if every sequential cluster point of $(x^k)_{k\in\NN}$ is in $\mc S$.
\end{thm}
\begin{proof}
Necessity is straightforward while sufficiency follows by Proposition \ref{prop_vm} and Lemma \ref{lemma_cluster}.
 \end{proof}

Let us conclude this section with a result that is particularly interesting for the conditions on the sequence that induce the metric. 
\begin{prop}[Proposition 4.1, \cite{combettes2013}] \label{prop_comb_vm}
Let $\beta>0$. Let $(\eta^k)_{k\in\NN}$ be a nonnegative sequence such that $\sum_{k=1}^\infty\eta^k<\infty$, and let $(W_k)_{k \in \mathbb{N}}$ be a sequence in $\mathcal{P}_{\beta}$ such that 
\begin{equation}\label{eq_cond_vm}
\begin{aligned}
&\mu=\sup _{k\in\NN}\left\|W_k\right\|<\infty \quad \text { and } \\
&(1+\eta^k) W_k \succeq W_{k+1} \text{ for all }k\in\NN
\end{aligned}
\end{equation}
Let $\mc S\subseteq\RR^n$ be nonempty, closed and convex and let $(x^k)_{k\in\NN}$ be a quasi-Fejér monotone sequence with respect to $\mc S$ relative to $(W_k)_{k \in \mathbb{N}}$. Then, for every $\bar x \in \mc S$, the sequence $(\|x^k-\bar x\|_{W_k})_{k\in\NN}$ converges.
\end{prop}
\begin{proof}
It follows from Corollary \ref{cor_polyak} and Proposition \ref{prop_vm}. For technical details we refer to \cite{combettes2013}.
 \end{proof}
The condition in \eqref{eq_cond_vm} is not hard to check on the problem data and it can be helpful for application purposes.

We conclude this section with an adaptation of Proposition \ref{prop_gen_rs} to the variable metric setup, i.e., an extension of Robbins-Siegmund Lemma (Lemma \ref{lemma_rs}) to variable metric quasi-Fejer monotone sequences.
\begin{prop}[Proposition 2.4, \cite{vu2016}]\label{prov_vm_rs}
Let $\mc X\subseteq\RR^n$ be a non-empty closed set and let $\phi:\RR_{\geq0}\rightarrow\RR_{\geq0}$. Let $\beta>0$, let $W$ and $(W_{k})_{k \in \mathbb{N}}$ be operators in $\mathcal{P}_{\beta}$ such that $W_{k} \rightarrow W$ pointwise. Let $(x_{k})_{k \in \mathbb{N}}$ be a sequence of random vectors. Suppose that, for every $\bar x \in \mc X$, there exist $(\delta^k)_{k\in\mathbb{N}}$ and $(\varepsilon^k)_{k\in\mathbb{N}}$ nonnegative sequences such that $\sum_{k=0}^\infty\delta^k<\infty$ and $\sum_{k=0}^\infty \varepsilon^k<\infty$ and
such that,
$$
\begin{aligned}
\EE[\phi(\|x_{k+1}&-\bar x\|_{W_{k+1}}) \mid \mathcal{F}_{k}]\\
& \leq(1+\delta_{k}(\bar x)) \phi(\|x_{k}-\bar x\|_{W_{k}})+\varepsilon_{k}(\bar x) \text{ a.s., for all } k\in\NN
\end{aligned}
$$
Suppose that $\phi$ is strictly increasing and $\lim _{t \rightarrow \infty} \phi(t)=+\infty .$ Then, the following hold.
\begin{itemize}
\item[(i)] $(\|x_{k}-\bar x\|_{W_{k}})_{k \in \mathbb{N}}$ is bounded and converges a.s.;
\item[(ii)] $(x_{k})_{k \in \mathbb{N}}$ converges a.s. to random vector in $\mc X$ if and only if every cluster point is in $\mc X$ a.s..
\end{itemize}
\end{prop}
\begin{proof}
(i) follows from Lemma \ref{lemma_rs} and by the fact the $\phi$ is strictly increasing.\\
(ii) Necessity is straightforward while sufficiency follows from the properties of a cluster point and Lemma \ref{lemma_cluster}. For more details we refer to \cite{vu2016}.
\end{proof}

\section{Applications of convergent deterministic sequences}\label{sec_app_det}

Since variational inequalities are the mathematical foundations of optimization-related problems, such as Nash equilibrium seeking \cite{facchinei2007,franci2020fb, yi2019}, convex optimization \cite{facchinei2007,jofre2019,bau2011} and machine learning \cite{gidel2018, franci2020gan}, many works in the literature rely on the results presented in the previous sections to prove convergence of a given algorithm to a solution of a variational equilibrium problem. Specifically, they are applied to prove that a given algorithm converges to the solution of a variational inequality or to a zero of the sum of (monotone) operators. Thus, let us first describe the variational problem, starting by the definition of variational inequality \cite{bau2011,facchinei2007}.
\begin{defn}
Given a set $\mc X\subseteq\RR^n$ and a mapping $F : \mc X \to \RR^{n}$, a variational inequality, denoted $\op{VI}(\mc X, F)$, is the problem 
\begin{equation}\label{eq_vi}
\text{find } x^*\in \mc X \text{ such that }\langle F(x^*),y-x^*\rangle \geq 0, \text{ for all } y \in \mc X
\end{equation}
The set of solutions to this problem is denoted by $\op{SOL}(\mc X, F).$
 \end{defn}

The geometric interpretation of \eqref{eq_vi} is that a point $x^*\in \mc X$ is a solution of $\op{VI} (\mc X, F)$ if and only if $F(x^*)$ forms an acute angle with every vector of the form $y - x^*$ for all $y \in \mc X$. 
In other words, (\ref{eq_vi}) also says that a vector $x\in \mc X$ solves $\op{VI}(\mc X, F)$ if
and only if $-F (x^*)$ is a vector in the normal cone of $\mc X$ at $x^*$ (see Appendix \ref{appendix} for the definition), i.e.,
\begin{equation}\label{cone_incl}
0 \in F(x^*)+\op{N}_{\mc X}(x^*).
\end{equation}
Sometimes, instead of problem (\ref{eq_vi}), a more general definition is proposed: 
\begin{equation}\label{gen_vi}
\text { find } x^* \in \mc X \text { s.t. } \langle F(x^*), y-x^*\rangle+ g(y)-g(x^*) \geq 0, \ \ \text{for all } y \in \mc X
\end{equation}
where $g$ is a proper lower semi-continuous and convex function.
Examples for the function $g$ are indicator functions to enforce the set constraints, or penalty functions that promote sparsity, or other desirable structure. 

Similarly to \eqref{eq_vi} and \eqref{cone_incl}, the problem in \eqref{gen_vi} can be rewritten as 
\begin{equation}\label{eq_cone_gen}
\text { find } x^* \in \mc X \text { s.t. }0\in (F +\partial g)(x^*),
\end{equation}
where $\partial g$ is the subdifferential of $g$ (definition in Appendix \ref{appendix}).
In fact, if in \eqref{gen_vi} we take $g$ as the indicator function, i.e., $g(x)=\iota_{\mc X}(x)$, we obtain the standard variation inequality (\ref{eq_vi}), and instead of \eqref{eq_cone_gen} we obtain the inclusion in \eqref{cone_incl} since $\partial g=\partial \iota_{\mc X}=\op{N}_{\mc X}$ \cite[Equation (14)]{combettes2020}.

Problems of the form (\ref{cone_incl}) and \eqref{eq_cone_gen} are usually called (monotone) inclusion problems, which aim, in the general form, at finding $x^*\in \mc X$ such that $0\in T(x^*)$ with $T:\RR^n\to\RR^n$. Moreover, in many cases it is possible to write a mapping as the summation of two (monotone) operators through an operator splitting technique \cite{ryu2016,bau2011}. In this case, the problem of finding a zero of a monotone operator $T=A + B$ can be rewritten as
\begin{equation}\label{mono_incl}
\text {find } x^* \in \mc X \text { such that } 0 \in(A+B)(x^*).
\end{equation}
Usually, $A : \RR^n \rightrightarrows \RR^n$ and $B : \RR^n \to \RR^n$ are a set valued and a single valued monotone operator, respectively. Inclusions as the above arise systematically in convex optimization \cite{malitsky2015,bot2020gan,malitsky2020siam,csetnek2019} and generalized Nash equilibrium problems in convex-monotone games  \cite{franci2020fb,franci2019fbhf,yi2019,pavel2019,gadjov2019,gadjov2020}.

\begin{exmp}[Inclusion problem]
Consider the minimization problem
\begin{equation}\label{eq_ex_min}
\min _{x \in \mc{X}} f(x)+g(x)
\end{equation}
where $g : \mc{X} \to\bar\RR$ is proper, lower semicontinuous and convex and $f : \mc{X} \to \RR$ is convex with Lipschitz continuous gradient. The solutions of the minimization problem in \eqref{eq_ex_min} are the points $x \in \mc X$ such that
\begin{equation}\label{minimization}
0 \in(\nabla f+\partial g)(x).
\end{equation}
where $\partial g$ denotes the subdifferential of $g$ and $\nabla f$ is the gradient of $f$. Equation (\ref{minimization}) is a monotone inclusion and it is equivalent to the generalized VI in \eqref{gen_vi} with $F=\nabla f$.
 \end{exmp}

We are now ready to present some algorithms where the lemmas of Section \ref{sec_det} are used. The algorithms often rely on the monotonicity properties of the operators involved (see Definitions \ref{def_mono} and \ref{def_lip}) and, unless otherwise mentioned, we suppose the following assumption to hold.

\begin{standass}\label{ass_sol}
The solution set of $\op{VI}(\mc X,F)$ is not empty, i.e., $\op{SOL}(\mc X, F)\neq\varnothing$, and $x_0\in\mc X$, i.e., the sequence starts in the set $\mc X$ which is closed and convex.
\end{standass}

For every algorithm, we also propose a sketch of the convergence proof to show how the lemmas are used.
A schematic representation of the necessary steps is provided in Figure \ref{fig_app_det}.
The main idea to prove convergence of an algorithm is to obtain a (quasi) F\'ejer inequality with respect to the solution set and then apply one of the lemmas to the sequence $v^k=\|x^k-x^*\|^2$ where $x^*\in \op{SOL}(\mc X,F)$ (see also Remark \ref{remark_tozero}). Analogously, one can show that a suitable Lyapunov function asymptotically goes to zero. 

We list the application depending on the type of problem but we name them after the convergence result that is used. We start with monotone inclusions, then move to VIs and Nash equilibrium problems, and finally consider an example of Lyapunov decrease.

\begin{figure}
\centering
 \begin{tikzpicture}[scale=1,every node/.style={circle,draw=black,scale=.5,minimum size=.8cm}]
 \LARGE
\draw[rounded corners=2ex,draw=black,dashed,thick,fill=white,fill=none] (-3,3) rectangle (3,5.5);
\draw[rounded corners=2ex,draw=black,dashed,thick,fill=white,fill=none] (-3,2.7) rectangle (3,-2);

\node[rectangle, rounded corners=.5ex, minimum size=1.5cm, fill=white](m1) at (0,5){iterative process $(x^k)_{k\in\NN}$};
\node[rectangle, rounded corners=.5ex, minimum size=1.5cm, fill=white](m2) at (0,3.5){find $V(x^k)=:v^k$};
\node[rectangle, rounded corners=.5ex, minimum size=1.5cm, fill=white](m3) at (0,1.75){$\begin{array}{c}\text{find inequality:}\\v^{k+1}\leq C^kv^k+\varepsilon^k-\theta^k\end{array}$};
\node[rectangle, rounded corners=.5ex, minimum size=1.5cm, fill=white](m4) at (0,0){see Table \ref{table_lemmi_det}};
\node[rectangle, rounded corners=.5ex, minimum size=1.5cm, fill=white](m5) at (0,-1.5){apply convergence lemma};
\node[rectangle, rounded corners=.5ex, minimum size=1.5cm, fill=white](m6) at (0,-3){$v^k\to v^*$};

\draw[thick] (m4)--(2,0)--(2,1.75);
\draw[thick,-to](2,1.75)--(m3);
\draw[thick] (m5)--(-2,-1.5)--(-2,3.5);
\draw[thick,-to](-2,3.5)-- (m2);

\foreach \from/\to in
{m1/m2,m2/m3,m3/m4,m4/m5,m5/m6}
\draw[-to,thick] (\from) -- (\to); 

\node[draw=none,fill=none,rotate=90] at (-3.5,0) {analysis};
\node[draw=none,fill=none,rotate=90] at (-3.5,4.25) {design};
\node[draw=none,fill=none] at (3.5,0){};
 

 \end{tikzpicture}
 \caption{Schematic representation of how the convergence lemmas for sequences can be used. Given the iterative process, a suitable nonnegative function (Lyapunov or distance-like) should be designed. Then, exploiting the properties of the application at hand, an inequality involving the iterates at times $k+1$ and $k$ can be retrieved. Hence, one should check if the inequality corresponds to a known result (Table \ref{table_lemmi_det} for sequences of real numbers) and use the corresponding result to prove convergence. The whole process may take repeated steps to find a suitable function and/or inequality. The same reasoning applies to the stochastic case, in which one should have an expected valued inequality (with $\EE[v^{k+1}]$) and refer to Table \ref{table_lemmi_stoc} for a convergence result on stochastic sequences. See also Figure \ref{fig_app_stoc} for an example.}\label{fig_app_det}
 \end{figure}
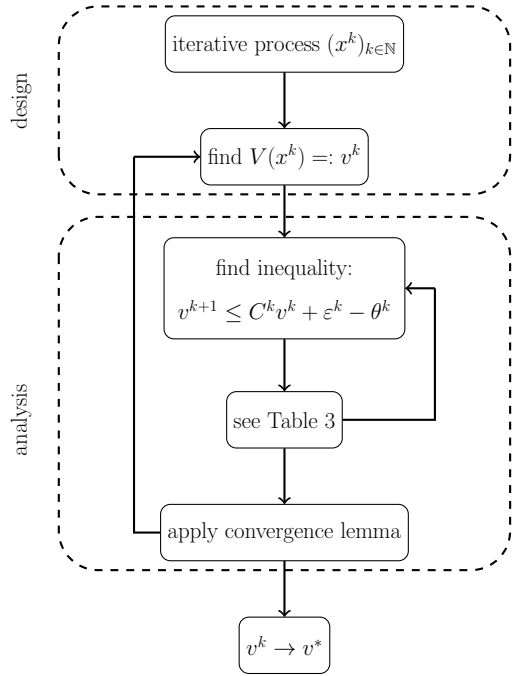

\subsection{Applications to Monotone Inclusions}

\paragraph*{Application of Lemma \ref{lemma_opial} and Lemma \ref{lemma_rate}}
Lemma \ref{lemma_opial} is used in \cite{malitsky2018fbf, malitsky2020siam} to prove convergence in the inclusion problem:
$$\text{find } x\in\mc X \text{ such that }0\in (A+B)(x)$$
where $A:\mc X\rightrightarrows \mc X$ and $B:\mc X\to \mc X$ are monotone operators. The sequence $(x^k)_{k\in\NN}$, generated by the algorithm, is defined according to
\begin{equation}\label{eq_mali_FRB}
x^{k+1}=J_{\alpha_kA}\left(x^k-\alpha_kB(x^k)-\alpha_{k-1}(B(x^k)-B(x^{k-1}))\right)
\end{equation}
where $J_{\alpha_kA}=(I+\alpha_k A)^{-1}$ is the resolvent of A (Definition \ref{def_projproxres}). The algorithm is named \textit{forward - reflected - backward splitting} and it is proven to converge to a zero of $A+B$.
\begin{thm}[Theorem 2.5, \cite{malitsky2020siam}]\label{malitzky_theo}
Let $A:\mc H\rightrightarrows \mc H$ be maximally monotone and $B:\mc H\to \mc H$ be monotone and $\ell$-Lipschitz continuous. Let $\epsilon >0$ and suppose $\alpha_k\in \left[\epsilon,\frac{1-2\epsilon}{2\ell}\right]$ for all $k\in\NN$.
Then, the sequence $(x^k)_{k\in\NN}$ generated by (\ref{eq_mali_FRB}) converges to a point $x^*\in \mc X$ such that $0\in (A+B)(x^*)$.
\end{thm}
\begin{proof}
Let $x^*\in (A+B)^{-1}(0)$. It is possible to show, by using monotonicity and some norm properties, that the following inequality holds:
\begin{equation}\label{FRB_eq}
\begin{aligned}
&\normsq{x^{k+1}-x^*}+2 \alpha_{k} \langle B(x^{k+1})-B(x^k), x^*-x^{k+1}\rangle+\\
&+\left(\frac{1}{2}+\epsilon\right)\normsq{x^{k+1}-x^k} \\ 
& \leq\normsq{x^k-x^*}+2 \alpha_{k-1}\langle B(x^k)-B(x^{k-1}), x^*-x^k\rangle\\
&+\frac{1}{2}\normsq{x^k-x^{k-1}} .
\end{aligned}
\end{equation}
Then, by doing a telescopic sum, using Lipschitz continuity and the properties of the parameters involved, the inequality in \eqref{FRB_eq} can be rewritten as
$$
\begin{aligned}
&\frac{1}{2}\|x_{k+1}-x^*\|^{2}+\varepsilon \sum_{i=0}^{k}\|x_{i+1}-x_{i}\|^{2} \\
&\quad \leq\|x_{0}-x^*\|^{2}+2 \lambda_{-1}\langle B(x_{0})-B(x_{-1}), x^*-x_{0}\rangle+\frac{1}{2}\|x_{0}-x_{-1}\|^{2}
\end{aligned}
$$
from which we deduce that $(x^k)_{k \in \mathbb{N}}$ is bounded and that $\lim_{k\to\infty}\norm{x^k-x^{k+1}} = 0$. 
Now, let $\bar x$ be a cluster point of $(x^k)_{k\in\NN}$. From the definition of the algorithm in \eqref{eq_mali_FRB} and the properties of $A+B$, it follows that $0\in(A+B)(\bar x)$. Using again \eqref{FRB_eq} and Lipschitz continuity it can be proven that $\lim_{k\to\infty}\normsq{x^k-\bar x}$ exists. Then, by Lemma \ref{lemma_opial}, the sequence is convergent.
 \end{proof}
 
The authors propose in the same paper also a variant of the algorithm with line search and a second one with inertia, but the convergence proof does not change its essence; in the first case, the authors use locally Lipschitz continuity \cite[Theorem 3.4]{malitsky2020siam}, while in the second they exploit the $1/\ell$-cocoercivity of the operator $B$ \cite[Theorem 4.3]{malitsky2020siam}. 
Moreover, under the assumption of strong monotonicity of the operator $A$, they also prove convergence with linear rate, using Lemma \ref{lemma_rate}.
\begin{thm}[Theorem 2.9, \cite{malitsky2020siam}]\label{theo_mali_rate}
Let $A:\mc H\rightrightarrows \mc H$ be maximally monotone and $\mu$-strongly monotone and $B:\mc H\to \mc H$ be monotone and $\ell$-Lipschitz continuous. Suppose $\alpha\in\left(0,\frac{1}{2\ell}\right)$. Then, the sequence $(x^k)_{k\in\NN}$ generated by \eqref{eq_mali_FRB} converges R-linearly to the unique point $\bar x\in \mc X$ such that $0\in (A+B)(\bar x)$.
\end{thm}
\begin{proof}
Similarly to the proof of Theorem \ref{malitzky_theo} but using strong monotonicity, one obtains the inequality
\begin{equation}
\begin{aligned}
(1+2 \mu \alpha) &\| x^{k+1}-x^* \|^{2}+2 \alpha\langle B(x^{k+1})-B(x^k), x^*-x^{k+1}\rangle\\
&+(1-\alpha \ell)\| x^{k+1}-x^k \|^{2} \\
 \leq&\|x^k-x^*\|^{2}+2 \alpha\langle B(x^k)-B(x^{k-1}), x^*-x^k\rangle\\
&+\frac{1}{2}\|x^k-x^{k-1}\|^{2}.
\end{aligned}
\end{equation}
Setting $\gamma=(1+2 \mu \alpha)>1$, $v_{k}:=\frac{1}{2}\|x^k-x^*\|^{2}$ and $\beta_{k}:=\frac{1}{2}\|x^k-x^*\|^{2}+2 \alpha\langle B(x^k)-B\left(x^{k-1}\right), x^*-x^k\rangle+\frac{1}{2}\|x^k-x^{k-1}\|^{2}$, one can apply Lemma \ref{lemma_rate} to conclude that the sequence $(x^k)_{k\in\NN}$ converges to the unique solution $\bar x$ and with a linear rate.
 \end{proof}

\paragraph*{Application of Corollary \ref{cor_liu}}
As an application of Corollary \ref{cor_liu}, let us consider the \textit{inertial forward-backward} algorithm proposed in \cite{dadashi2019} for approximating a zero of an inclusion problem $x\in(A+B)^{-1}(0)$:
\begin{equation}\label{eq_dadashi}
\left\{\begin{array}{l}
y^k=J_{\alpha_{k}A}\left(x^k-\alpha_{k}B x^k\right) \\
x^{k+1}=\nu_{k} x^k+\beta_{k} y^k+\gamma_{k} e^{k}
\end{array}\right.
\end{equation}
where $J_{\alpha_{k}A}$ is the resolvent of $A$ (Definition \ref{def_projproxres}) and $e^k$ is an error vector. By using Corollary \ref{cor_liu}  the authors prove the following result.

\begin{thm}[Theorem 3.1, \cite{dadashi2019}]\label{theo_dadashi}
Let $B$ be $\alpha$-cocoercive and let $A$ be maximally monotone. Let $\nu_k, \beta_{k}, \gamma_{k} \in(0,1)$ be such that $\nu_k+\beta_{k}+\gamma_{k}=1$ and
\begin{enumerate}
\item[1.] $\lim _{k \rightarrow \infty} \gamma_{k}=0,$ and $\sum_{k=1}^\infty \gamma_{k}=\infty$,
\item[2.] $\lim_{k\to\infty}e^{k} = 0$,
\item[3.] $0<a \leq \nu_k \leq b<1$ and $0<c \leq \beta_{k} \leq d<1$,
\item[4.] $0<c \leq \alpha_{k}<2 \alpha$ and $\lim _{k \rightarrow \infty}\left(\alpha_{k}-\alpha_{k+1}\right)=0$.
\end{enumerate}
Then, the sequence $(x^k)_{k\in\NN}$ generated by \eqref{eq_dadashi} converges to the point $x^* \in(A+B)^{-1}(0),$ where $x^*=\op{proj}_{(A+B)^{-1}(0)}(0)$.
\end{thm}

\begin{proof}
Using the nonexpansiveness of the resolvent of a maximally monotone operator \cite[Corollary 23.9]{bau2011} and the cocoercivity of the mapping $B$, one can prove that the sequence $(x^k)_{k\in\NN}$ is bounded. Then, using some properties of the resolvent \cite[Lemma 2.6]{dadashi2019} and of the convex combination of bounded sequences \cite[Lemma 2.8]{dadashi2019} and using the monotonicity of $A$, the following inequality hold:
$$\normsq{x^{k+1}-x^*}\leq\normsq{x^{k}-x^*}-\delta^k,$$
where $\delta^k$ is a quantity depending on the error $e_k$ and on $x^*$ and such that the assumption of Lemma \ref{cor_liu} are satisfied. Therefore, convergence holds.
 \end{proof}

\subsection{Applications to Variational Inequalities}

\paragraph*{Application of Lemma \ref{lemma_opial} and Lemma \ref{lemma_det_rs}}
The authors in \cite{malitsky2019} consider the general variational inequality problem in \eqref{gen_vi} where
$g : \mathcal{X} \rightarrow\bar\RR$ is a proper convex lower semicontinuous function and $F : \operatorname{dom} g \to \mathcal{X}$ is monotone. They propose the \textit{Golden Ratio Algorithm} (GRAAL) whose iterations are given by
\begin{equation}\label{mali_golden}
\begin{array}{c}
\tilde x^k =\frac{(\varphi-1) x^k+\bar{x}^{k-1}}{\varphi} \\ [10pt]
x^{k+1} =\operatorname{prox}_{\alpha g}(\tilde x^k-\alpha F(x^k)) \end{array}
\end{equation}
where $\varphi=\frac{\sqrt{5}+1}{2}$ is the golden ratio, i.e., $\varphi^{2}=1+\varphi .$
To prove convergence, they use Lemma \ref{lemma_opial} and Lemma \ref{lemma_det_rs}.
\begin{thm}[Theorem 1, \cite{malitsky2019}]\label{theo_mali2}
Let F be $\ell$-Lipchitz continuous and monotone that $g$ be lower semicontinuous and let $\alpha \in\left(0, \frac{\varphi}{2 \ell}\right]$. Then the sequences $(x^{k})_{k\in\NN}$ and $(\tilde x^{k})_{k\in\NN}$, generated by (\ref{mali_golden}), converge to a solution of the VI in (\ref{gen_vi}).
\end{thm}
\begin{proof}
Using the fact that $F$ is Lipschitz continuous and monotone and that the proximal operator is firmly nonexpansive, it holds that
\begin{equation}\label{eq_step_mali}
\begin{aligned}
&(1+\varphi)\left\|\tilde{x}^{k+1}-x^{*}\right\|^{2}+\frac{\varphi}{2}\left\|x^{k+1}-x^k\right\|^{2} \leq\\
&\leq(1+\varphi)\left\|\tilde x^k-x^{*}\right\|^{2}+\frac{\varphi}{2}\left\|x^k-x^{k-1}\right\|^{2}-\varphi\left\|x^k-\tilde x^k\right\|^{2}.
\end{aligned}
\end{equation}
Then, $(\tilde x^k)_{k\in\NN}$ is bounded and $\lim _{k \rightarrow \infty}\|x^k-\tilde x^k\|=0$ by Lemma \ref{lemma_det_rs}. Hence, $(x^k)_{k\in\NN}$ has at least one cluster point. 
Then, using the properties of $g$, all cluster points of $\tilde x^k$ are solutions of $\op{VI}(\mc X, F)$. Since the sequence on the righthandside is non increasing, it is also convergent to a point in the solution set $\op{SOL}(\mc X,F)$. Therefore, using the fact that $\lim _{k \rightarrow \infty}\|x^k-\tilde x^k\|=0$ and the definition of $\tilde x^k$ in \eqref{mali_golden}, Lemma \ref{lemma_opial} can be applied to conclude that $(x^k)_{k\in\NN}$ converges to a solution of \eqref{gen_vi}.
 \end{proof}
\begin{rem}
Interestingly, in a preliminary version of the paper \cite{malitsky2018}, the authors use Theorem \ref{theo_qfm_comb} to prove convergence. In fact, given Equation \eqref{eq_step_mali} and using the properties of the mapping $g$, they obtain that $(x^k)_{k\in\NN}$ has a cluster point and they can directly apply Theorem \ref{theo_qfm_comb} to conclude convergence.
 \end{rem}
In \cite{malitsky2019}, the authors prove convergence of the \textit{explicit GRAAL} which is a variation of algorithm in \eqref{mali_golden} with an adaptive step size rule. In this case, they only use locally Lipschitz continuity and conclude convergence via Lemma \ref{lemma_opial} \cite[Theorem 2]{malitsky2019}. 

The algorithm has been recently extended tot he stochastic case and for stochastic generalized Nash equilibrium problems \cite{franci2021} and generative adversarial networks \cite{franci2020gen} with a proof that relies on Lemma \ref{lemma_rs} on the same line of Section \ref{sec_app_rs}.

\paragraph*{Application of Corollary \ref{cor_mali1}}
Corollary \ref{cor_mali1} is used in \cite{malitsky2015} to prove convergence of the \textit{projected reflected gradient method} for variational inequalities as in \eqref{eq_vi}. In details, the algorithm reads as
\begin{equation}\label{mali_proj_ref}
x^{k+1}=\op{proj}_{\mc X}(x^{k}-\alpha F(2 x^{k}-x^{k-1}))
\end{equation}
and they show that the following result holds. 
\begin{thm}[Theorem 3.2, \cite{malitsky2015}]\label{theo_mali3}
Let $F$ be monotone and $\ell$-Lipschitz continuous and $\alpha \in(0, \tfrac{\sqrt{2}-1}{\ell})$. Then the sequence $(x^k)_{k \in \NN}$ generated by (\ref{mali_proj_ref}) converges to a solution of $\op{VI}(\mc X,F)$ in \eqref{eq_vi}.
\end{thm}
\begin{proof}
Using the firmly nonexpansiveness of the projection, the fact that the mapping is monotone and $\ell$-Lipschitz continuous and the bound on the step sizes, the following inequality holds:
$$
\begin{aligned}
\left\|x^{k+1}-x^*\right\|^{2}+& \alpha \ell\left\|x^{k+1}-y_k\right\|^{2}+2 \alpha\langle F(z), x^{k}-x^*\rangle \\
\leq &\left\|x^{k}-x^*\right\|^{2}+\alpha \ell\left\|x^{k}-y^{k-1}\right\|^{2}\\
&+2 \alpha\langle F(x^*), x^{k-1}-x^*\rangle \\
&-(1-\alpha\ell(1+\sqrt{2}))\left\|x^{k}-x^{k-1}\right\|^{2},
\end{aligned}
$$
where $x^*\in\op{SOL}(\mc X,F)$. Now, by letting
$$
\begin{array}{l}
v^k=\|x^{k}-x^*\|^{2}+\alpha \ell\|x^{k}-y^{k-1}\|^{2}+2 \alpha\langle F(x^*), x^{k-1}-x^*\rangle \\
\theta^k=(1-\alpha \ell(1+\sqrt{2}))\|x^{k}-x^{k-1}\|^{2},
\end{array}
$$
it follows that $v^{k+1}\leq v^k-\theta^k$ as in Corollary \ref{cor_mali1}, from which it is possible to deduce that $(x^k)_{k\in\NN}$ is bounded and has at least one cluster point $\bar x$ and that $\lim_{k\to\infty}\norm{x^k-x^{k-1}}=0$. By Minty Theorem \cite[Lemma 2.2]{malitsky2015} one have that any cluster point $\bar x$  is also a solution of the VI. By contradiction, it is possible to prove that $(x^k)_{k\in\NN}$ cannot have two cluster points, therefore $\lim_{k\to\infty}x^k= \bar x\in\op{SOL}(\mc X,F)$.
 \end{proof}

Since the constant $\ell$ can be hard to compute, to avoid using $\ell$-Lipschitz continuity, in the same paper, the authors also propose a variant of the algorithm in \eqref{mali_proj_ref} that includes a prediction-correction technique to select the step sizes. The convergence result \cite[Theorem 4.4]{malitsky2015} is proved similarly to the original result, using Corollary \ref{cor_mali1}.
Moreover, they also provide an estimation of the convergence rate when the mapping $F$ is strongly monotone (similarly to Theorem \ref{theo_mali_rate}) using a result similar to Lemma \ref{lemma_rate} \cite[Lemma 2.9]{malitsky2015}.

This algorithm has been recently extended to the stochastic case \cite{cui2019, cui2016} and proved similarly, by exploiting Lemma \ref{lemma_fake_rs}.

\subsection{Applications to Nash equilibrium problems}
\paragraph*{Application of Lemma \ref{lemma_xu02}}
The fact that Lemma \ref{lemma_xu02} guarantees convergence to zero (Remark \ref{remark_tozero}) is used in \cite{duvocelle2019} to compute a Nash equilibrium in traffic networks.
In a dynamic traffic assignment problem, travelers participate in a non-cooperative Nash game choosing a departure time and a route.
The author propose a \textit{forward-backward-forward} algorithm (inspired by \cite{tseng2000}), given by
\begin{equation}\label{algo_traffic}
\begin{array}{l}
z^{k}=\op{proj}_{\mathcal{X}}[x^{k}-\alpha F(x^{k})] \\
y^{k}=z^{k}+\alpha(F(x^{k})-F(z^{k})) \\
x^{k+1}=(1-\nu_{k}-\gamma_{k}) x^{k}+\gamma_{k} z^{k},
\end{array}
\end{equation}
to solve the associated variational problem. The convergence result is stated next and it shows convergence to the solution of the VI associated to the Nash equilibrium problem \cite[Proposition 1.4.2]{facchinei2007}.
\begin{thm}[Theorem 3.1, \cite{duvocelle2019}] \label{thoe_duvo}
Let F be pseudomonotone and $\ell$-Lipschitz continuous. Let $(\nu^{k})_{k \in \mathbb{N}}$ and $(\gamma_{k})_{k \in \mathbb{N}}$ be sequences in $(0,1),$ such that $(\gamma_{k})_{k \in \mathbb{N}} \subset\left(\nu, 1-\nu_{k}\right)$ for some $\nu>0,$ and let $\lim _{k \rightarrow \infty} \nu_{k}=0$ and $\sum_{k=1}^{\infty} \nu_{k}=\infty$. Then, the sequence $(x^{k})_{k \in \mathbb{N}}$ generated by \eqref{algo_traffic} converges to $x^* \in \op{SOL}(\mc X,F)$ where $x^*=\operatorname{argmin}\left\{\|z\|: z \in \op{SOL}(\mc X,F)\right\}$.
\end{thm}
\begin{proof}
Using the definition of the algorithm in \eqref{algo_traffic} and some preliminary inequalities \cite[Lemma 4.1]{duvocelle2019}, it holds that \cite[Lemma 4.3]{duvocelle2019}
\begin{equation}\label{eq_traffic}
\begin{aligned}
&\|x^{k+1}-x^*\|^{2} \leq(1-\nu_{k})\|x^{k}-x^*\|^{2}+\\
&\nu_{k}[2 \gamma_{k}\|x^{k}-y^{k}\| \cdot\|x^{k+1}-x^*\|+2\langle x^*, x^*-x^{k+1}\rangle]
\end{aligned}
\end{equation}
To apply Lemma \ref{lemma_xu02} to the sequence $v^k=\normsq{x^{k+1}-x^*}$, the authors check the conditions on $\beta^k=2 \gamma_{k}\|x^{k}-y^{k}\| \cdot\|x^{k+1}-x^*\|+2\langle x^*, x^*-x^{k+1}\rangle$. First, note that since $\mc X$ is closed and convex, there exists a unique $x^*\in\op{SOL}(\mc X,F)$ such that $x^*=\op{proj}_{\op{SOL}(\mc X,F)}(0)$. Now, suppose that there exists $k_0\in\NN$ such that $\normsq{x^{k+1}-x^*}\leq\normsq{x^k-x^*}$ for all $k\geq k_0$. Then, $\lim_{k\to\infty}\normsq{x^k-x^*}$ exists. Then, exploiting the properties of the step size and using monotonicity and Lipschitz continuity, it can be proven that $\lim_{k\to\infty}\normsq{x^k-z^k}=0$ and that, by the definition of $y_k$, $\lim_{k\to\infty}\normsq{y^k-x^k}=0$. Therefore, $\lim_{k\to\infty}\normsq{x^{k+1}-x^k}=0$. Since the sequence is bounded \cite[Lemma 4.2]{duvocelle2019}, there exists a subsequence $(x^{k_j})$ such that $x^{k_j}\to \bar y$ and $\limsup _{k \rightarrow \infty}\langle x^*,x^*-x^{k}\rangle=\langle x^*,x^*-\bar y\rangle \leq 0$, by the definition of $x^*$. Therefore, by \cite[Lemma 4.4]{duvocelle2019}, also for a subsequence $(z^{k_j})$ it holds $z^{k_j}\to \bar y$. Then, $\lim_{k \rightarrow \infty}\langle x^*, x^*-x^{k+1}\rangle=\langle x^*,x^*-\bar y\rangle \leq 0$ and by Lemma \ref{lemma_xu02}, $\lim_{k\to\infty}\normsq{x^k-x^*}=0$. For more details, we refer to \cite{duvocelle2019}.
 \end{proof}

\paragraph*{Application of Lemma \ref{lemma_polyak1}}
An instance of how Lemma \ref{lemma_polyak1} can be used to prove convergence is given in \cite{kannan2012} where the authors propose a Nash equilibrium seeking algorithm via a Tikhonov regularization. The iterations, for each agent $i\in\mc I=\{1,\dots,N\}$, read as 
\begin{equation}\label{eq_tik}
x_{i}^{k+1}=\op{proj}_{\mc C_{i}}(x_{i}^{k}-\gamma_{i}^{k}(F_{i}(x^{k})+\epsilon_{i}^{k} x_{i}^{k}))
\end{equation}
where $\gamma_k=(\gamma_i^k)_{i=1}^N$ and $\epsilon_k=(\epsilon_i^k)_{i=1}^N$ are the step size and regularization sequences, respectively, and $\mc C_i$ is the local feasible set for each player $i$. Then, the following result holds.
\begin{thm}[Theorem 2.4, \cite{kannan2012}]\label{theo_kannan}
Suppose $F$ is monotone and $\ell$-Lipschitz continuous over a closed convex set $\mc C$ and let $(\gamma^k)_{k\in\NN}$ and $(\epsilon^k)_{k\in\NN}$ be such that
\begin{enumerate}
\item[1.] $\sum_{k=1}^{\infty} \gamma_{j}^{k} \epsilon_{j}^{k}=\infty$
\item[2.] $\lim _{k \rightarrow \infty} \frac{(\gamma_{\max }^{k})^{2}}{\gamma_{\min }^{k} \epsilon_{\min }^{k}}=0$
\item[3.] $\sum_{k=1}^{\infty}(\gamma_{j}^{k})^{2}<\infty$
\item[4.] $\sum_{k=1}^{\infty}(\epsilon_{j}^{k} \gamma_{j}^{k})^{2}<\infty$
\item[5.] $\lim _{k \rightarrow \infty} \frac{\epsilon_{\max }^{k-1}-\epsilon_{\min }^{k}}{\epsilon_{\min }^{k}(\gamma_{\min }^{k})^{2}}=0$
\item[6.] $\lim _{k \rightarrow \infty} \frac{\gamma_{\max }^{k} \epsilon_{\max }^{k}-\gamma_{\min }^{k} \epsilon_{\min }^{k}}{\gamma_{\min }^{k} \epsilon_{\min }^{k}}=0$
\item[7.] $\lim _{k \rightarrow \infty} \epsilon_{j}^{k}=0$ for all $j=1, \ldots, N$.
\end{enumerate}
Then, the sequence $(x^{k})_{k\in\NN}$ generated by \eqref{eq_tik} converges to a Nash equilibrium as $k \rightarrow \infty$.
\end{thm}
\begin{proof}
Since the classic Tikhonov relaxation, i.e., the iterative process where $y^{k+1}$ solves $\op{VI}(\mc X,F^k)$ and $F^k(y)=F(y)+\epsilon^ky$, is convergent \cite{tikhonov1963,bau2011}, the authors first show that \cite[Proposition 2.3]{kannan2012}
$$
\|z^{k+1}-y^{k}\|  \leq q_{k}\|z^{k}-y^{k-1}\|+\frac{q_{k} M \sqrt{N}(\epsilon_{\max }^{k-1}-\epsilon_{\min }^{k})}{\epsilon_{\min }^{k}},
$$
where $q_{k}^{2}  =(1-\gamma_{\min }^{k} \epsilon_{\min }^{k})^{2}+(\gamma_{\max }^{k})^{2} \ell^{2}+2(\gamma_{\max }^{k} \epsilon_{\max }^{k}-\gamma_{\min }^{k} \epsilon_{\min }^{k}) \ell$. Then, once they have $v^{k+1}\leq q^kv^k+\varepsilon^k$ with $v^k=\norm{x^k-y^{k-1}}$, they prove that there exists a $\bar k$ such that $q^k<1$ for all $k\geq\bar k$ (as in Remark \ref{remark_kannan}). Thus, Lemma \ref{lemma_polyak1} can be applied to conclude convergence.
 \end{proof}

\subsection{Application to Lyapunov decrease}\label{sec_lyap}
\paragraph*{Application of Corollary \ref{cor_mali1}}
In this application, we show how the convergence results can be used in combination with a Lyapunov function. Let us consider the classic gradient method \cite{polyak1987,bau2011}
\begin{equation}\label{eq_grad_method}
x^{k+1}=x^k-\gamma \nabla f(x^k)
\end{equation}
to find the minimum of a function $f:\mc X\subseteq\RR^n\to\RR$. 
\begin{thm}[Theorem 1.4.1, \cite{polyak1987}]\label{thoe_lyap}
Let $f(x)$ be differentiable on $\RR^{n}$ and bounded from below, i.e., $f(x) \geq f^{*}>-\infty $. Let $\nabla f$ be $\ell$-Lipschitz continuous and let $\gamma\in\left(0,\frac{2}{\ell}\right)$. Then, in method \eqref{eq_grad_method} the gradient tends to zero, i.e., $\lim _{k \rightarrow \infty} \nabla f(x^{k})=0$ and the function $f(x)$ monotonically decreases, i.e., $f(x^{k+1}) \leq f(x^{k})$.
\end{thm}
\begin{proof}
Using differentiability and Lipschitz continuity, we obtain
$$
f(x^{k+1}) \leq f(x^{k})-\gamma\left(1-\frac{\ell \gamma}{2}\right)\|\nabla f(x^{k})\|^{2},
$$
Then, applying Corollary \ref{cor_mali1} the claim follows.
 \end{proof}

Let us now note that the function $V(x)=f(x)-f^*$ is a Lyapunov function for the problem and that, in the proof of Theorem \ref{thoe_lyap}, we show that $V$ is decreasing along the discrete-time state trajectory $( x^k )_{k \in \mathbb{N}}$ (see also \cite[Section 2.2]{polyak1987}).

\subsection{Other applications}

Opial Lemma (Lemma \ref{lemma_opial}) is widely used for deterministic problems, in discrete \cite{boct2016,csetnek2019} and continuous time \cite{csetnek2019,bot2016}.
Moreover, another application of Lemma \ref{lemma_opial} can be found in \cite{bot2020gan} where the authors propose a forward-backward-forward algorithm \cite{tseng2000,bot2020} with an application to generative adversarial networks \cite{goodfellow2014,goodfellow2016}. 

Concerning inclusion problems, the interested reader may find an application of Lemma \ref{lemma_meno} in \cite{boct2016} while, for a different iterative scheme, Corollary \ref{cor_qin} is used in \cite{dadashi2019}; finally, an application of Lemma \ref{lemma_he} can be found in \cite{cholamjiak2018}.\\
Lemma \ref{lemma_he} is used also for a variational problem in \cite{he2013}, along with Lemma \ref{lemma_xu03}.
Moving to Nash equilibrium problems, Lemma \ref{lemma_polyak1} is used in \cite{kannan2012,lei2020br} while Lemma \ref{lemma_xu02} is used in \cite{lei2020cdc}.

\section{Applications of convergent stochastic sequences}\label{sec_app_stoc}

Similarly to the deterministic case, many applications of the lemmas for random sequences concern the study of convergent algorithms for stochastic variational inequalities.
Most of the literature relies on Robbins--Siegmund Lemma and on the monotone and Lipschitz properties of the operator (see Definitions \ref{def_mono} and \ref{def_lip} in Appendix \ref{appendix_op}). 

Before entering the details on how the lemmas are applied, we recall some preliminary notions on stochastic VIs (SVIs). For an extensive overview, we refer to \cite{shanbhag2013} and reference therein.
More precisely, we are interested in solving $\op{SVI}(\mc X,\FF)$, where $\FF$ is an expected value function $\FF(x)=\EE[f(x,\xi(\omega))]$, for some measurable mapping $f:\mc X\times\RR^d\to\RR$.
$\xi:\Omega\to\RR^d$ is a random variable and $(\Omega,\mc F,\PP)$ is the probability space. For brevity, $\xi$ is used to denote $\xi(\omega)$. Analogously to \eqref{eq_vi}, we say that $x^*\in \mc X$ solves the $\op{SVI}(\mc X,\FF)$ if
\begin{equation}\label{eq_svi}
\langle \FF(x^*),y-x^*\rangle\geq 0, \text{ for all } y\in \mc X,
\end{equation}
and analogously to the deterministic case, we can consider the general variational inequality as in \eqref{gen_vi}
$$\text { find } x^* \in \mc X \text { s.t. } \langle \FF(x^*), y- x^*\rangle+ g(y)-g(x^*) \geq 0  \text{ for all } y \in \mc X
$$
or a monotone inclusion as in \eqref{mono_incl}, i.e., find $x^*\in\mc X$ such that $0\in(\FF+\partial g)(x^*)$. We do not consider the case of stochastic functions $g$.

If the expected value of $f(x,\xi)$ is known, then the stochastic variational inequality can be solved with a standard solution technique for deterministic variational problems. However, the operator $\FF(x)$ is usually not directly accessible, due to the computational burden or lack of information on the distribution of the random variable. 
Therefore, in general the focus is on $\hat F(x,\xi)$, an approximation of $\FF(x)$, given some realizations $\xi$ of the random variable. 

There are two main methodologies available: stochastic approximation (SA) and sample average approximation (SAA). In the first case, $\FF(x)$ is approximated by considering only one (or a finite number of) realization, at each iteration, of the random variable $\xi$ \cite{koshal2013,robbins1951,iusem2017,cui2019,kushner2003}. In the second approach, instead, an infinite number of samples is taken at each iteration, then the approximation is given by the average over all the samples. The SAA scheme is mostly used to study existence of a solution \cite{shapiro2003, kleywegt2002,shapiro2008}, and it is essentially a deterministic problem, therefore, in this work, we focus on the SA scheme. Hence, let us formalize it. If only one sample is available, the expected value mapping is approximated at each iteration as
\begin{equation}\label{eq_sa}
F^{\textup{SA}}(x^k,\xi^k)=f(x^k,\xi^k),
\end{equation}
where $\xi^k$ is a realization of the random variable at time $k$.
This approach is computationally cheap, but it requires, in general, stronger assumptions on the monotonicity of the mappings involved.
Therefore, sometimes it is used in combination with the so-called variance reduction (VR).
In this case, at each iteration, the approximation of $\FF(x)$ has the form
\begin{equation}\label{eq_saa}
\begin{aligned}
F^{\textup{VR}}(x,\xi^k)=& \frac{1}{\mc N_k} \sum_{i=1}^{\mc N_k} f(x, \xi_i^k)\\
=&\frac{1}{\mc N_k} \sum_{i=1}^{\mc N_k} F^\textup{SA}(x, \xi_i^k) \text { for all } x \in \mc X.
\end{aligned}
\end{equation}
The batch size sequence $(\mc N_k)_{k\in\NN}$ determines the number of samples taken at each iteration. The sequence $\xi^k=(\xi_1^k,\dots,\xi_{\mc N_k}^k)$ is an i.i.d. random sequence. We suppose that $\mc N_k$ satisfy the following assumption any time the approximation scheme in \eqref{eq_saa} is used.

\begin{standass}\label{ass_batch}
The batch size sequence $(\mc N_k)_{k\geq 1}$ is such that, for some $c,k_0,a>0$, 
\begin{equation}\label{eq_batch}
\mc N_k\geq c(k+k_0)^{a+1}, \text{ for all } k\in\NN.
\end{equation}
\end{standass}
It follows from Standing Assumption \ref{ass_batch} that the batch size sequence is summable and this is fundamental to control the error committed in the approximation (see also Lemma \ref{lemma_variance}).

From now on, whenever we refer to an approximation without specifying the type, we use the symbol $\hat F$, while if it is one of the two schemes we explicitly use $F^{\textup{SA}}$  or $F^{\textup{VR}}$.

Since we study an approximation (independently on the scheme), let us indicate the stochastic error, that is, the distance between the expected value and its approximation, with
$$\epsilon^k=F^\textup{SA}(x^k,\xi^k)-\FF(x^k),$$
where $\xi^k$ is a (vector of) realization of the random variable at iteration $k\in\NN$.
Sometimes this term is also called martingale difference (Definition \ref{def_marti}) \cite{kushner2003,ljung2012}.



Standard assumptions on the stochastic error $\epsilon^k$ are that it has zero mean and bounded variance \cite{iusem2017,bot2020,franci2020fb,lei2018}.
\begin{standass}\label{ass_variance}
The stochastic error is such that
$$\EEk{\epsilon^k}=0 \text{ a.s., for all } k\in\NN.$$
Moreover, for all $x\in \mc X$ and $p\geq 1$ let
$$
s_{p}(x) = \mathbb{E}\left[\| F^\textup{SA}(x, \xi)-\FF(x)\|^{p}\right]^{\frac{1}{p}}.
$$
There exist $p \geq 2$, $\sigma_{0} \geq 0$ and a measurable locally bounded function $\sigma : \op{SOL}(\mc X,\FF) \rightarrow \mathbb{R} $ such that for all $ x \in \mc X $ and all $x^{*} \in \op{SOL}(\mc X,\FF)$
\begin{equation}\label{eq_variance}
s_{p}(x) \leq \sigma\left(x^{*}\right)+\sigma_{0}\left\|x-x^{*}\right\|.
\end{equation}
\end{standass}
In the following, for ease of reading, we use a stronger condition than that in  \eqref{eq_variance}, namely,
\begin{equation}\label{variance}
\EEx{\norm{ F^\textup{SA}(x, \xi)-\FF(x)}^{p}}^{\frac{1}{p}}\leq \sigma.
\end{equation}
While Condition \eqref{eq_variance} is known in the literature as variance reduction, the stronger formulation \eqref{variance} is called uniform bounded variance. 
Assumption \eqref{eq_variance} is more realistic in those cases where the feasible set $\mc X$ is unbounded, and it is always satisfied when the mapping $f$ is Carath\'eodory and random Lipschitz continuous \cite[Example 1]{bot2020}.
Since in many realistic examples the feasible set is bounded, we use \eqref{variance} as a variance control assumption. We also remark that many of the following results hold also in the more general case given by Assumption \ref{ass_variance} and using the $L_p$ norm for any $p\geq2$. We refer to \cite{iusem2017, bot2020} and references therein for a more detailed insight on this general case.

\begin{rem}
When we use the SA scheme with variance reduction, the following relation between the stochastic error and the batch size sequence holds (see Lemma \ref{lemma_variance}):
for all $k\geq 0$, $c>0$, $\sigma$ as in \eqref{variance} and $\mc N_k$ as in \eqref{eq_batch}, 
\begin{equation}\label{eq_error}
\EE\left[\norm{\epsilon^k}^2|\mc F_k\right]\leq\frac{c\sigma^2}{\mc N_k} a.s..
\end{equation}
Essentially, Lemma \ref{lemma_variance} says that the second moment of the error decreases with the increasing number of samples of the random variable.\\
Sometimes more general results hold for the bound in \eqref{eq_error} (see, e.g., \cite{iusem2017,bot2020,franci2020fb}) but they lie outside the scopes of the survey.
 \end{rem}

We are now ready to describe how the lemmas are used. The first applications that we present are all related to Robbins--Siegmund Lemma (Lemma \ref{lemma_rs}). We differentiate the applications on how the negative term $-\theta_k$ is exploited (Remark \ref{remark_theta}). Nonetheless, in all of them, the summability of the term is used differently to obtain convergence. For the first application we also provide a scheme (inspired by Figure \ref{fig_app_det}) of the step that should be taken to use a lemma for sequences of random numbers (Figure \ref{fig_app_stoc}). The section ends with an application of Lemma \ref{lemma_fake_rs}. As the reader may note, the forthcoming applications rely on the existence of a martingale, associated to the process, that the lemmas prove to be convergent \cite{benaim1996}.

\begin{figure}[t]
\centering
 \begin{tikzpicture}[scale=1,every node/.style={circle,draw=black,scale=.5,minimum size=.8cm}]
 \LARGE
\draw[rounded corners=2ex,draw=black,dashed,thick,fill=white,fill=none] (-2.5,3) rectangle (2.5,5.8);
\draw[rounded corners=2ex,draw=black,dashed,thick,fill=white,fill=none] (-2.5,2.7) rectangle (2.5,-2.1);

\node[rectangle, rounded corners=.5ex, minimum size=1.5cm, fill=white](m1) at (0,5){$\begin{array}{c}\text{iterative process } (x^k)_{k\in\NN}\\ \text{generated by \eqref{eq_fbf}}\end{array}$};
\node[rectangle, rounded corners=.5ex, minimum size=1.5cm, fill=white](m2) at (0,3.5){define $v^k:=\normsq{x^k-x^*}$};
\node[rectangle, rounded corners=.5ex, minimum size=1.5cm, fill=white](m3) at (0,1.75){$\begin{array}{c}\text{find inequality as in \eqref{eq_tab_fbf}:}\\v^{k+1}\leq C^kv^k+\varepsilon^k-\theta^k\end{array}$};
\node[rectangle, rounded corners=.5ex, minimum size=1.5cm, fill=white](m4) at (0,0){see Table \ref{table_lemmi_stoc}};
\node[rectangle, rounded corners=.5ex, minimum size=1.5cm, fill=white](m5) at (0,-1.5){apply Lemma \ref{lemma_rs}};
\node[rectangle, rounded corners=.5ex, minimum size=1.5cm, fill=white](m6) at (0,-3){$\begin{array}{c}x^k\to x^* \text{ and} \sum_{k=1}^\infty\theta^k<\infty\\
\text{imply } r_{\alpha_k}(x^k) \to 0\end{array}$};

\foreach \from/\to in
{m1/m2,m2/m3,m3/m4,m4/m5,m5/m6}
\draw[-to,thick] (\from) -- (\to); 

\node[draw=none,fill=none,rotate=90] at (-3,0.5) {analysis};
\node[draw=none,fill=none,rotate=90] at (-3,4.25) {design};
\node[draw=none,fill=none,rotate=90] at (-3,-3) {Theorem \ref{theo_bot}};

 \end{tikzpicture}
 \caption{Schematic representation of Theorem \ref{theo_bot}. First, a distance like function is defined, to obtain a quasi-Fejer inequality. Since the inequality correspond to Robbins-Siegmund Lemma, Lemma \ref{lemma_rs} can be applied. Not only convergence is proved, but also the fact that the negative term is summable contributes to showing that, asymptotically, a solution is reached. For general guidelines, see also Figure \ref{fig_app_det}.}\label{fig_app_stoc}
 \end{figure}
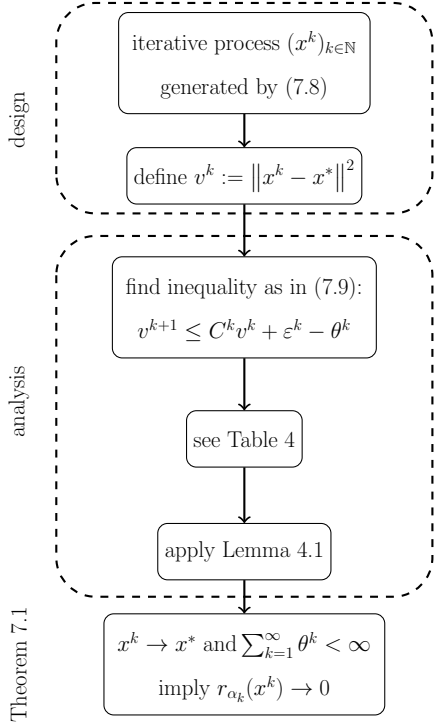
 
\subsection{Applications of Robbins-Siegmund Lemma}\label{sec_app_rs}
\paragraph*{Application of Lemma \ref{lemma_rs} with residual}\label{app_SFBF}
In \cite{bot2020,iusem2017}, the residual ($r_{\alpha}(x)$) is used to prove convergence (see Appendix \ref{appendix} for a definition and Remark \ref{remark_res}).
Specifically, in \cite{bot2020}, the authors formulate a \textit{stochastic forward-backward-forward algorithm}, inspired by \cite{tseng2000}, given by the following updating rule:
\begin{equation}\label{eq_fbf}
\begin{aligned}
y^k&=\op{proj}_{\mc X}(x^k-\alpha^k F^{\textup{VR}}(x^k,\xi^k))\\
x^{k+1}&=y^k+\alpha^k(F^{\textup{VR}}(x^k,\xi^k)-F^{\textup{VR}}(y^k,\eta^k))
\end{aligned}
\end{equation}
where $\xi_n$ and $\eta^k$ are i.i.d. random variables and $F^{\textup{VR}}$ is as in \eqref{eq_saa}. 

Robbins--Siegmund Lemma is used for concluding that the sequence $(x^k)_{k \in \NN}$ converges a.s. to a solution of the SVI in \eqref{eq_svi}, proving that the residual goes to zero (Remark \ref{remark_res}). A scheme of the proof and of how Lemma \ref{lemma_rs} is used can be found in Figure \ref{fig_app_stoc}.

\begin{thm}[Theorem 1, \cite{bot2020}]\label{theo_bot}
Let $f$ be a Carath\'eodory map and let $\FF$ be pseudomonotone and $\ell$-Lipschitz continuous with $\ell>0$. Let $0< \inf _{k \geq 0} \alpha_{k} \leq \alpha_k\leq \sup _{k \geq 1} \alpha_{k}<\frac{1}{\sqrt{2} \ell}$.
Then, the sequence $(x^k)_{k \in \NN}$ generated by \eqref{eq_fbf} converges a.s. to a limit random variable $x^*\in \op{SOL}(\mc X,\FF)$, and $\lim_{k\to\infty} \EE[r_{\alpha_k} (x^k)^2] = 0$.
\end{thm}
\begin{proof}
Using monotonicity and Lipschitz continuity of the mapping $\FF$ and the definition of the algorithm in \eqref{eq_fbf}, it is possible to prove a recursion \cite[Lemma 5]{bot2020} that, taking the expected value \cite[Proposition 1]{bot2020} and using some bounds on the stochastic error \cite[Lemma 6]{bot2020} (see also Lemma \ref{lemma_variance}), reads as
\begin{equation}\label{eq_tab_fbf}
\mathbb{E}[\|x^{k+1}-x^*\|^{2} | \mathcal{F}_k] \leq\|x^{k}-x^*\|^{2}-\frac{\rho_{k}}{2} r_{\alpha_{k}}(x^{k})^{2}+\frac{\kappa_{k}\sigma^2}{\mc N_k},
\end{equation}
where $\rho_k=1-2\ell^2\alpha_k^2$ and $\kappa_k$ is a constant that depends on the Lipschitz constant and on the step size. To use Lemma \ref{lemma_rs}, let $v^{k}=\|x^{k}-x^*\|^{2}$, $\theta_{k} =\frac{\rho_{k}}{2} r_{\alpha_{k}}(x^{k})^{2}$ and $\varepsilon^k =\frac{\kappa_{k} \sigma_{0}^{2}}{\mc N_k}$. Then the claim follows using the fact that $\theta^k$ is summable and therefore the residual tends to zero.
 \end{proof}


The use of the residual to prove convergence to the solution of the SVI in \eqref{eq_svi} was previously introduced in \cite{iusem2017} where the authors propose a \textit{stochastic extragradient method} inspired by \cite{korpelevich1976}. The iterations are given by
\begin{equation}\label{iusem_algo}
\begin{aligned}
&z_i^k=\op{proj}_{\mc X}\left[x_i^k-\alpha^k F^{\textup{VR}}_i(x^k,\xi_{i}^k)\right]\\
&x_i^{k+1}=\op{proj}_{\mc X}\left[x_i^k-\alpha^k F^{\textup{VR}}_i(z^k,\eta_{i}^k)\right],
\end{aligned}
\end{equation}
where $(\xi^k)_{k\in\NN}$ and $(\eta^k)_{k\in\NN}$ are i.i.d. samples of the random variable such that $(\xi^k)_{k\in\NN}$ and $(\eta^k)_{k\in\NN}$ are independent of each other.
They have assumptions on the parameters similar to \cite{bot2020} and the variance reduction hypothesis.
The main result is the asymptotic convergence of the algorithm.
\begin{thm}[Theorem 3.18, \cite{iusem2017}]\label{theo_iusem}
Let $f$ be a Carath\'eodory map such that $\EE[\norm{f(x,\xi)}]<\infty$. Let $\FF$ be pseudomonotone and $\ell$-Lipschitz continuous mapping. Let $0<\inf _{k \in \mathbb{N}} \alpha_{k} \leq \alpha_k\leq\sup _{k \in \mathbb{N}} \alpha_{k}<\frac{1}{\sqrt{6} \ell}$ for all $k\in\NN$. Then, the sequence $(x^k)_{k \in \NN}$ generated by (\ref{iusem_algo}) is bounded, $\lim_{k\to\infty}d(x^k,\op{SOL}(\mc X,\FF))=0$ and $r_{\alpha_k}(x^k)$ converges to 0. In particular, any cluster point of $(x^k)_{k\in\NN}$ belongs to $\op{SOL}(\mc X,\FF)$.
\end{thm}
\begin{proof}
Given the properties of the operator $\FF$ \cite[Lemma 3.11]{iusem2017} and of the parameters involved \cite[Lemma 3.12]{iusem2017}, it holds that \cite[Proposition 3.15]{iusem2017}
$$
\begin{aligned}
\mathbb{E}\left[\left\|x^{k+1}-x^{*}\right\|^{2} | \mathcal{F}_{k}\right]& \leq\left(1+\frac{C(\sigma^2,x^*)}{\mathcal{N}_{k}}\right)\left\|x^{k}-x^{*}\right\|^{2}\\
&-\frac{\rho_{k}}{2} r_{\alpha_{k}}(x^{k})^{2}+\frac{C(\sigma^2,x^*)}{\mathcal{N}_{k}}
\end{aligned}
$$
where $\rho_k=(1-6\ell^2\alpha_k^2)$, $C(\sigma^2,x^*)$ is a bounded quantity that depends on the solution $x^*$ and on the variance \cite[Remark 3.17]{iusem2017}, and $r_{\alpha_k}(x^k)$ is the residual of $x^k$. Then the claim follows as in the proof of Theorem \ref{theo_bot}, using Robbins--Siegmund Lemma.
 \end{proof}

\paragraph*{Application of Lemma \ref{lemma_rs} with strict monotonicity}
Robbins--Siegmund Lemma can also be used to prove the convergence of the \textit{partially coordinated iterative proximal point scheme} to a Nash equilibrium \cite{koshal2013}. The possibility to reach a Nash equilibrium in a game theoretic framework is related to the fact that they can be obtained as the solution of a suitable (S)VI \cite[Proposition 1.4.2]{facchinei2007}. The updating rule of the algorithm is given by:
\begin{equation}\label{koshal_algo}
x^{k+1}=\op{proj}_{\mc X}[x^k-\alpha_k(\hat F(x^k,\xi^k)+\mu^k(x^k-x^{k-1}))]
\end{equation}
where $\alpha_k\in\RR^n$ and $\mu^k\in\RR^n$ are the step size and the centering parameters, respectively, and $n$ is the number of agents in the Nash equilibrium problem. 

\begin{prop}[Proposition 3, \cite{koshal2013}]\label{prop_koshal}
Let $\FF:\mc X\to\RR^n$ be strictly monotone and $\ell$-Lipschitz continuous over $\mc X$. Let the following conditions hold: 
\begin{enumerate}
\item $\alpha_{k, \max } \mu_{k, \max } \leq (1+2 \alpha_{k, \max }^{2} \ell^{2})\cdot$ $ \alpha_{k-1, \min } \mu_{k-1, \min }$ for all $k \in\NN$;
\item $\lim _{k \rightarrow \infty} \frac{\alpha_{k, \max }^{2} \mu_{k, \min }^{2}}{\alpha_{k, \min } \mu_{k, \min }}=c$ with $c \in\left[0, \frac{1}{2}\right)$;
\item $\sum_{k=0}^{\infty} \alpha_{k, i}=\infty \text { and } \sum_{k=0}^{\infty} \alpha_{k, i}^{2}<\infty$ for all $i\leq n$; 
\item $\sum_{k=0}^{\infty}\left(\alpha_{k, \max }-\alpha_{k, \min }\right)<\infty$;
\item $\sum_{k=0}^{\infty} \alpha_{k, \max }^{2} \mathbb{E}[\|\epsilon_{k}\|^{2} | \mc{F}_k]<\infty$ a.s..
\end{enumerate}
Then, the sequence $(x^k)_{k \in \NN}$ generated by (\ref{koshal_algo}) converges a.s. to a solution of $\op{SVI}(\mc X,\FF)$.
\end{prop}

\begin{proof}
Using the nonexpansiveness of the projection and some norm properties, one can obtain
$$
\begin{aligned} 
\mathbb{E} & [ \|x^{k+1}-x^{*} \|^{2} | \mathcal{F}_{k} ]+\alpha_{k, \min } \mu_{k, \min } \|x^{k}-x^{*} \|^{2} \\ 
\leq &(1+\delta_{k})(\|x^{k}-x^{*}\|^{2}+\alpha_{k-1, \min} \mu_{k-1, \min }\|x^{k-1}-x^{*}\|^{2}) \\ 
&-\alpha_{k, \min } \mu_{k, \min }(1-d) \|x^{k}-x^{k-1} \|^{2}+\alpha_{k, \max }^{2} \mathbb{E} [ \|\epsilon_{k} \|^{2} | \mathcal{F}_{k} ] \\ 
&-2 \alpha_{k, \min } (x^{k}-x^{*} )^{T} (\FF(x^{k})-\FF(x^{*} ) ),
\end{aligned}
$$
where $x^*\in\op{SOL}(\mc X,\FF)$, $\delta^k$ depends on the Lipschitz constant and on the step sizes and $d\in (0,1)$.
To apply Lemma \ref{lemma_rs}, let $v^k= \|x^{k}-x^{*} \|^{2}+\alpha_{k-1, \min} \mu_{k-1, \min } \|x^{k-1}-x^{*} \|^{2}$, $\theta^k=\alpha_{k, \min } \mu_{k, \min }(1-d) \|x^{k}-x^{k-1} \|^{2}+2 \alpha_{k, \min } \langle x^{k}-x^{*},\FF(x^{k})-\FF(x^{*})\rangle$ and $\varepsilon^k=\alpha_{k, \max }^{2} \mathbb{E} [ \|\epsilon_{k} \|^{2} | \mathcal{F}_{k} ]$. Then, it folows that $(x^k)_{k\in\NN}$ is bounded and has a cluster point $\bar x$. Since $\theta^k$ is summable, $\langle x^{k}-x^{*},\FF(x^{k})-\FF(x^{*} ) \rangle\to0$ and taking the limit for $k\to\infty$, $\langle\bar x-x^{*} ,\FF(\bar x)-\FF(x^{*} ) \rangle=0$. Since the mapping is strictly monotone (Definition \ref{def_mono}) and the solution set is not empty (Standing Assumption \ref{ass_sol}), there is only one solution $x^*$ \cite[Theorem 2.3.3]{facchinei2007}, and we have that $\bar x=x^*$.
 \end{proof}


\subsection{Applications of Robbins-Siegmund Lemma to specific problems}

\paragraph*{Application of Lemma \ref{lemma_rs} to model predictive control}
An interesting application of Robbins-Siegmund Lemma is provided in \cite{lee2015}, where the authors propose the \textit{gossip-based random projections} (GRP) algorithm for distributed robust model predictive control (MPC). In their problem, $m$ private facilities aim at finding an optimal control law $u=\op{col}(u(1),\dots,u(T))$ of a dynamic system such that the resulting trajectory $x(t)$, for $t = 1, \dots, T$, remains close to the locally known facilities and the terminal state $x(T)$ is inside some uncertain box with minimum control effort. Formally, the distributed MPC optimization problem is given by 
$$\begin{aligned}
\min_u \quad& f(u)=\sum_{i=1}^m f_i(u)\\
s.t. \quad& x(t+1)=Ax(t)+Bu(t)\\
& u\in\mc X\\
& \max _{\ell=1,2,3,4}\left\{(a_{\ell}+c_{\ell}\right)^{\top} x(T)-b_{\ell}\} \leq 0
\end{aligned}$$
where $u\in\mc X$ represent the uncertain input constraint and the last inequality describe the random terminal constraint ($\mc T$ from now on). We refer to \cite{lee2015} for a specific choice of $f$, $A$ and $B$. 
The algorithm is based on random projections and a gossip communication protocol inspired by \cite{boyd2006}. At each time $k$, only an agent $I_k\in\mc I$ and its neighbor $J_k\in\mc I$ wake up. They draw a sample of one of the linear inequality terminal constraints and they update their estimate while the other agents do nothing. Then, they project their current iterate on the selected constraint $\mc T$ and on $\mc X$.
The GRP algorithm reads, for $i\in\{I_k,J_k\}$, as 
\begin{equation}\label{eq_opdyn}
\begin{aligned}
w_{i}^k &=\frac{u_{I_{k}}^{k-1}+u_{J_{k}}^{k-1}}{2} \\
u_{i}^k &=\op{proj}_{\mathcal{X}\cap\mc T}[w_{i}^k-\alpha_{i}^k \nabla f_{i}(w_{i}^k)]
\end{aligned}
\end{equation}
where $\{\alpha_i^k\}_{k\in\NN}$ is the step size sequence, defined such that $\alpha_{i}^k=1 / \Gamma_{i}^k$, where $\Gamma_i^k$ is the number of updates $i$ has performed until time $k$.\\
Let us denote $z_i^k=\op{proj}_{\mc X}(w_{i}^k)$ and let $\tilde w=\frac{1}{m}\sum_{i=1}^mw_i$ (analogously $\tilde z$). Then, convergence of the algorithm is proven as follows.
\begin{prop}\label{prop_mpc}\cite[Proposition 1]{lee2015}
Let the communication graph be connected and let the set $\mc X$ be closed and convex. Let the functions $f_i:\RR^d\to\RR$ be convex and differentiable and let their gradients $\nabla f_i$ be Lipschitz continuous and bounded over $\mc X$, i.e., $\|\nabla f_{i}(u)\| \leq G_{f}$ for all $u \in \mathcal{X}$ and all $i \in \mc I$.
Let $\mc X^*$ be a nonempty optimal set. Then, the sequences $\{u_{i}^k\}_{k\in\NN}$, $i \in \mc I$, generated by \eqref{eq_opdyn} converge to some random point $u^*\in\mathcal{X}^{*}$ a.s., i.e., $\lim _{k \rightarrow \infty} u_{i}^k=u^*$ a.s. for all $i \in \mc I$.
\end{prop}
\begin{proof}
First, it is possible to show (by using Lemma \ref{lemma_rs}) that $w_{i}^k$ approaches $\mc X$ \cite[Lemma 3]{lee2015}, and that any two sequences $\{w_i ^k\}_{k\in\NN}$ and $\{w_j^k\}_{k\in\NN}$ have the same limit points a.s. \cite[Lemma 4]{lee2015}. Then, it holds by \cite[Lemma 2]{lee2015} and some properties of the projection, of the norms and of the mappings involved, that 
\begin{equation}\label{eq_lee}
\begin{aligned}
\sum_{i=1}^{m} \EE&[\|u_{i}^k-u^{*}\|^{2} \mid \mathcal{F}_{k-1}] \\
\leq &\left(1+\frac{a_{4}}{k^{\frac{3}{2}}-q}\right) \sum_{j=1}^{m}\left\|u_{j}^{k-1}-u^{*}\right\|^{2} \\
&-\frac{2}{k} \EE\left[f(\tilde{z}^k)-f^{*}\right) \mid \mathcal{F}_{k-1}] \\
&+\frac{4 G_{f}}{k} \sum_{i=1}^{m} \EE\left[\left\|w_{i}^k-\tilde{w}^k\right\| \mid \mathcal{F}_{k-1}\right]+\frac{a_2 m}{k^{\frac{3}{2}}-q}
\end{aligned}
\end{equation}
where $q$, $a_2$ and $a_4$ are constants. Equation \eqref{eq_lee} satisfies the conditions from Lemma \ref{lemma_rs} \cite[Lemma 4]{lee2015}. Hence, $\{\normsq{u_i^k-u^*}\}_{k\in\NN}$ is convergent a.s. for any $i\in \mc I$ and $u^*\in\mc X^*$. Moreover, by Lemma \ref{lemma_rs}, $\liminf _{k \rightarrow \infty}f(\tilde{z}^k)-f^{*}=0$ a.s. and by \cite[Lemma 3]{lee2015} $\lim _{k \rightarrow \infty}\left\|w_{i}^k-z_{i}^k\right\|=0$ for all $i \in \mc I$ a.s.. Hence also the sequences $\{\normsq{w_{i}^k-u^*}\}_{k\in\NN}$ and $\{\normsq{z_i^k-u^*}\}_{k\in\NN}$ and their averages $\{\normsq{\tilde w^k-u^*}\}_{k\in\NN}$ and $\{\normsq{\tilde z^k-u^*}\}_{k\in\NN}$ are convergent and the sequences $\{\tilde w^k\}_{k\in\NN}$ and $\{\tilde z^k\}_{k\in\NN}$ are bounded and have an accumulation point in $\mc X^*$. Since $\liminf _{k \rightarrow \infty}\left\|w_{i}^k-\tilde{w}^k\right\|=0$ for all $i \in \mc I$ a.s., it follows that $\lim _{k \rightarrow \infty}\left\|w_{i}^k-u^{*}\right\|=0$ for all $i \in \mc I$ a.s.. 
Finally, by \cite[Lemma 3]{lee2015}, $\lim _{k \rightarrow \infty}\left\|u_{i}^k-w_{i}^k\right\|=0$
for all $i \in \mc I$ a.s., which leads to, $\lim _{k \rightarrow \infty} u_{i}^k=u^*$ for all $i \in \mc I$ a.s.. 
\end{proof}

\paragraph*{Application of Corollary \ref{cor_duflo} to the Law of Large Numbers}
Remarkably, the convergence results for sequences can be used also for others scopes beside convergence of an algorithm. This is the case of Corollary \ref{cor_duflo} which is used to prove the Law of Large Numbers. To introduce this application, let us define the notion of increasing process associated to a martingale \cite{duflo2013,chung1990}. 

Let $y_k$ be a martingale such that $\EE[y_k^2]<\infty$ for all $k\in\NN$. Then, its increasing process is the sequence $(\langle y\rangle_{k})_{k\in\NN}$ defined by $\langle y\rangle_{0}=0$ and 
$\langle y\rangle_{k+1}-\langle y\rangle_{k}=\EE[y_{k+1}^{2}-y_{k}^{2} \mid \mathcal{F}_{k}]$ 
\cite[Proposition 1.3.7]{duflo2013}. For instance, if $(z_k)_{k\in\NN}$ is a sequence of identically distributed random variables with mean $\mu$ and variance $\sigma^{2}$ then $y_{k}=z_{1}+\ldots+z_{k}-k \mu$ satisfies $\EE[y_k^2]<\infty$ and is such that $\langle y\rangle_{k}=k \sigma^{2} .$ Then, the following generalization of the Law of Large Numbers for martingales holds.
\begin{thm}[Theorem 1.3.15, \cite{duflo2013}]\label{lln}
Let $(y_k)_{k\in\NN}$ be a martingale such that $\EE[y_k^2]<\infty$ for all $k\in\NN$ and let $(\langle y\rangle_k)_{k\in\NN}$ be its increasing process. Then, a.s. $\lim_{k\to\infty}\frac{y_k}{\langle y\rangle_k}= 0$.
\end{thm}
\begin{proof}
To apply Corollary \ref{cor_duflo}, let $v_k=y_k^2$, $\varepsilon_k=\langle y\rangle_{k+1}-\langle y\rangle_k$ and $a_k=\langle y\rangle_{k+1}(\op{ln}(\langle y\rangle_{k+1}))^{1+\gamma}$. Then, if $\langle y\rangle_{k_0}>1$, $\sum_{k=k_0}^\infty a_k^{-1}\varepsilon_k<\infty$, $\lim_{k\to\infty}a_k^{-1}v_k= 0$ and the claim follows.
 \end{proof}
\begin{rem}
We note that, even if it does not involve an algorithm, Theorem \ref{lln} is in agreement with Figures \ref{fig_app_det} and \ref{fig_app_stoc}. In fact, we have an iterative process ($(y^k)_{k\in\NN}$), to which we assign a sequence ($v^k=y_k^2$) to obtain an inequality as in Corollary \ref{cor_duflo}, retrieved from Table \ref{table_lemmi_stoc}, and prove the result.
\end{rem}

\subsection{Applications of Lemma \ref{lemma_fake_rs}}
\paragraph*{Application of Lemma \ref{lemma_fake_rs} to a variational problem}
In \cite{yousefian2014}, a \textit{smoothing extragradient scheme} with stochastic approximation, similar to \eqref{iusem_algo}, is proposed. The iterations read as 
\begin{equation}\label{eq_smooth_eg}
\begin{aligned}
y^{k+1}&=\op{proj}_{\mc X}[x^{k}-\gamma_{k} \hat F(x^{k}+w^{k}, \eta^{k})] \\
x^{k+1}&=\op{proj}_{\mc X}[x^{k}-\gamma_{k} \hat F(y^{k+1}+z^{k}, \xi^{k})]
\end{aligned}
\end{equation}
where $(\gamma_k)_{k\in\NN}$ is the step size sequence, $\eta^k$ and $\xi^k$ are i.i.d. samples of the random variable and the sequences $(w^k)_{k\in\NN}$ and $(z^k)_{k\in\NN}$ are also i.i.d. random variables drawn from an uniform distribution on $\left[-\frac{\delta_k}{2},\frac{\delta_k}{2}\right]$ where $\delta^k$ is the smoothing sequence. Let $\mc X_\delta=\mc X+C_n(0,\delta)$ where $C_n$ is a $n$-dimensional cube centered at the origin and $\delta$ is an upper bound on $\delta_k$. Then, the following holds.
\begin{thm}[Theorem 2, \cite{yousefian2014}]\label{theo_yousef}
Let $\mc X\subseteq\RR^n$ be closed, convex, and $M$-bounded.
Let $\FF$ be strictly monotone over $\mc X$ and bounded on $\mc X^\delta$ for some $C>0$. Suppose the sequence $(\gamma_k)_{k\in\NN}$ is such that $\sum_{k=1}^\infty \gamma_k^2<\infty$ and $\sum_{k=1}^\infty\gamma^k=\infty$ that the sequence $\delta_k$ is diminishing according to \cite[Equation 19]{yousefian2014}. Then, the sequence $(x_k)_{k\in\NN}$ generated by \eqref{eq_smooth_eg} converges to a solution $x^*$ of $\op{SVI}(\mc X,\FF)$ a.s. as $k \rightarrow \infty$.
\end{thm}
\begin{proof}
Using the assumptions, from \cite[Lemma 4]{yousefian2014} the following inequality holds:
$$
\mathrm{E}[\|x^{k+1}- x^*\|^{2} \mid\mathcal{F}_k] \leq\left(1-\frac{\alpha \gamma_k}{M}\right)\|x^k-x^*\|^{2}+q \gamma_k^{2}$$
where $q$ is a constant that depends on the constant $C$ and on the variance of the stochastic error and $M$ is the bound on the set $\mc X$. Then, convergence follows applying Lemma \ref{lemma_fake_rs} to $v^k=\normsq{x^k-x^*}$, $\epsilon^k=q \gamma_k^{2}$ and $\delta^k=\frac{\alpha \gamma_k}{M}$.
 \end{proof}
\paragraph*{Application of Lemma \ref{lemma_fake_rs} to opinion dynamics}

The fact the Lemma \ref{lemma_fake_rs} provides convergence to zero is used in \cite{shi2013} to prove agreement in an opinion dynamics model. Let us consider the spreading of true or false information over a communication network or faults propagations in large scale control systems. In these models, there are $n$ nodes and each of them (node $i$) activates with a probability $1/n$, then it picks a neighbor $j$ with probability $a_{ij}$. The probabilities are collected in the interaction matrix $A=[a_{ij}]_{i,j=1}^n$.
The dynamics is described as follows, given $\alpha+\beta+\gamma=1$:
\begin{itemize}
\item[(i)] (Attraction) With probability $\alpha$, node $i$ updates its opinion toward that of its neighbor $j$,
$$
x_{i}^{k+1}=x_{i}^{k}+T_{k}(x_{j}^{k}-x_{i}^{k}),
$$
where $0<T_{k} \leq 1$ is the trust level;
\item[(ii)] (Neglect) With probability $\beta$, node $i$ keeps its own opinion,
$$
x_{i}^{k+1}=x_{i}^{k};
$$
\item[(iii)] (Repulsion) With probability $\gamma$, node $i$ moves away from $j$, i.e., it updates with a negative coefficient,
$$
x_{i}^{k+1}=x_{i}^{k}-S_{k}(x_{j}^{k}-x_{i}^{k})
$$
where $S_{k}>0$.
\end{itemize}
The authors propose in \cite{shi2013} some conditions on the quantities involved under which agreement or disagreement can be obtained with a time-invariant trust level. As a measure of disagreement, let $L^k=\sum_{i=1}^{n} \normsq{x_{i}^{k}-x_{\text {ave}}}$, where $x_{\text {ave }}=\sum_{i=1}^{n} \frac{x_{i}^0}{n}$ is the average of the initial values. Then, the following result holds.
\begin{thm}\label{theo_op_dyn}\cite[Theorem 5]{shi2013} 
Let the communication graph be weakly connected and suppose that the updates are symmetric. Let $\lambda_{2}^{*}$ be the second smallest eigenvalue of $D-(A+A^\top)$ with $D=\operatorname{diag}\left(d_{1} \ldots d_{n}\right), d_{i}=\sum_{j=1}^{n}\left(a_{i j}+a_{j i}\right)$.
Let $T_{k} \equiv T^*\in[0,1]$ and $S_{k} \equiv S^*>0$. Then
$$
D^{*}=S^*\left(1+S^*\right) \gamma-T^*\left(1-T^*\right) \alpha
$$
is a critical convergence measure regarding the state convergence of the considered network. Specifically, if $D_{*}<0$, then global agreement convergence is achieved, i.e., $\lim _{k \rightarrow \infty} \EE[L^k]=0$ a.s..
\end{thm}
\begin{proof}
Given some preliminary results \cite[Proposition 3]{shi2013}, it holds that
$$
\EE[L^{k+1} \mid ^k] \leq\left(1-\frac{2}{n} D^{*} \lambda_{2}^{*}\right) L^k
$$
then $\lim_{k\to\infty}\EE[L^k]= 0$ a.s. by Lemma \ref{lemma_fake_rs}.
\end{proof}

\subsection{Other applications}

Other applications of Robbins-Siegmund Lemma (Lemma \ref{lemma_rs}) can be found in \cite{cui2019, iusem2019, iusem2017, jiang2008, wang2015,kannan2019,alacaoglu2021} for variational problems and monotone inclusions. 
Concerning Nash equilibrium problems, it is used in \cite{franci2020fb, koshal2013, franci2021}. In the specific case of generative adversarial networks, Lemma \ref{lemma_rs} is used in \cite{franci2020gan,franci2020gen}. For an application of this stochastic result to a deterministic problem, we refer to \cite{koshal2016}. Regarding dynamic systems and Lyapunov analysis, other utilizations of Robbins-Siegmund Lemma are in \cite[Section 3]{benaim1996}, \cite[Section I.1]{ljung2012} and \cite{bharath1999}.\\
For other applications of Lemma \ref{lemma_fake_rs} instead, the interested reader may refer to \cite{koshal2013,yousefian2017,iusem2019,kannan2019,lei2020br,kannan2019,cui2019}.\\

\section{Application of convergent sequences with variable metric}\label{sec_app_vm}
The variable metric framework is not studied as much as the classic setup. Therefore, we propose only the following application. For other references see \cite{combettes2013,vu2013,cui2019vm}.

\paragraph*{Application of Proposition \ref{prop_vm} and Theorem \ref{theo_comb_var}} 
A study of the \textit{forward-backward-forward} algorithm \cite{bot2020,tseng2000} with variable metric is considered in \cite{vu2013}. There, the authors consider the splitting of a sum of a maximally monotone operator $A$ and a monotone, Lipschitz continuous operator $B$ of the form \eqref{mono_incl} and they suppose that multiple errors (sequences $a^k$, $b^k$ and $c^k$) can be made at each iteration. Formally, their proposed algorithm reads as
\begin{equation}\label{eq_fbf_vm}
\begin{aligned}
&y^k=x^k-\gamma^kW_k (Bx^k+a^k)\\
&v^k=(\op{Id}+\gamma^kW_kA)^{-1}y^k+b^k\\
&u^k=y^k+\gamma^kW_k(Bv^k+c^k)\\
&x^{k+1}=x^k-y^k+u^k
\end{aligned}
\end{equation}
where $(W_k)_{k\in\NN}$ is the sequence of operators used to induce the metric. Then, they prove the following convergence result.
\begin{thm}[Theorem 3.1, \cite{vu2013}]\label{theo_vu}
Let $\beta,\ell>0$, let $ (\eta^k )_{k\in\NN}$ be a nonnegative sequence such that $\sum_{k=1}^\infty\eta^k<\infty$ and let $(W_k )_{k\in\NN}$ be a sequence in $\mc P_\beta$ such that 
$$\mu=\sup _{k\in\NN} \|W_k \|<+\infty$$ 
$$(1+\eta^k ) W_{k+1} \succeq W_k, \text{ for all } k\in\NN.$$
Let $A: \RR^n  \rightrightarrows \RR^n$ be maximally monotone, let $B: \RR^n  \to \RR^n$ be monotone and $\ell$-Lipschitz continuous. Let $ (a^k )_{k\in\NN}, (b^k )_{k\in\NN}$ and $ (c^k )_{k\in\NN}$ be such that $\sum_{k=1}^\infty |a^k|<\infty$, $\sum_{k=1}^\infty |b^k|<\infty$ and $\sum_{k=1}^\infty |c^k|<\infty$. Let $x_{0} \in \RR^n,$ $\varepsilon \in(0,1 /(\ell \mu+1) )$ and let $(\gamma^k )_{k\in\NN}$ be a sequence in $[\varepsilon,(1-\varepsilon) /(\ell \mu)]$. Let $x^* \in \op{zer}(A+B)$ and let $(x^k)_{k\in\NN}$ be the sequence generated by \eqref{eq_fbf_vm}. Then, the following hold:
\begin{itemize}
\item[(i)] $\sum_{k=1}^\infty\| x^k-v^k\|^{2}<+\infty$,
\item[(ii)] $\lim_{k\to\infty}x^k  = x^*$.
\end{itemize}
\end{thm}
\begin{proof}
After using some results from \cite{combettes2013} to guarantee that the sequences are well defined and that the monotonicity properties of the operators $W_kA$ and $W_kB$ hold, a quasi-F\'ejer inequality can be proven, i.e.,
\begin{equation}
\begin{aligned}
\normsq{x^{k+1}-x^*}_{W_{k+1}^{-1}}&\leq(1+\eta^k)\normsq{x^{k}-x^*}_{W_{k}^{-1}}\\
&-\mu^{-1}(1-\gamma_k^2\ell^2\mu^2)\normsq{x^k-v^k}+\varepsilon^k,
\end{aligned}
\end{equation}
where $(\varepsilon^k)_{k\in\NN}$ is a summable sequence depending on the error sequences $(a^k)_{k\in\NN}$, $(b^k)_{k\in\NN}$ and $(c^k)_{k\in\NN}$.
Then, from Proposition \ref{prop_vm}, it holds that $(\norm{x^{k}-x^*}_{W_{k}^{-1}})_{k\in\NN}$ is bounded and from Lemma \ref{lemma_comb},it follows that $\sum_{k=1}^\infty\normsq{x^k-v^k}<\infty$. Convergence holds as a consequence of Theorem \ref{theo_comb_var} by letting $\bar x$ be a cluster point.
 \end{proof}

\section{Conclusion}

In this survey, we tried to answer the question posed by Polyak in 1987. We show that the importance of convergence theorems for mathematical system theory lays on their connection to Lyapunov analysis and F\'ejer monotonicity and in the variety of areas and applications where they are used.

Thanks to the notions of (quasi) F\'ejer monotonicity and Lyapunov decrease, results showing the convergence of sequences of (random) real numbers can be exploited in Nash equilibrium problems, machine learning and optimization. In these contexts, these results are fundamental to analyze and design convergent learning processes.

\appendix
\gdef\thesection{\Alph{section}} 
\makeatletter
\renewcommand\@seccntformat[1]{Appendix \csname the#1\endcsname.\hspace{0.5em}}
\makeatother
\section{Auxiliary notions}\label{appendix}
\subsection{Preliminaries}

In this section, we recall some notions from operator theory \cite{bau2011,ryu2016,facchinei2007}. Let us start with some notation.

$\iota_{\mc X}$ is the indicator function of the set $\mc X$, that is, $\iota_{\mc X}(x)=0$ if $x\in \mc X$ and $\iota_{\mc X}(x)=\infty$ otherwise. The set-valued mapping $\mathrm{N}_{\mc X} : \RR^{n} \to \RR^{n}$ denotes the normal cone operator of the set $\mc X$, i.e., $\mathrm{N}_{\mc X}(x)=\varnothing$ if $x \notin \mc X$, $\mathrm{N}_{\mc X}(x)=\left\{v \in \RR^{n} | \sup _{z \in \mc X} \langle v,z-x\rangle \leq 0\right\}$ otherwise. Given $\psi : \RR^{n} \to \RR$, $\op{dom}(\psi) :=\{x \in \RR^{n} | \psi(x)<\infty\}$ is the domain of $\psi$ and its subdifferential is the set-valued mapping $\partial \psi(x) :=\{v \in \RR^{n} | \psi(z) \geq \psi(x)+\langle v,z-x\rangle,\; \forall z \in \operatorname{dom}(\psi)\}$.

Let us now provide the definition of projection, proximal operator and resolvent which are used in many algorithms.

\begin{defn}\label{def_projproxres}
Let $x\in\RR^n$ and let $\mc C\subseteq\RR^n$ be nonempty closed and convex.
\begin{itemize}
\item The projection operator onto $\mc C$ is the operator defined as
$$\op{proj}_{\mc C}(x)=\underset{z \in\mc  C}{\op{argmin}}\normsq{z-x}.$$
The point $\op{proj}_{\mc C}(x)$ is the closest point to x in $\mc C$. It always exists and it is unique.
\item Given a proper lower semicontinuous convex function $f:\RR^n\to\bar\RR$, the proximity operator of f is the operator defined as 
$$
\op{prox}_{f}(x) = \underset{y \in \RR^n}{\op{argmin}}\left(f(y)+\frac{1}{2}\|x-y\|^{2}\right)
$$
\item Given $A : \RR^n \rightrightarrows \RR^n$, the resolvent of $A$ is the operator defined as
$$J_{A}=(\operatorname{Id}+A)^{-1}$$
where $\op{Id}$ is the identity function. 
\end{itemize}
\end{defn}

The notions discussed until now are related by the following example.
\begin{exmp}[Example 23.4, \cite{bau2011}]
Let $f =\iota_{\mc C}$ to be the indicator function of the set $\mc C\subseteq\RR^n$, then
$$J_{\op{N}_{\mc C}}=\left(\mathrm{Id}+\op{N}_{\mc C}\right)^{-1}=\operatorname{prox}_{\iota_{\mc C}}=\op{proj}_{\mc C}$$
where $\op{N}_{\mc C}$ is the normal cone of $\mc C$.
 \end{exmp}


Given a closed and convex set $\mc C\subseteq\RR^n$, $F:\mc C\to\RR^n$ and $\alpha>0$, let $\mathrm{res}_{\alpha}(x)=\norm{x-\op{proj}_{\mc X}(x-\alpha F(x))}$ be the residual function.
\begin{rem}\label{remark_res}
Let $F:\mc C\to\RR^n$ and $\mc C\subseteq\RR^n$ be closed and convex. Then, \cite[Proposition 1.5.8]{facchinei2007}
$$\mathrm{res}(x^*)=0 \text{ if and only if } x^*\in\op{SOL}(\mc C,F).$$
Moreover, it holds that
$$x^*\in\op{SOL}(\mc C,F) \text{ if and only if }x^*=\op{proj}_{\mc C}(x^*-\alpha F(x^*)).$$
\end{rem}




\subsection{Operator Theory}\label{appendix_op}

The convergence properties of the algorithms proposed for VIs or monotone inclusions are strictly related to the properties of the operators, to its monotonicity in particular. For this reason, here we recall some definition that are useful for the applications. 

\begin{defn}\label{def_mono}
Let $\op{gra}(F)=\{(x,u):u\in F(x)\}$ be the graph of $F : \mc X \subseteq \RR^{n} \to \RR^{n}$. Then, F is said to be
\begin{itemize}
\item[(a)] pseudomonotone on $\mc X$ if 
$\langle F(y),x-y\rangle \geq 0 \Rightarrow \langle F(x),x-y\rangle \geq 0 \text{ for all } x, y \in \mc X;$
\item[(b)] monotone on $\mc X$ if
$\langle F(x)-F(y),x-y\rangle \geq 0, \text{ for all } x, y \in \mc X;$
\item[(c)] strictly monotone on $\mc X$ if
$\langle F(x)-F(y),x-y\rangle>0, \text{ for all } x, y \in \mc X \text{ and } x \neq y;$
\item[(e)] $\mu$-strongly monotone on $\mc X$ if there exists a constant $\mu>0$ such that
$\langle F(x)-F(y),x-y\rangle \geq \mu\|x-y\|^{2}, \text{ for all } x, y \in \mc X;$
\item[(f)] maximally monotone (or maximal monotone) if there exists no monotone operator 
$G : \mc X \to \RR^n$ such that $\operatorname{gra} G$ properly contains $\operatorname{gra} F$, i.e.,
$(x, u) \in \operatorname{gra} F$  $\Leftrightarrow$ for all $(y, v) \in \operatorname{gra} F$ it holds $ \langle x-y | u-v\rangle \geq 0$;
\item[(g)] uniformly monotone at $y$ if there exist an increasing function $\phi:\RR_{\geq0}\to\RR_{\geq0}$ vanishing only at 0 such that
$\langle F(x)-F(y),x-y\rangle \geq\phi(\norm{x-y}).$ 
\end{itemize}
\end{defn}
The weakest assumption is pseudomonotonicity while strong monotonicity implies all the other notions. Strictly monotone operators are widely used in variational inequalities problems since this is the weaker assumption that guarantees uniqueness of the solution \cite[Theorem 2.3.3]{facchinei2007}. It implies monotonicity that, in turn, implies pseudomonotonicity.

Many results are also related to the Lipschitz and cocoercivity constants of the operator.
\begin{defn}\label{def_lip}
A mapping $F : \mc X \subseteq \RR^{n} \to \RR^{n}$ is said to be
\begin{itemize}
\item[(a)] $\ell$-Lipschitz continuous with constant $\ell>0$ if for all $x, y\in \mc X$ it holds
$\|F(x)-F(y)\| \leq \ell\|x-y\|$;
if $\ell<1$, F is called a contraction;
\item[(b)] nonexpansive if it is $1$-Lipschitz continuous, i.e., 
$\|F(x)-F(y)\| \leq \|x-y\|$ for all $x,y\in\mc X$;
\item[(c)] firmly nonexpansive if
$\|F(x)-F(y)\|^{2}+\|(\mathrm{Id}-F) (x)-(\mathrm{Id}-F) (y)\|^{2} \leq\|x-y\|^{2},\text{ for all } x, y \in \mc X$
\item[(d)] $\beta$-cocoercive on $\mc X$ if there exists a constant $\beta > 0$ such that
$\langle F(x)-F(y),x-y\rangle \geq \beta\normsq{F(x)-F(y)},$ for all $x, y \in \mc X. $
\end{itemize}
\end{defn}
It follows (using Cauchy-Schwartz inequality) that if a map is $\beta$-cocoercive, it is also $1/\beta$-Lipschitz continuous. Moreover, cocoercivity implies monotonicity.

Sometimes, in the stochastic case, we mention that the map is Carath\'eodory.
\begin{defn}
A mapping $F:\RR^d\times\Xi\to\RR^n$ is a Carath\'eodory map if $x \mapsto F(x, \xi)$ is continuous for almost every $\xi \in \Xi$, and $\xi \mapsto F(x, \xi)$ is measurable for all $x \in \RR^{d}$ where $\xi$ is a random variable with values in $\Xi,$ defined on a probability space $(\Omega, \mathcal{F}, \mathbb{P}).$ 
 \end{defn}

\subsection{Auxiliary results}\label{appendix_stoc}

Let us recall some results on martingales as this property for $L_p$ norms, known as Burkholder-Davis-Gundy inequality \cite{stroock2010,kushner2003}.
\begin{lem}[Burkholder-Davis-Gundy inequality]\label{BDG_Lemma}
Let $(\mc F_k)_{k\in\NN}$ be a filtration and $\{u^k\}_{k\geq0}$ a vector-valued martingale relative to this filtration. Then, for all $p\in [1,\infty)$, there exists a universal constant $c_p > 0$ such that for every $k\geq1$
$$\EE\left[\left(\sup _{0 \leq i \leq N}\left\|u_{i}\right\|\right)^{p}\right]^{\frac{1}{p}} \leq c_{p} \EE\left[\left(\sum_{i=1}^{N}\left\|u_{i}-u_{i-1}\right\|^{2}\right)^{\frac{p}{2}}\right]^{\frac{1}{p}}.$$
\end{lem}
We also recall the Minkowski inequality: for given functions $f,g\in L^p(\Xi,\mc F, \PP)$, $\mc G\subseteq \mc F$ and $p\in[0;\infty]$
$$\EE\left[\norm{f+g}^{p} | \mc{G}\right]^{\frac{1}{p}} \leq \EE\left[\norm{f}^{p} | \mc{G}\right]^{\frac{1}{p}}+\EE\left[\norm{g}^{p} | \mc{G}\right]^{\frac{1}{p}}.$$
When combined with the Burkholder-Davis-Gundy inequality, it leads to the fact that for all $p\geq2$, there exists a constant $c_p >0 $ such that, for every $k\geq1$,
$$\EE\left[\left(\sup _{0 \leq i \leq N}\norm{u_{i}}\right)^{p}\right]^{\frac{1}{p}} \leq c_p\sqrt{\sum_{k=1}^{N} \EE\left(\norm{u_{i}-u_{i-1}}^{p}\right)^{\frac{2}{p}}}.$$

The following result is presented for uniformly bounded variance but holds also for more general assumptions. For similar results, one can refer to \cite{iusem2017, bot2020,franci2020fb}.
\begin{lem}\label{lemma_variance}
Let $c>0$. Let $\sigma$ be as in Equation \eqref{variance} and $\mc N_k$ as in Standing Assumption \ref{ass_batch}. Then, it holds a.s. that
$$\EE\left[\norm{\epsilon^k}^2|\mc F_k\right]\leq\frac{c\sigma^2}{\mc N_k}, \text{ for all } k\in\NN.$$
\end{lem}
\begin{proof}
We first prove that 
$$\EE\left[\norm{\epsilon^k}^2|\mc F_k\right]^{\frac{1}{2}}\leq\frac{c_2\sigma}{\sqrt{\mc N_k}}$$
then the claim follows immediately. To this aim, let us first notice that $F^\textup{VR}(x,\xi)=\frac{1}{\mc N}\sum_{k=1}^{\mc N}F^\textup{SA}(x,\xi^k)$. Then, let us define the process $\{M_S^S(x)\}_{i=0}^S$ as $M_0(x)=0$ and for $1\leq i\leq S$
$$M_i^S(x)=\frac{1}{S}\sum_{j=1}^iF^\textup{SA}(x,\xi_j)-\FF(x).$$
Let $\mc F_i=\sigma(\xi_1,\dots,\xi_i)$. Then $\{M_i^S(x),\mc F_i\}_{i=1}^S$ is a martingale starting at $0$.
Let
$$\begin{aligned}
\Delta M_{i-1}^S(x)&=M_i^S(x)-M_{i-1}^S(x)\\
&=\frac{1}{S}(F^\textup{SA}(\bs x,\xi_i)-\FF(\bs x)).
\end{aligned}$$
Then, by Equation (\ref{variance}), we have
\begin{equation*}
\EEx{\normsq{\Delta M_{i-1}^S}}^{\frac{1}{2}}=\frac{1}{S}\EEx{\normsq{F^\textup{SA}(x,\xi_i)-\FF(x)}}^{\frac{1}{2}}\leq \frac{\sigma}{S}.
\end{equation*}
By applying Lemma \ref{BDG_Lemma}, we have
\begin{equation*}
\begin{aligned}
\EEx{\normsq{M_S^S(x)}}^{\frac{1}{2}}&\leq c_2 \sqrt{\sum_{i=1}^{N} \EEx{\normsq{\frac{F^\textup{SA}(\bs x, \xi_i)-\FF(x)}{S}}}}\\
&\leq c_2 \sqrt{\frac{1}{S^2} \sum_{i=1}^{N} \EEx{\normsq{F^\textup{SA}(\bs x, \xi_i)-\FF(x)}}}\\
&\leq\frac{c_2\sigma}{\sqrt{S}}.
\end{aligned}
\end{equation*}
We note that $M_{S}^{S}(x^k)=\epsilon^k$, hence by taking the square we conclude that 
\begin{equation*}
\EEk{\normsq{\epsilon^k}}\leq\frac{c\sigma^2}{\mc N_k}.  
\end{equation*}
\end{proof}

\section*{References}

\bibliography{Biblio}

\end{document}